\documentclass[a4paper,12pt,reqno]{amsart}
\usepackage{fullpage}
\usepackage{mymacros}
\usepackage{xcolor}
\usepackage{tabularx}
\usepackage{ltablex}

\usepackage{tikz}
\usetikzlibrary{cd}
\usepackage{braids}

\usepackage[backref,hypertexnames=false,linktoc=all]{hyperref}

\usepackage{todonotes}

\newcommand{\derived}{\bfD}
\newcommand{\dual}{\vee}
\newcommand{\Auteqtriv}{\Auteq^{ K _{ 0 }-\mathsf{triv}}}
\newcommand{\btriv}{B ^{ K _{ 0 }-\mathsf{triv}}}

\newcommand{\std}{\mathsf{std}}
\newcommand{\cEstd}{\cE^{\std}}
\newcommand{\cEbar}{\underline{\cE}}
\newcommand{\cFbar}{\underline{\cF}}
\newcommand{\ec}{\mathsf{EC}}
\newcommand{\ecvb}{\mathsf{ECVB}}
\newcommand{\fec}{\mathsf{FEC}}
\newcommand{\fsec}{\mathsf{FSEC}}
\newcommand{\fecvb}{\mathsf{FECVB}}
\newcommand{\numfec}{\mathsf{numFEC}}
\newcommand{\numec}{\mathsf{numEC}}
\newcommand{\quadric}{ \ensuremath{ \bP ^{ 1 } \times \bP ^{ 1 } } }
\newcommand{\ah}{\Auteq(\derived(\hirzebruchtwo))}
\newcommand{\gen}{\operatorname{\mathsf{gen}}}
\newcommand{\genvb}{\operatorname{\mathsf{gen} \vert_{ \ecvb_{ 4 } (\cX_{0})}}}
\newcommand{\kgr}[1]{\operatorname{K _{ 0 }} \left( #1 \right)}
\newcommand{\Map}{\operatorname{\mathsf{Map}}}
\newcommand{\hirzebruchtwo}{\Sigma _{ 2 }}
\newcommand{\swap}{\operatorname{\mathsf{swap}}}
\newcommand{\reduced}{\operatorname{\mathsf{red}}}
\newcommand{\cohomology}{\cH ^{ \bullet }}

\newcommand{\ext}{\operatorname{ext}}
\newcommand{\length}{\operatorname{length}}

\title{Exceptional collections on \( \Sigma _{ 2 }\)}

\author[A. Ishii]{Akira Ishii}
\address{Graduate School of Mathematics, Nagoya University, Furocho, Chikusa-ku, Nagoya, 464-8602, Japan}
\email{akira141@math.nagoya-u.ac.jp}

\author[S. Okawa]{Shinnosuke Okawa}
\address{Department of Mathematics,
Graduate School of Science, Osaka University
\newline 1-1, Machikaneyamacho, Toyonaka, Osaka 560-0043, Japan}
\email{okawa@math.sci.osaka-u.ac.jp}

\author[H. Uehara]{Hokuto Uehara}
\address{Department of Mathematical Sciences, Graduate School of Science,
Tokyo Metropolitan University, 1-1 Minamiohsawa, Hachioji-shi, Tokyo, 192-0397, Japan}
\email{hokuto@tmu.ac.jp}

\begin{document}

\begin{abstract}
    Structure theorems for exceptional objects and exceptional collections of the bounded derived category of coherent sheaves on del Pezzo surfaces are established by Kuleshov and Orlov in~\cite{MR1286839}. In this paper we propose conjectures which generalize these results to weak del Pezzo surfaces.
    Unlike del Pezzo surfaces, an exceptional object on a weak del Pezzo surface is not necessarily a shift of a sheaf and is not determined by its class in the Grothendieck group.
    Our conjectures explain how these complications are taken care of by spherical twists, the categorification of \((-2)\)-reflections acting on the derived category.
    
    This paper is devoted to solving the conjectures for the prototypical weak del Pezzo surface \(\hirzebruchtwo\), the Hirzebruch surface of degree \( 2 \).
    Specifically, we prove the following results:
    Any exceptional object is sent to the shift of the uniquely determined exceptional vector bundle by a product of spherical twists which acts trivially on the Grothendieck group of the derived category.
    Any exceptional collection on \(\hirzebruchtwo\) is part of a full exceptional collection.
    We moreover prove that the braid group on 4 strands acts transitively on the set of exceptional collections of length \(4\) (up to shifts).
\end{abstract}

\maketitle

\tableofcontents

%
%
\section{Introduction}\label{sc:Introduction}

Semiorthogonal decomposition is among the most fundamental notions of triangulated categories. The finest ones, i.e., semiorthogonal decompositions whose components are equivalent to the bounded derived category of a point, are identified with (full) exceptional collections of the triangulated category. If there is an exceptional collection in a triangulated category, then virtually always there are infinitely many of them because of the group action explained below. Hence the classification of exceptional collections is not obvious at all.

Let \( X \) be a smooth projective variety and \(\derived ( X )\) be the bounded derived category of coherent sheaves on \( X \). Concerning the classification of the exceptional collections of \( \derived ( X )\), the prototypical results are given in \cite[Section 5]{Gorodentsev_1987} for \( X = \bP ^{ 2 } \). The similar results for \( X = \bP ^{ 1 } \times \bP ^{ 1 } \) are given in \cite{MR966985}. Based on these works, Kuleshov and Orlov established in \cite{MR1286839} the similar results for arbitrary del Pezzo surfaces (see also \cite[Corollary 4.3.2, Theorem 4.3.3, Theorem 4.6.1]{MR2108443}).
They are summarized as follows.

\begin{theorem}[\cite{MR1286839}]\label{th:Kuleshov and Orlov}
Let \( X \) be a del Pezzo surface over an algebraically closed field \( \bfk \).

\begin{enumerate}
    \item \label{it:Kul-O structure theorem for exceptional object}
    Any exceptional object
    \( \cE \in \derived ( X ) \)
    is either a vector bundle or a line bundle on a \( ( - 1 ) \)-curve, up to shifts.

    \item
    The isomorphism class of an exceptional object
    \( \cE \in \derived ( X ) \)
    is determined by its class
    \(
        [ \cE ] \in \kgr{ X }
    \),
    up to shifts by \( 2 \bZ \).\label{it:class determines exceptional object}

    \item Any exceptional collection of \(\derived ( X )\) can be extended to a full exceptional collection.\label{it:constructibility for del Pezzo}

    \item
    \pref{cj:Bondal Polishchuk} below is true for \(X\); i.e., the action \( G _{ r } \curvearrowright \ec _{ r } ( X )\) is transitive.
    \label{it:BP Conjecture is true for del Pezzo}
\end{enumerate}
\end{theorem}

The symbol \(\kgr{X}\) denotes the Grothendieck group of the triangulated category \( \derived ( X ) \).
We briefly explain \eqref{it:BP Conjecture is true for del Pezzo}.
Let \(\ec _{ r } ( X ) \) denote the set of isomorphism classes of exceptional collections of length \(r\) of \( \derived ( X ) \), where \( r = \rank \kgr{ X } \). An important fact relevant to the classification is that there is a standard action of the group
\(
    G _{ r } \coloneqq \bZ ^{ r } \rtimes \Br _{ r }
\)
on \( \ec _{ r }( X ) \) by shifts and mutations, where \(\Br _{ r }\) is the braid group on \(r\) strands (see \pref{sc:Deformation and mutation of exceptional collections} for details).

Inspired by the earlier works mentioned above, Bondal and Polishchuk gave the following conjecture.
\begin{conjecture}[{\(=\)\cite[Conjecture~2.2]{MR1230966}}]\label{cj:Bondal Polishchuk}
    Suppose that \( X \) is a smooth projective variety such that \(\derived ( X )\) admits a full exceptional collection of length \(r\). Then the action \( G _{ r } \curvearrowright \ec _{ r } ( \derived ( X ) )\) is transitive.
\end{conjecture}

In dimensions greater than \(2\), the classification of exceptional collections is widely open. Even for \( \bP ^{ 3 }\), only partial results seem to be known (\cite{1995alg.geom..7014P}) and \pref{cj:Bondal Polishchuk} is still open.

The aim of this paper is to investigate the generalization of \pref{th:Kuleshov and Orlov} to \emph{weak del Pezzo surfaces}; i.e., smooth projective surfaces \(X\) whose anti-canonical bundle \(\omega _{ X } ^{ - 1 } \) is nef and big.
As it turns out, the generalization is not straightforward at all.

A weak del Pezzo surface \(X\) which is not a del Pezzo surface admits at least one (and only finitely many) \((-2)\)-curve(s); i.e., a smooth curve \( C \subset X \) with \( C ^{ 2 } = - 2 \) and \( C \simeq \bP ^{ 1 }\). Line bundles on \(C\), as objects of \( \derived ( X )\), are \emph{\(2\)-spherical objects} and hence yield non-trivial autoequivalences of \( \derived ( X )\) called \emph{spherical twists} due to Seidel and Thomas \cite{Seidel-Thomas}. For a \(2\)-spherical object \(\alpha \in \derived ( X )\), the corresponding spherical twist
\( T _{ \alpha }\) acts as a reflection on \(\kgr{X}\) whereas it always satisfies \(T _{ \alpha } ^{ 2 } \not \simeq \id _{ \derived ( X )}\). For example, the exceptional object
\( T _{ \cO _{ C } } ^{ 2 } ( \cO _{ X } ) \) has the same class as \( \cO _{ X } \) in \(\kgr{ X }\) but is not isomorphic to \( \cO _{ X } \). In fact, by direct computation one can confirm the following.
\begin{align}
    \cH ^{ i } \left( T _{ \cO _{ C } } ^{ 2 } ( \cO _{ X } ) \right)
    \simeq
    \begin{cases}
        \cO _{ X } & i = 0\\
        \cO _{ C } & i = 1, 2\\
        0 & \text{otherwise}
    \end{cases}
\end{align}
Moreover, by applying \( T _{ \cO _{ C } } ^{ 2 } \) repeatedly, we obtain a collection of infinitely many exceptional objects of unbounded cohomological amplitudes which share the same class in \( \kgr{ X }\).
Hence \pref{th:Kuleshov and Orlov} \eqref{it:Kul-O structure theorem for exceptional object} \eqref{it:class determines exceptional object} are not true for \( X \) at all.

\pref{cj:main conjecture}, which is the main conjecture of this paper, generalizes \pref{th:Kuleshov and Orlov} to weak del Pezzo surfaces while taking into account all the complications mentioned in the previous paragraph. In a word, it asserts that the failure of \pref{th:Kuleshov and Orlov} \eqref{it:Kul-O structure theorem for exceptional object} \eqref{it:class determines exceptional object} on weak del Pezzo surfaces is remedied by spherical twists and that \pref{th:Kuleshov and Orlov} \eqref{it:constructibility for del Pezzo} \eqref{it:BP Conjecture is true for del Pezzo} should hold for weak del Pezzo surfaces too. The main theorem of this paper is that \pref{cj:main conjecture} is true for \(\hirzebruchtwo = \bP _{ \bP ^{ 1 } } ( \cO \oplus \cO ( 2 ) )\); namely, the Hirzebruch surface of degree \( 2 \), a paradigm of weak del Pezzo surfaces. In fact we formulated \pref{cj:main conjecture} by generalizing the results we obtained for \(\hirzebruchtwo\), \pref{th:Kuleshov and Orlov}, and \cite[Theorem 1.2]{MR3758518} simultaneously.

\begin{conjecture}\label{cj:main conjecture}
Let \( X \) be a weak del Pezzo surface over an algebraically closed field \( \bfk \).
\begin{enumerate}
\item \label{it:conjecturally twist equivalent to a sheaf}

For any exceptional object
\(
    \cE \in \derived ( X )
\),
there exists a sheaf \( \cF \) on \( X \) which is either an exceptional vector bundle or a line bundle on a \((-1)\)-curve,
a sequence of line bundles on \((-2)\)-curves
\(
    \cL _{ 1 }, \dots, \cL _{ n }
\),
and an integer
\(
    m \in \bZ
\)
such that
\begin{align}
    \cE \simeq
    ( T _{ \cL _{ n } } \circ \cdots \circ T _{ \cL _{ 1 } } ) ( \cF ) [ m ].
\end{align}

\item
\label{it:conjecture on objects in the same class}
For any pair of exceptional objects \( \cE, \cE ' \in \derived ( X )\) such that
\( [ \cE ] = [ \cE ' ] \in \kgr{ X } \), there are a product \( b \) of spherical twists and inverse spherical twists which acts trivially on \(\kgr{ X }\) and \( m \in 2 \bZ \) such that
\( \cE ' \simeq b ( \cE ) [ m ] \). Moreover, for each exceptional object \( \cE \in \derived ( X ) \), there is a unique exceptional vector bundle \( \cF \) such that either
\( [ \cE ] = [ \cF ] \) or \( [ \cE ] = - [ \cF ] \) holds.

\item
\label{it:conjecturally constructible}

Any exceptional collection on \( X \) can be extended to a full exceptional collection.

\item
\label{it:conjecturally transitive}

\pref{cj:Bondal Polishchuk} is true for \( X \).
\end{enumerate}
\end{conjecture}

As it is more or less visible,
\eqref{it:conjecturally twist equivalent to a sheaf},
\eqref{it:conjecture on objects in the same class},
\eqref{it:conjecturally constructible}, and
\eqref{it:conjecturally transitive} of \pref{cj:main conjecture} generalizes
\eqref{it:Kul-O structure theorem for exceptional object},
\eqref{it:class determines exceptional object},
\eqref{it:constructibility for del Pezzo}, and
\eqref{it:BP Conjecture is true for del Pezzo} of \pref{th:Kuleshov and Orlov}, respectively.

To avoid repetition, let us simply state the main result of this paper as follows.
\begin{theorem}[MAIN THEOREM]\label{th:main}
    \pref{cj:main conjecture} is true for \( X = \hirzebruchtwo \).    
\end{theorem}
More specifically, for \( X = \hirzebruchtwo \)
\begin{itemize}
    \item 
    \pref{cj:main conjecture}
    \eqref{it:conjecturally twist equivalent to a sheaf} is solved affirmatively in \pref{sc:Twisting exceptional objects down to exceptional vector bundles} as \pref{th:exceptional objects are equivalent to vector bundles}.
    
    \item
    \pref{cj:main conjecture}
    \eqref{it:conjecture on objects in the same class} is solved affirmatively in \pref{sc:Exceptional objects sharing the same class} as \pref{cr:exceptional objects in the same numerical class}.
    
    \item
    \pref{cj:main conjecture}
    \eqref{it:conjecturally constructible} is solved affirmatively in \pref{sc:Constructibility of exceptional collections} as \pref{cr:constructibility}.
    
    \item
    \pref{cj:main conjecture}
    \eqref{it:conjecturally transitive} is solved affirmatively in \pref{sc:Braid group acts transitively on the set of full exceptional collections} as \pref{th:transitivity}.
\end{itemize}

\begin{remark}
    We give some comments on the preceding works which are related to \pref{cj:main conjecture} and \pref{th:main}, and one important consequence of \pref{cj:main conjecture} \pref{it:conjecturally transitive} on the fullness of exceptional collections of maximal length.
\begin{itemize}
    \item
    This paper is a continuation of \cite{MR3431636} by the 2nd and the 3rd authors and almost completely supersedes it.
    \pref{cj:main conjecture} \eqref{it:conjecturally twist equivalent to a sheaf} for \(\hirzebruchtwo\) is stated as \cite[Conjecture 1.3]{MR3431636}, and is solved for exceptional \emph{sheaves} in \cite[Theorem~1.4]{MR3431636}.

    \item
    \pref{cj:main conjecture} \eqref{it:conjecturally twist equivalent to a sheaf} for torsion exceptional sheaves is partially solved in \cite[Theorem 1.2]{MR3758518}. A weaker version of \pref{cj:main conjecture} \eqref{it:conjecturally twist equivalent to a sheaf} is stated as \cite[Conjecture 1.1]{MR3758518}.
    
    \item
    Exceptional objects and exceptional collections of vector bundles on weak del Pezzo surfaces is systematically studied in \cite{MR1604186}. In fact we use some results of this work in this paper.
    
    \item It follows from the definition of mutations that the triangulated subcategory generated by an exceptional collection is invariant under mutation. In particular, the fullness of an exceptional collection is preserved by mutations. On the other hand, any \( X \) as in \pref{cj:main conjecture} admits a full exceptional collection. Hence \pref{cj:main conjecture} \pref{it:conjecturally transitive} would imply that any exceptional collection of length equal to
\(
    \rank \kgr{ X } = \rank \Pic ( X ) + 2
\)
of \( \derived ( X ) \) is full, though it would also follow from \pref{cj:main conjecture} \pref{it:conjecturally constructible}.
\end{itemize}
\end{remark}

\subsection{Summary of each section and structure of the paper}

\pref{sc:Preliminaries} is a preliminary section. We recall the rudiments of mutations and the group \( B \) of autoequivalences generated by spherical twists from \cite{Ishii-Uehara_ADC}. Among others, we prove in \pref{cr:btriv is generated by squares} that \( \btriv \), the subgroup of \( B \) acting trivially on \( \kgr{ \hirzebruchtwo } \), is generated by squares of spherical twists by line bundles on \( C \).

From \pref{sc:Twisting exceptional objects down to exceptional vector bundles} till the end of the paper, we restrict ourselves to the proof of \pref{th:main} and in particular discuss the case of \( \hirzebruchtwo \) only.

\pref{sc:Twisting exceptional objects down to exceptional vector bundles} is the main component of the paper and devoted to the proof of \pref{th:exceptional objects are equivalent to vector bundles}. Namely, in this section, we prove that for each exceptional object \( \cE \in \derived ( \hirzebruchtwo ) \) there is a sequence of integers
\(
    a _{ 1 }, \dots, a _{ n }
\)
such that
\(
    \left( T _{ a _{ n } } \circ \cdots \circ T _{ a _{ 1 } } \right) ( \cE )
\)
is isomorphic to a shift of a vector bundle. In accordance with the steps of the proof, \pref{sc:Twisting exceptional objects down to exceptional vector bundles} is divided into six subsections.

In \pref{sc:First properties} we prove some properties of the cohomology sheaf \( \cohomology ( \cE ) = \bigoplus _{ i \in \bZ } \cH ^{ i } ( \cE ) \).
Among others we show that there is the index \( i _{ 0 } \in \bZ \) such that
\(
    \Supp \cH ^{ i _{ 0 } } ( \cE ) = \hirzebruchtwo
\)
and for any \( i \ne i _{ 0 } \) the cohomology sheaf
\(
    \cH ^{ i } ( \cE )
\), if not \(0\), is a pure sheaf whose \emph{reduced} support is \( C \).
Then in \pref{sc:Properties of the schematic support}, we prove that actually the \emph{schematic} support of
\(
    \cH ^{ i } ( \cE )
\)
for \( i \ne i _{ 0 } \) is \( C \). This is the most technical part of the paper, and actually this had been the main obstacle for the whole work. Fortunately one can use the result of this subsection as a black box to read the rest of the paper.

In \pref{sc:More on the structure} we prove that there is a decomposition \( \cH ^{ i _{ 0 } } ( \cE ) \simeq T \oplus E \), where \( E = E ( \cE ) \) is an exceptional sheaf and \( T \) is a torsion sheaf. We moreover show that if \( \tors E \ne 0 \), then \( T \) is a direct sum of copies of \( \cO _{ C } ( a ) \) for some \( a \in \bZ \) and that \( \tors \cohomology ( \cE ) \) is a direct sum of copies of \( \cO _{ C } ( a ) \) and \( \cO _{ C } ( a + 1 ) \). This integer \( a = a ( \cE ) \) plays a central role throughout the paper.

In \pref{sc:Derived dual of exceptional objects} we investigate the relationship between the cohomology sheaves of \( \cE \) and those of \( \cE ^{ \vee } \), the derived dual of \( \cE \). We in particular show in \pref{cr:E(cE vee) has non-trivial torsion} that if \( \cE \) is not isomorphic to a shift of a vector bundle and \( \tors E ( \cE ) = 0 \), then \( \tors E ( \cE ^{ \vee } ) \ne 0 \). In this paper we mainly discuss the case \( \tors E ( \cE ) \ne 0 \), and by this result we can settle the case where \( \tors E ( \cE ) = 0 \) by passing to \( \cE ^{ \vee } \).

In \pref{sc:Length of the torsion part} we introduce the notion of the length of the ``torsion part'' of an object at the generic point \( \gamma \) of \( C \). Formally speaking, for \( \cE \) it is defined as \( \ell ( \cE ) = \sum _{ i \in \bZ } \length _{ \cO _{ \hirzebruchtwo, \gamma } } \tors \cH ^{ i } ( \cE ) _{ \gamma } \). It follows that \( \cE \) is isomorphic to a shift of a vector bundle if and only if \( \ell ( \cE ) = 0 \), and hence it suffices to show that \( \ell ( T _{ c } ( \cE ) ) < \ell ( \cE ) \) for some \( c = c ( \cE ) \in \bZ \) if \( \ell ( \cE ) > 0 \). This is exactly what we achieve in \pref{sc:Proof of Theorem 3.1}. We show in \pref{th:decrease the length by appropriate twist} that
\(
    c = a ( \cE )
\)
works if \( \tors E ( \cE ) \ne 0 \) and otherwise
\(
    c = - a ( \cE ^{ \vee } ) - 3
\)
does.

Spherical objects on the minimal resolution of type \(A\) singularity is classified in \cite{Ishii-Uehara_ADC}. More specifically, the proof of \pref{th:exceptional objects are equivalent to vector bundles} is an adaptation of the proof of \cite[Proposition 5.1]{Ishii-Uehara_ADC}. It is, however, much more involved than that of \cite[Proposition 5.1]{Ishii-Uehara_ADC}. This is due to the fact that the support of an exceptional object on \(\hirzebruchtwo\) is never concentrated in the \((-2)\)-curve \(C\). On the contrary the reduced support of a spherical object on \(\hirzebruchtwo\) is concentrated in \(C\), and it immediately implies that the \emph{schematic} supports of the cohomology sheaves of the spherical object coincide with \(C\).

\pref{sc:Exceptional objects sharing the same class} is devoted to the proof of \pref{th:exceptional object is K0-trivial equivalent to vector bundle}. It is almost immediately obtained by combining \pref{th:exceptional objects are equivalent to vector bundles} with a small trick on squares of spherical twists (\pref{pr:TaTb as product of O (m C ) and squares of twists}) and the fact that an exceptional \emph{vector bundle} is uniquely determined by its class in \(\kgr{\hirzebruchtwo}\) (\pref{lm:exceptional vb is determined by K0}).

\pref{sc:Constructibility of exceptional collections} is devoted to the proof of \pref{cr:constructibility}. Take an exceptional collection \( \cEbar \). We first show in \pref{th:twisting exceptional collection to vector bundles} that there is a product of spherical twists \( b \) such that
\(
    b ( \cEbar )
\)
consists of vector bundles up to shifts. This is achieved in a one-by-one manner. The key is that if \( ( \cB, \cE )\) is an exceptional pair such that \( \cB \) is a vector bundle and \( \ell ( \cE ) > 0 \), then, surprisingly enough, \( T _{ c ( \cE ) } (\cB) \) remains to be a vector bundle (though it may not be isomorphic to \( \cB \)). Recall that \( c ( \cE ) \in \bZ \) depends only on \( \cE \) and that \( T _{ c( \cE ) } \) strictly decreases the length of \( \cE \).

\pref{cr:constructibility} is known for exceptional collections of vector bundles by a slight generalization \pref{th:Constructibility for bundle collections} of a result by Kuleshov in \cite{MR1604186}. Applying it to
\(
    b ( \cEbar )
\),
we immediately obtain the proof of \pref{cr:constructibility} for \( \cEbar \).

\pref{sc:Braid group acts transitively on the set of full exceptional collections} is devoted to the proof of \pref{th:transitivity}. We use the deformation of \(\hirzebruchtwo\) to \(\quadric\). The difficulty is that there are infinitely many exceptional objects on \( \hirzebruchtwo \) which deform to the same exceptional object on \( \quadric \) (non-uniqueness of the specialization, which is translated into the \emph{non-separatedness} of the moduli space of semiorthogonal decompositions introduced in \cite{2020arXiv200203303B}). Actually, two exceptional objects have the same deformation to \( \quadric \) if and only if they have the same class in \( \kgr{ \hirzebruchtwo }\) (up to shifts by \( 2 \bZ \)).

In \pref{st:reduction to numerically standard collection} of the proof, we use the fact that the corresponding result is already known by \pref{th:Kuleshov and Orlov} \eqref{it:BP Conjecture is true for del Pezzo} for the del Pezzo surface \(\quadric\). Since deformation of exceptional collections commutes with mutations, the result for \( \quadric \) immediately implies that for any exceptional collection of length \(4\) on \(\hirzebruchtwo\) there is a sequence of mutations which brings it to a collection which is numerically equivalent to the standard collection \(\cEstd\) (i.e., having the same classes as \( \cEstd \) in \( \kgr{ \hirzebruchtwo } \).  See \eqref{eq:standard collection} for the definition of \(\cEstd\)).

It remains to show that an exceptional collection \( \cEbar \) on \( \hirzebruchtwo \) which is numerically equivalent to \( \cEstd \) can be sent to \( \cEstd \) by mutations. In \pref{st:from numerically standard to standard}, as an intermediate step, we find \( b \in B \) such that \( \cEbar = b ( \cEstd ) \). We construct such \( b \) again in one-by-one manner.

In \pref{st:replace b with mutations} we prove that \( b \) can be replaced by a sequence of mutations. Thanks to the fact that mutations commute with autoequivalences, it is enough to show the assertion only for \( b \in \{ T _{ 0 }, T _{ - 1 } \} \). Recall that \( T _{ 0 }, T _{ - 1 } \) generate \( B \). At this point the problem is concrete enough to be settled by hand.

\subsection{Some words on future directions}

Though \pref{cj:main conjecture} is stated for arbitrary weak del Pezzo surfaces, in this paper we restrict ourselves to the study of the case of \( \hirzebruchtwo \). This is partly because the proof is rather involved already in this case. We nevertheless think that the case of \( \hirzebruchtwo \) should serve as a paradigm for the further investigations of \pref{cj:main conjecture}.

The proof of \pref{th:main} is based on the classification of rigid sheaves on the \( ( - 2 ) \)-curve; i.e., the fundamental cycle of the minimal resolution of the \( A _{ 1 } \)-singularity. So far such classification is achieved only for the minimal resolution of type \(A\) singularities by \cite{Ishii-Uehara_ADC}. If one wants to push the strategy of this paper, it seems inevitable to establish the similar classification for the minimal resolution of type \( D \) and type \(E\) singularities. See \cite{MR3983387} for results in this direction.

A weak del Pezzo surface over an algebraically closed field \(\bfk\) is isomorphic to either
\(
    \quadric, \hirzebruchtwo
\),
or a blowup of \( \bP ^{ 2 } \) in at most eight points in almost general positions (see, say, \cite[Theorem 8.1.15, Corollary 8.1.24]{MR2964027}). Hence our strategy based on the deformation to del Pezzo surfaces, in principle, is applicable to all weak del Pezzo surfaces.

Structure theorems for exceptional collections on \( \bP ^{ 2 }\) play an important role in showing the contractibility of (the main component of) the space of Bridgeland stability conditions \( \Stab ( \bP ^{ 2 } )\) in \cite{MR3703470}. Our results should be similarly useful for studying \( \Stab ( \hirzebruchtwo ) \). More specifically, they should be very closely related to the relationship between \( \Stab ( \hirzebruchtwo )\) and \( \Stab ( \quadric )\); see \pref{rm:stability conditions} for details.

\subsection{Notation and convention}

We work over an algebraically closed field \( \bfk \), unless otherwise stated. To ease notation, we will write
\(
    \ext ^{ i } = \dim _{ \bfk } \Ext ^{ i },
    h ^{ i } = \dim _{ \bfk } H ^{ i },
    e ^{ p, q } _{ 2 } = \dim _{ \bfk } E _{ 2 } ^{ p , q },
\)
and so on. Below is a list of frequently used symbols.

\begin{tabularx}{\linewidth}{rX}
    \( \Sigma _{ d } \) & the Hirzebruch surface of degree \(d\) \eqref{eq:Hirzebruch}\\
    \( C \) & the \((-2)\)-curve of \( \hirzebruchtwo\)\\
    \(f\) & the divisor class of a fiber of the morphism \( \hirzebruchtwo \to \bP ^{ 1 }\) \eqref{eq:intersection form}\\
    \( \ast ^{ \vee }\) & the derived dual of \( \ast \in \derived ( \hirzebruchtwo )\) \eqref{eq:derived dual}\\
    \( T _{ a } \ ( \text{resp. } T ' _{ a }) \in \Auteq ( \hirzebruchtwo ) \) & the (inverse) spherical twist by \( \cO _{ C } ( a ) \) \eqref{eq:Ta and T'a}\\
    \(B < \Auteq ( \hirzebruchtwo )\) & the group of autoequivalences generated by spherical twists (\pref{df:subgroup of spherical twists})\\
    \( \ec _{ N }, \ecvb _{ N }, \fec, \fecvb \) & various sets of exceptional collections (\pref{df:sets of exceptional collections})\\
    \( \Br _{ N }\) (resp. \( G _{ N } \)) & the braid group on \( N \) strands \eqref{eq:braid group} (resp. the extension of \( \Br _{ N } \) by \( \bZ ^{ N } \) \eqref{equation:GN})\\
    \(\gen\) & the generalization map for exceptional collections from the central fiber to the generic fiber \eqref{eq:generalization map}\\
    \(\numec _{ N },\numfec\) & various sets of numerical exceptional collections (\pref{df:numerical exceptional collection})\\
    \(\cEstd \in \fecvb ( \hirzebruchtwo )\) & the standard full exceptional collection of \(\derived ( \hirzebruchtwo )\) (\pref{df:the standard exceptional collection on Sigma2})\\
    \( \cEstd _{ \xi } \in \fecvb ( \cX _{ \gen } )\) & the standard full exceptional collection of \( \derived ( \cX _{ \gen } )\) \eqref{eq:standard collection generic}\\
    \( \cohomology ( \cE ) \ (\text{resp. } \cH ^{ i } ( \cE ))\) & the total (resp. \(i\)-th) cohomology of \( \cE \in \derived ( X ) \) with respect to the standard t-structure \eqref{eq:cohomology object}\\
    \( i _{ 0 } = i _{ 0 } ( \cE ) \) & the unique index such that \( \Supp \cH ^{ i _{ 0 } } ( \cE ) = \Sigma _{ 2 } \) (\pref{df:i0})\\
    \( ( R, \frakm )\) & the complete local ring of the \( A _{ 1 } \)-singularity (\pref{nt:A1 singularity})\\
    \( ( 0 : I ) _{ M } \subseteq M \) & the maximal submodule of \( M \) annihilated by the ideal \( I \subseteq R \) \eqref{eq:annihilator}\\
    \( \cI _{ C } \subset \cO _{ \hirzebruchtwo } \) & the ideal sheaf of \( C \subset \hirzebruchtwo \) \eqref{equation:cIC}\\
    \( D ( \ast ) \) & the dual \( \bfk \)-vector space of \( \ast \) \eqref{equation:k-dual}\\
    \( \cH ^{ i _{ 0 } } ( \cE ) \simeq E ( \cE ) \oplus T ( \cE ) \) & the canonical decomposition into an exceptional sheaf and a torsion sheaf (\pref{lm:exceptional sheaf is a direct summand})\\
    \( \cT ( \cE ), \ \cF ( \cE ) \) & the torsion (resp. the torsion free) part of \( \cH ^{ i _{ 0 } } ( \cE ) \) (\pref{df:E and cF})\\
    \( a, s, t \in \bZ \) & integers specified by the irreducible decomposition of \( \cT \) \eqref{eq:decomposition of cT} (see also \pref{lm:exceptional sheaf is a direct summand})\\
    \( \ell ( \ast )\) & ``length of the torsion part of \(\ast\)'' at the generic point \( \gamma \) of \(C\) (\pref{df:length})\\
    \( b, r, s \in \bZ \) & integers specified by the irreducible decomposition of an exceptional vector bundle restricted to \(C\) (\pref{lm:exceptional vector bundle restricted to C})\\
\end{tabularx}

%
%
\subsection*{Acknowledgements}
During the preparation of this paper, A.I. was partially supported by
JSPS Grants-in-Aid for Scientific Research
(19K03444).
S.O. was partially supported by
JSPS Grants-in-Aid for Scientific Research
(16H05994,
16H02141,
16H06337,
18H01120,
20H01797,
20H01794).
H.U. was partially supported by
JSPS Grants-in-Aid for Scientific Research
(18K03249).

%
%

\section{Preliminaries}
\label{sc:Preliminaries}
%
%
\subsection{The Hirzebruch surfaces}

The Hirzebruch surface of degree
\(
    d \in \bZ _{ \ge 0}
\)
is the ruled surface
\begin{align}\label{eq:Hirzebruch}
    p \colon \Sigma _{ d } \coloneqq
    \bP _{ \bP ^{ 1 } }
    \left( \cO _{ \bP ^{ 1 } } \oplus \cO _{ \bP ^{ 1 } } ( d ) \right)
    \to
    \bP ^{ 1 }.
\end{align}

The history of Hirzebruch surfaces goes back, at least, to the first paper by Hirzebruch \cite{MR45384}. An explicit isotrivial degeneration of \(\Sigma _{ d }\) to \( \Sigma _{ d ' } \) for \( d > d ' \) and \( d - d ' \in 2 \bZ \) is constructed in \cite[p.86~Example]{MR153033}.
As a special case, there exists a smooth projective morphism (defined over \(\Spec \bZ\))
\begin{align}\label{eq:degeneration of Sigma0 to Sigma2}
    \cX \to \bA ^{ 1 } _{ t }
\end{align}
such that the fiber over
\(
    t = 0
\)
is isomorphic to
\(
    \hirzebruchtwo
\)
and the restriction of the family over the open subscheme
\(
    \bG _{ m } \hookrightarrow \bA ^{ 1 }
\)
is isomorphic to the trivial family
\(
    \left(\quadric\right) \times \bG _{ m } \stackrel{\pr _{ 2 }}{\to} \bG _{ m }
\).

In this paper we investigate the bounded derived category of coherent sheaves on
\(
    \hirzebruchtwo
\),
which is the most basic example of weak del Pezzo surfaces. We let
\(
    C \subset \hirzebruchtwo
\)
denote the unique negative curve, and
\(
    f
\)
the (linear equivalence class of) the fiber of \( p \).
Recall that
\begin{align}
    \Pic \hirzebruchtwo = \bZ C \oplus \bZ f,
\end{align}
where
\begin{align}\label{eq:intersection form}
    C ^{ 2 } = - 2,
    f ^{ 2 } = 0,
    C . f = 1.
\end{align}
The anti-canonical bundle is given by
\begin{align}
    - K _{ \hirzebruchtwo } = 2 C + 4 f.
\end{align}

%
%
\subsection{Derived category, spherical twist, and the autoequivalence group of \( \hirzebruchtwo\)}

\begin{definition}
For a quasi-compact scheme \( Y \), we let
\( \Perf Y \) denote the perfect derived category of \( Y \) with the standard structure of a triangulated category. When \( Y \) is equipped with a morphism to \( \Spec \bfk \), we think of
\( \Perf Y \)
as a triangulated \( \bfk \)-linear category. It comes with the natural symmetric monoidal structure given by the tensor product over
\(
    \cO _{ Y }
\),
but we do not take it into account unless otherwise stated.

When \( Y \) is a smooth and projective variety over a field \( \bfk \), we identify
\( \Perf Y \) with the bounded derived category \( \derived ( Y ) \) of coherent sheaves on
\( Y \).

The following equivalence of tensor triangulated categories
\begin{align}\label{eq:derived dual}
    {} ^{ \vee } \colon
    \left( \Perf Y \right) ^{ \op } \simto \Perf Y;\quad
    \cE \mapsto \cE ^{ \vee } \coloneqq \bR\cHom _{ Y } ( \cE, \cO _{ Y } )
\end{align}
will be called the derived dual.
\end{definition}

One can easily verify that there exits a canonical natural isomorphism
\begin{align}\label{eq:double dual}
    \id \stackrel{\sim}{\Rightarrow} {} ^{ \vee \vee }.
\end{align}

\begin{definition}
For smooth projective varieties
\(
    X, Y
\)
over \( \Spec \bfk \)
and an object
\(
    K \in \derived ( X \times _{ \bfk } Y )
\),
the integral transform by the kernel \( K \) will be denoted and defined as follows.
\begin{align}
    \Phi _{ K }
    \coloneqq
    \Phi _{ K } ^{ X \to Y }
    \colon
    \derived ( X ) \to \derived ( Y );
    \quad
    E \mapsto \bR p _{ Y \ast } \left( p _{ X } ^{ \ast } E
    \otimes _{ X \times _{ \bfk } Y } ^{ \bL } K \right)
\end{align}
\end{definition}

Let
\(
    X
\)
be a smooth projective variety over a field \( \bfk \).
Recall that an object
\(
    \alpha \in \derived ( X )
\)
is \emph{spherical} if
\(
    \alpha \otimes _{ \cO _{ X } } \omega _{ X }
    \simeq
    \alpha
\)
and
\(
    \RHom _{ X } ( \alpha, \alpha )
    \simeq
    \bfk \oplus \bfk [ - \dim X ]
\).
The \emph{spherical twist} by \( \alpha \)
is the endofunctor
\begin{align}
    T _{ \alpha }
    \coloneqq
    \Phi _{ K _{ \alpha } }
\end{align}
defined by the kernel
\(
    K _{ \alpha }
    \coloneqq
    \cone \left( \alpha ^{ \dual } \boxtimes \alpha
    \xrightarrow{ \ev } \cO _{ \Delta _{ X } } \right)
\).

Consider the exchange automorphism
\begin{align}
    \swap \colon X \times X \to X \times X; \, (x, y) \mapsto (y,x).
\end{align}
Recall that
\begin{align}
    \left( \swap^* K _{ \alpha } \right) ^{ \dual } \otimes p _{ 1 } ^{ \ast } \omega _{ X } [ \dim X ]
\end{align}
is called the right adjoint kernel of \(K _{ \alpha }\) and it enjoys the following adjoint property (see, say, \cite[Definition 5.7]{MR2244106}).
\begin{align}
    T _{ \alpha } = \Phi _{ K _{ \alpha } } \dashv
    \Phi _{ \left( \swap^* K _{ \alpha } \right) ^{ \dual } \otimes p _{ 1 } ^{ \ast } \omega _{ X } [ \dim X ] } \eqqcolon T ' _{ \alpha }
\end{align}
It follows that
\(
    T _{ \alpha }
\)
is an autoequivalence, so that
\(
    T ' _{ \alpha } \simeq T _{ \alpha } ^{ - 1 }
\).
Note that there exists the obvious isomorphism
\begin{align}
    \swap ^{ \ast } K _{ \alpha } \simeq K _{ \alpha ^{ \vee } },
\end{align}
so that
\begin{align}\label{eq:a description of T'}
    T ' _{ \alpha ^{ \vee } } \simeq \Phi _{ K _{ \alpha } ^{ \dual } \otimes p _{ 1 } ^{ \ast } \omega _{ X } [ \dim X ]}.
\end{align}

A typical example of a spherical object on \( \hirzebruchtwo \) is
\(
    \cO _{ C } ( a )
\)
for
\(
    a \in \bZ
\).
The corresponding (inverse) spherical twist will be denoted as follows, for short.
\begin{align}\label{eq:Ta and T'a}
    T _{ a } \coloneqq
    T _{ \cO _{ C } ( a ) },\quad
    T ' _{ a } \coloneqq
    T _{ \cO _{ C } ( a ) } ^{ - 1}
\end{align}

By definition, for each spherical object
\(
    \alpha
\)
and any object
\(
    E \in \derived ( X )
\),
there are standard triangles as follows.
To ease notation, we simply let \( \otimes _{ \cO _{ X } }\) denote the \emph{derived} tensor product.
The morphisms \( \varepsilon \) and \( \eta \) are the evaluation and the coevaluation maps respectively, both of which are obtained from the standard adjoint pair of functors
\(
    - \otimes _{ \cO _{ X } } \ast
    \dashv
    \RHom _{ X } ( \ast, - )
\).

\begin{align}
    \RHom _{ X } \left( \alpha, E \right) \otimes _{ \bfk } \alpha
    \xrightarrow{\varepsilon}
    E
    \to
    T _{ \alpha } ( E )
    \xrightarrow{ + 1 }\label{eq:triangle of spherical twist}\\
    T ' _{ \alpha } ( E )
    \to
    E
    \xrightarrow{ \eta }
    \RHom _{ X } \left( E, \alpha \right) ^{ \vee } \otimes _{ \bfk } \alpha
    \xrightarrow{ + 1 }\label{eq:triangle of inverse spherical twist}
\end{align}

\begin{lemma}[{\cite[Lemma~4.14]{Ishii-Uehara_ADC}}]\label{lm:conjugation of spherical twist}
For each
\(
    \Phi \in \Auteq ( X )
\), there is an isomorphism of autoequivalences as follows.
\begin{align}
    \Phi \circ T _{ \alpha } \circ \Phi ^{ - 1 }
    \simeq
    T _{ \Phi ( \alpha ) }
\end{align}
In particular, for any
\(
    a, m \in \bZ
\)
it holds that
\begin{align}\label{eq:exchanging Ta and O(C)}
    \left( \cO _{ \hirzebruchtwo } ( m C ) \otimes _{ \cO _{ \hirzebruchtwo } } - \right)
    \circ
    T _{ a }
    \simeq
    T _{ a - 2 m }
    \circ
    \left( \cO _{ \hirzebruchtwo } ( m C ) \otimes _{ \cO _{ \hirzebruchtwo } } - \right).
\end{align}
\end{lemma}

\begin{proof}
    Since \( X \) is a smooth projective variety over \( \bfk \), any autoequivalence \( \Psi \) of
    \(
        \derived ( X )
    \)
    is a Fourier-Mukai transform; i.e., it is isomorphic to the integral transform by an appropriate kernel by~\cite[Theorem~2.2]{Orlov_EDCKS}. It then follows that \( \Psi \) is isomorphic to \( \id _{ \derived ( X ) } \) if and only if \( \Psi ( E ) \simeq E \) holds for any \( E \in \derived ( X ) \). With this in mind, one can confirm the assertion by using~\eqref{eq:triangle of spherical twist}.
\end{proof}

The operation of taking duals is related to (inverse) spherical twists nicely.

\begin{lemma}\label{lm:spherical twist and dual}
\begin{enumerate}
\item
For any spherical object    
\(
    \alpha \in \derived ( X )
\),
there exists a natural isomorphism of functors
\begin{align}\label{eq:dual and inverse}
    \left( T _{ \alpha } - \right) ^{ \vee }
    \simeq
    T ' _{ \alpha ^{ \vee } } \left( - ^{ \vee } \right).
\end{align}

\item
When
\(
    X = \hirzebruchtwo
\),
for any \( a \in \bZ \), there exists a natural isomorphism of functors
\begin{align}\label{eq:dual and spherical twists}
    \left( T _{ a } - \right) ^{ \vee }
    \simeq
    T ' _{ - 2 - a } \left( - ^{ \vee } \right).
\end{align}
\end{enumerate}
\end{lemma}

\begin{proof}
There is a natural isomorphism
\begin{align}
    \left( T_\alpha - \right)^\vee
    =
    \left( \Phi_{K_\alpha} ( - ) \right)^\vee
    &
    =
    \cRHom _{ X }
        \left(
            \bR p_{2\ast}(p_1^{\ast} - \otimes^{\bL}_{X \times _k X} K_\alpha), \cO _{ X }
        \right)\\
    &
    \simeq
    \cRHom _{ X }
        \left(
            \bR p_{2\ast} \left( - \boxtimes \omega _{ X } \otimes^{\bL}_{X \times _k X} K_\alpha [ \dim X ] \right),
            \omega _{ X } [ \dim X ]            
        \right)\\
    &
    \stackrel{ \ast }{\simeq}
    \bR p_{2\ast} \cRHom _{ X \times _{ \bfk } X }
        \left(
            - \boxtimes \omega _{ X } \otimes^{\bL}_{X \times _k X} K_\alpha [ \dim X ],
            \omega _{ X } \boxtimes \omega _{ X } [ 2 \dim X ]
        \right)\\
    &
    \simeq
    \bR p _{ 2 \ast }
    \left(
        p _{ 1 } ^{ \ast } - ^{ \vee } \otimes _{ X \times _{ \bfk } X } \left( K _{ \alpha } ^{ \vee } \otimes p _{ 1 } ^{ \ast } \omega _{ X } [ \dim X ] \right)
    \right)\\
	& \simeq \Phi_{K_\alpha^\vee \otimes p_1 ^\ast \omega _{ \hirzebruchtwo }  [ \dim X ] } ( - ^{ \vee } ),
\end{align}
where the isomorphism \(\stackrel{\ast}{\simeq}\) follows from
\cite[Chapter VII, Corollary 4.3 b)]{MR0222093}.

Comparing this with
\eqref{eq:a description of T'}, we obtain \eqref{eq:dual and inverse}.

The second item follows from \eqref{eq:dual and inverse} and the isomorphism
\(
    \cO _{ C } ( a ) ^{ \vee } [ 1 ]
    \simeq
    \cO _{ C } ( - 2 - a )
\).
\end{proof}

\begin{lemma}[{\cite[Lemma~4.15~(i)~(2)]{Ishii-Uehara_ADC}}]
There is an isomorphism of autoequivalences
\begin{align}\label{eq:Ta Ta+1 = O(C)}
    T _{ a } \circ T _{ a + 1 }
    \simeq
    \cO _{ \hirzebruchtwo } ( C ) \otimes _{ \cO _{ \hirzebruchtwo } } -,
\end{align}
which implies
\begin{align}\label{eq:square of inverse twists as square of twists and O(-2C)}
    \begin{aligned}
    T _{ a } ' \simeq T _{ a + 1 } \circ \cO _{ \hirzebruchtwo } ( - C ) \otimes _{ \cO _{ \hirzebruchtwo } } -,\\
    { T _{ a } ' } ^{ 2 } \simeq T _{ a + 1 } \circ T _{ a + 3 } \circ
    \cO _{ \hirzebruchtwo } ( - 2 C ) \otimes _{ \cO _{ \hirzebruchtwo } } -.
    \end{aligned}
\end{align}
\end{lemma}

\begin{definition}
For a triangulated category \( \derived \),
we define the group
\(
    \Auteqtriv ( \derived )
\)
of \( K _{ 0 } \)-trivial autoequivalences by the following exact sequence.

\begin{align}
    1
    \to
    \Auteqtriv ( \derived )
    \to
    \Auteq  ( \derived )
    \to
    \Aut \left( \kgr{\derived} \right)
\end{align}
\end{definition}

The following well-known lemma, which follows from \eqref{eq:triangle of spherical twist} and \eqref{eq:triangle of inverse spherical twist}, asserts that spherical twists are categorification of root reflections. 
\begin{lemma}\label{lm:square of spherical twists act trivially on K0}
For any spherical object \( \alpha \in \derived (\hirzebruchtwo) \),
it holds that
\begin{align}
    T _{ \alpha} ^{ 2 },
    {T ' _{ \alpha}} ^{ 2 }
    \in
    \Auteqtriv ( \derived (\hirzebruchtwo) ).
\end{align}
\end{lemma}

\begin{proposition}\label{pr:TaTb as product of O (m C ) and squares of twists}
For any
\(
    a, b \in \bZ
\),
there exists a sequence
\(
    a _{ 1 }, \dots, a _{ \ell } \in \bZ
\)
and
\(
    m \in \bZ
\)
such that
\begin{align}
    T _{ a } T _{ b }
    =
    \cO _{ \hirzebruchtwo } ( m C ) \otimes _{ \cO _{ \hirzebruchtwo } } -
    \circ T _{ a _{ \ell } } ^{ \pm 2 } \cdots T _{ a _{ 1 } } ^{ \pm 2 }.
\end{align}
\end{proposition}

\begin{proof}
By \eqref{eq:exchanging Ta and O(C)}, we may and will reduce the proof to that of the following claim.

\begin{claim}
\( T _{ a } T _{ b } \) is contained in the subgroup of \( \Auteq ( \derived ( \hirzebruchtwo ) ) \)
generated by
\(
    \cO _{ \hirzebruchtwo } ( C ) \otimes _{ \cO _{ \hirzebruchtwo } } -
\)
and
\(
    T _{ a } ^{ 2 }
\)
for all
\(
    a \in \bZ
\).
\end{claim}
In the rest, we prove this claim. Consider first the case
\(
    a \le b
\).
We induct on
\(
    b - a \ge 0
\).
If it is \( 0 \), we have nothing to show.
In general, we have
\begin{align}
    T _{ a } T _{ b }
    \simeq
    \left( T _{ a } T _{ a + 1 } \right) \left( { T ' _{ a + 1 } } ^{ 2 }
\right)    \left( T _{ a + 1 } T _{ b } \right).    
\end{align}
By \eqref{eq:Ta Ta+1 = O(C)}, the first term of the right hand side is isomorphic to
\( \cO _{ \hirzebruchtwo } ( C ) \otimes _{ \cO _{ \hirzebruchtwo } } - \).
By the induction hypothesis, the 3rd term is also contained in the subgroup mentioned in the claim.
Hence we are done with this case.

Now suppose that
\(
    a > b
\).
Then
\begin{align}
    T _{ a } T _{ b }
    \simeq
    T _{ a } ^{ 2 } \left( T _{ b } T _{ a } \right) ' T _{ b } ^{ 2 },
\end{align}
and we know that the middle term is contained in the subgroup as shown in the previous paragraph.
\end{proof}

\begin{definition}\label{df:subgroup of spherical twists}
    The subgroup of
    \(\ah\)
    generated by spherical twists will be denoted by
    \(
        B
    \).
\end{definition}

\begin{corollary}\label{cr:btriv is generated by squares}
    Fix any \( a _{ 0 } \in \bZ \). Then 
    elements of \( B \) are exhausted by those of the following form, where \( a _{ 1 }, \dots, a _{ n }, m \in \bZ \).
    \begin{align}
        \left( \cO _{ \hirzebruchtwo } ( m C ) \otimes _{ \cO _{ \hirzebruchtwo } } - \right)
        \circ
        \left(
            T _{ a _{ n } } ^{ \pm 2 }
            \circ
            \cdots
            \circ
            T _{ a _{ 1 } } ^{ \pm 2 }            
        \right)\label{eq:even}\\
        \left( \cO _{ \hirzebruchtwo } ( m C ) \otimes _{ \cO _{ \hirzebruchtwo } } - \right)
        \circ
        T _{ a _{ 0 } }
        \circ
        \left(
            T _{ a _{ n } } ^{ \pm 2 }
            \circ
            \cdots
            \circ
            T _{ a _{ 1 } } ^{ \pm 2 }            
        \right)\label{eq:odd}
    \end{align}

    Moreover, the normal subgroup
    \begin{align}
        \btriv \coloneqq B \cap \Auteqtriv ( \derived ( X ) ) \triangleleft B
    \end{align} is generated by
    \(
        \{ T _{ a } ^{ 2 } \mid a \in \bZ \} \subset \btriv
    \).
\end{corollary}

\begin{proof}
    The first assertion immediately follows from \eqref{eq:square of inverse twists as square of twists and O(-2C)}, \eqref{eq:exchanging Ta and O(C)}, and \pref{pr:TaTb as product of O (m C ) and squares of twists}.

    For the second assertion, take \( b \in \btriv \). If \(b\) is as in \eqref{eq:even}, then obviously \( m = 0 \). Also one can verify that \(b\) is never like \eqref{eq:odd}, say, by using that it must preserve the classes
    \(
        [ \cO _{ \hirzebruchtwo } ( ( a _{ 0 } + 1 ) f ) ]
    \)
    and
    \(
        [ \cO _{ \hirzebruchtwo } ( ( a _{ 0 } + 2 ) f ) ]
    \).
\end{proof}

\pref{lm:conjugation of spherical twist} implies that
\(
    B
\)
is a normal subgroup of \(\ah\). As we explain next, sets of generators of \( B \) are well understood.

\begin{theorem}\label{th:generator of the group B}
For any
\(
    a \in \bZ
\),
the group
\(
    B
\)
is generated by the two spherical twists
\(
    T _{ a }, T _{ a + 1 }
\).
\end{theorem}
\begin{proof}
The case
\(
    a = - 2
\)
is explicitly mentioned in \cite[Lemma 4.6]{Ishii-Uehara_ADC}.
The general case inductively follows from this by the isomorphism of functors
\begin{align}
    T _{ a - 1 } T _{ a }
    \simeq
    \cO _{ \hirzebruchtwo } ( C ) \otimes _{ \cO _{ \hirzebruchtwo } } -
    \simeq
    T _{ a } T _{ a + 1 },
\end{align}
which is \eqref{eq:Ta Ta+1 = O(C)}.
\end{proof}

%
%
\subsection{Deformation and mutation of exceptional collections}
\label{sc:Deformation and mutation of exceptional collections}

Let
\begin{align}\label{eq:f cX -> B, general}
    f \colon \cX \to B
\end{align}
be a smooth projective morphism of Noetherian schemes with a closed point
\(
    0 \in B
\), and let
\begin{align}
    X _{ 0 } = \cX \times _{ f, B } \{ 0 \}
\end{align}
be the central fiber. Note that the properness and smoothness of \( f \) implies it is a perfect morphism, in the sense that the derived pushforward
\(
    f _{ \ast }
\)
respects perfect complexes (\cite[Proposition 2.1]{MR2346195}).

\begin{definition}
\begin{enumerate}        
\item
An object
\(
    \cE \in \Perf ( \cX )
\)
is
\( f \)-exceptional if
\(
    \bR f _{ \ast } \bR \cHom ( \cE, \cE )
\)
is a line bundle on \( B \).

\item
A collection of \( f \)-exceptional objects
\(
    \cE _{ 1 }, \dots, \cE _{ N } \in \Perf \cX
\)
is an
\(
    f
\)-exceptional collection if
\(
    \bR f _{ \ast } \bR \cHom ( \cE _{ j }, \cE _{ i } ) = 0
\)
for any
\(
    ( i, j )
\)
with
\(
    1 \le i < j \le N
\).

\item
An \( f \)-exceptional collection as above is said to be strong if moreover
\(
    \bR f _{ \ast } \bR \cHom ( \cE _{ j }, \cE _{ i } )
\)
is isomorphic to a locally free sheaf (regarded as a complex concentrated in degree \( 0 \)) for any
\(
    ( i, j )
\).

\item
An \( f \)-exceptional collection as above is said to be full if the minimal \( B \)-linear (i.e., closed under
\(
    - \otimes f ^{ \ast } \cF
\)
by any
\(
    \cF \in \Perf B
\))
triangulated subcategory which contains all of the objects in the collection is equivalent to
\(
    \Perf \cX
\).
\end{enumerate}
\end{definition}

\begin{definition}\label{df:sets of exceptional collections}
We say that two \(f\)-exceptional collections
\(
    \cE _{ 1 }, \dots, \cE _{ N }
\)
and
\(
    \cE ' _{ 1 }, \dots, \cE ' _{ N }
\)
are \emph{isomorphic} if
\(
    \cE _{ i } \simeq \cE ' _{ i }
\)
for each \( i = 1, \dots, N \).
We will let \( \ec _{ N } ( f ) \) (or \( \ec _{ N } ( \cX ) \), if \( f \) is obvious from the context) denote the set of isomorphism classes of \( f \)-exceptional collections of length \( N \). The set of isomorphism classes of \( f \)-exceptional collections of length
\(
    N
\)
consisting entirely of locally free sheaves will be denoted by
\(
    \ecvb _{ N } ( f )
\)
(or \( \ecvb ( \cX ) \)), which comes with the obvious injection
\(
    \ecvb _{ N } ( f ) \hookrightarrow \ec _{ N } ( f )
\).
Similarly, the set of equivalence classes of full \( f \)-exceptional collections will be denoted by either
\(
    \fec ( f )
\)
or
\(
    \fec ( \cX )
\),
and the set of isomorphism classes of full \( f \)-exceptional collections consisting entirely of locally free sheaves will be denoted by \( \fecvb ( f ) \) (or \( \fecvb ( \cX ) \)), which comes with the obvious injection
\(
    \fecvb ( f ) \hookrightarrow \fec ( f )
\).
\end{definition}

\begin{lemma}\label{lm:f-exceptional object and semiorthogonal decomposition}
\begin{enumerate}
\item
Let
\(
    \cE \in \Perf \cX
\)
be an \( f \)-exceptional object.
Then the functor
\begin{align}
    \Phi _{ \cE } \colon \Perf B \to \Perf \cX;\quad
    F \mapsto f ^{ \ast } F \otimes \cE
\end{align}
is fully faithful and admits a right adjoint as follows.
\begin{align}\label{eq:adjoint pair induced from an f-exceptional object}
    \Phi _{ \cE } \dashv \phi _{ \cE } ^{ R } \coloneqq f _{ \ast } \left( - \otimes \cE ^{ \vee } \right)
\end{align}

\item\label{it:f-exceptional collection induces f-linear semiorthogonal decomposition}
Let
\(
    \left( \cE _{ 1 }, \dots, \cE _{ N } \right) \in \ec _{ N } ( \cX )
\)
be an \( f \)-exceptional collection. Then the smallest \( B \)-linear triangulated subcategory of
\(
    \Perf \cX
\)
which contains
\(
    \cE _{ 1 }, \dots, \cE _{ N }
\)
admits a \( B \)-linear semiorthogonal decomposition
\(
    \langle
    \Phi _{ \cE _{ 1 } } ( \Perf B ),
    \dots
    \Phi _{ \cE _{ N } } ( \Perf B )
    \rangle
\).
\end{enumerate}
\end{lemma}
We say that a semiorthogonal decomposition
\(
    \Perf \cX = \langle \cA _{ 1 }, \dots, \cA _{ N } \rangle
\)
is \emph{\( B \)-linear} if
\(
    \cA _{ i } \otimes f ^{ \ast } b \subseteq \cA _{ i }
\)
holds for any \( i = 1, \dots, N \) and
\(
    b \in \Perf B
\)
 (see \cite[Section 2.3]{MR2801403}).

We will freely use the following very useful base change theorem from \cite[Corollary~2.1.4]{Bonsdorff_FTHB}.
See also \cite[Section~2.4]{MR2238172} and \cite[\href{https://stacks.math.columbia.edu/tag/08IB}{Tag 08IB}]{stacks-project} for treatise from different points of view.

\begin{lemma}\label{lm:base change}
Consider the following Cartesian square of finite dimensional noetherian schemes, where \( f, g \) are perfect.
\begin{equation}
    \begin{tikzcd}
    Y \arrow[swap]{d}{ g } \arrow{r}{ \psi } & \cX \arrow{d}{f}\\
    C \arrow[swap]{r}{\varphi} & B
    \end{tikzcd}
\end{equation}
Then the standard natural transformation of functors
\begin{align}\label{eq:base change isomorphism}
    \varphi ^{ \ast } \circ f _{ \ast } \Rightarrow g _{ \ast } \circ \psi ^{ \ast }
    \colon
    \Perf \cX \to \Perf C
\end{align}
is an isomorphism if either
\(
    f
\)
or
\(
    \varphi
\)
is flat.
\end{lemma}

\begin{corollary}\label{cr:base change of exceptional collections}
    Consider a morphism of schemes as in \eqref{eq:f cX -> B, general} and a (strong) \( f \)-exceptional collection
    \(
        \left( \cE _{ 1 }, \dots, \cE _{ N } \right) \in \ec _{ N } ( \cX )
    \). Take any morphism \( \varphi \colon C \to B \), where \( C \) is a finite dimensional noetherian scheme, and consider the Cartesian diagram as in \pref{lm:base change}. Then 
    \begin{align}
        \left( \psi ^{ \ast } \cE _{ 1 }, \dots, \psi ^{ \ast } \cE _{ N } \right)
    \end{align}
    is a (strong) \( g \)-exceptional collection on \( Y \).
\end{corollary}

\begin{proof}
    As \( f \) is flat, the assertion immediately follows from the following computation.
    All functors are derived.
    \begin{align}
        g _{ \ast } \cHom _{ Y } \left( \psi ^{ \ast } \cE _{ i }, \psi ^{ \ast } \cE _{ j } \right)
        \simeq
        g _{ \ast } \psi ^{ \ast } \cHom _{ \cX } \left( \cE _{ i }, \cE _{ j } \right)
        \stackrel{\text{\pref{lm:base change}}}{\simeq}
        \varphi ^{ \ast } f _{ \ast } \cHom _{ \cX } \left( \cE _{ i }, \cE _{ j } \right)
    \end{align}
\end{proof}

\begin{lemma}\label{lm:exceptional collection deforms}
Suppose that
\(
    B = \Spec R
\)
for a complete local Noetherian ring
\(
    ( R, \frakm, \bfk )
\).
Then the natural restriction maps
\begin{align}
    \ec _{ N } ( f )
    \to
    \ec _{ N } ( X _{ 0 } ),\\
    \ecvb _{ N } ( f )
    \to
    \ecvb _{ N } ( X _{ 0 } )
\end{align}
obtained in \pref{cr:base change of exceptional collections} are bijections for any
\( N \).
\end{lemma}

\begin{proof}
Let us first show that any exceptional object
\(
    \cE \in \Perf X _{ 0 }
\)
deforms to an object \( \cE _{ R } \in \Perf \cX \).
By definition of exceptional object, we know that
\(
    \Ext ^{ 1 } _{ X _{ 0 } } ( \cE, \cE )
    =
    \Ext ^{ 2 } _{ X _{ 0 } } ( \cE, \cE )
    =
    0
\).
By the deformation theory of objects (see, say, \cite[Corollary 3.4]{MR2578562}), for each
\(
    n \ge 1
\)
one finds the unique lift in
\(
    \Perf ( \cX \otimes _{ R } R / \frakm ^{ n + 1 } )
\)
of
\(
    E
\).
Then it algebrizes uniquely to an actual object
\(
    \cE _{ R } \in \Perf ( \cX )
\)
by~\cite[Proposition 3.6.1]{MR2177199}.

We next show that
\(
    \cE _{ R }
\)
is an \( f \)-exceptional object. This is equivalent to the assertion
\(
    \cC = 0
\),
where
\(
    \cC \in \Perf B
\)
is defined as the cone of the following standard morphism.
\begin{align}
    \cO _{ B } \to \bR f _{ \ast } \bR \cHom _{ \cX } ( \cE _{ R }, \cE _{ R } )
\end{align}
Consider the following Cartesian diagram.
\begin{equation}
    \begin{tikzcd}
    X _{ 0 } \arrow{r}{\iota} \arrow[swap]{d}{f _{ 0 }}
    \arrow[dr, phantom, "\lrcorner", very near start]
    & \cX \arrow{d}{f}\\
    \Spec \bfk \arrow[swap]{r}{0} & \Spec R = B
    \end{tikzcd}
\end{equation}

By \pref{lm:base change} it follows that
\begin{align}
    \bL 0 ^{ \ast } \bR f _{ \ast } \cRHom _{ \cX } ( \cE _{ R }, \cE _{ R } )
    \simeq
    \RHom _{ X _{ 0 } } ( \cE, \cE )
    =
    \bfk \id _{ \cE } [ 0 ],
\end{align}
so that
\(
    \bL 0 ^{ \ast } \cC = 0
\).
By Nakayama's lemma, this implies that \( \cC = 0 \).
The semiorthogonality of the collection
\(
    \cE _{ 1, R }, \dots, \cE _{ N, R }
\)
is shown by similar arguments.
\end{proof}

\begin{remark}\label{rm:deformation of full collection is full}
One can similarly show that the deformation of a strong collection is also strong.
This follows from the fact that if
\(
    P \in \Perf B
\)
satisfies
\(
    0 ^{ \ast } P \simeq \bfk ^{ \oplus r } [0]
\)
for some \( r \ge 0\), then
\(
    P \simeq R ^{ \oplus r } [ 0 ]
\).
Also, it follows from \pref{lm:f-exceptional object and semiorthogonal decomposition} \pref{it:f-exceptional collection induces f-linear semiorthogonal decomposition} that the deformation of a full exceptional collection is again full.
\end{remark}

\begin{definition}
Let \( f \) be a morphism as in \eqref{eq:f cX -> B, general}.
An (\( f \)-)exceptional pair is an (\(f\)-)exceptional collection of length 2.
For an \( f \)-exceptional pair
\(
    \cE, \cF
\),
the left mutation
\(
    L _{ \cE } \cF
\)
of \( \cF \) through \( \cE \) and the right mutation
\(
    R _{ \cF } \cE
\)
of \( \cE \) through \( \cF \) are defined by the following distinguished triangles.
\begin{align}
     f ^{ \ast } \bR f _{ \ast } \cRHom _{ \cX } ( \cE, \cF ) \otimes _{ \cO _{ \cX } } \cE
     \xrightarrow{ \varepsilon }
     \cF
     \to
     L _{ \cE } \cF,\\
     R _{ \cF } \cE
     \to
     \cE
     \xrightarrow{ \eta }
     f ^{ \ast } \bR f _{ \ast } \cRHom _{ \cX } ( \cE, \cF ) ^{ \vee } \otimes _{ \cO _{ \cX } } \cF
\end{align}
\end{definition}

\begin{remark}
The definition of mutations given above differs by shifts from the one in \cite[Section 2]{Bondal_RAACS}, but is slightly simpler in that for an orthogonal exceptional pair, the mutations just exchange the two objects without any shift.
\end{remark}

By the base change theorem \pref{lm:base change}, one can easily verify that mutations commute with base change.
\begin{lemma}\label{lm:mutation commutes with base change}
    Under the notation and the assumptions of \pref{lm:base change}, suppose that \( f \) is flat and hence the natural transformation \eqref{eq:base change isomorphism} is an isomorphism.
    For any \( f \)-exceptional pair
    \(
        \left( \cE, \cF \right)
    \), it follows that
    \(
        \left( \psi ^{ \ast } \cE, \psi ^{ \ast }\cF \right)
    \)
    is an \( g \)-exceptional pair and the following isomorphisms hold.
\begin{align}
    L _{\psi ^{\ast}\cE} \left(\psi ^{\ast}\cF\right) \simeq \psi ^{\ast}( L _{\cE } \cF )\\
    R _{\psi ^{\ast}\cF} \left(\psi ^{\ast}\cE \right) \simeq \psi ^{\ast}( R _{\cF } \cE )
\end{align}
\end{lemma}

Next, let us recall a group action on \( \ec _{ N } ( f )\) from \cite[Proposition 2.1]{MR1230966}.
Let
\(\Br _{N}\)
be the braid group on \( N \) strands, which admits the following famous presentation by generators and relations.
\begin{align}\label{eq:braid group}
    \Br_{N} = \langle \sigma_{1},\dots,\sigma_{N-1} \mid\,
    &
    \sigma _{ i } \sigma _{ i + 1 } \sigma _{ i } = \sigma _{ i + 1 } \sigma _{ i } \sigma _{ i + 1 }
    \quad
    i = 1, \dots, N - 1\\
    &
    \sigma _{ i } \sigma _{ j } = \sigma _{ j } \sigma _{ i }
    \quad
    | i - j | \ge 2
    \rangle
\end{align}
Consider the action
\begin{align}\label{eq:the braid group action}
    \Br_{N}\curvearrowright \ec _{ N } ( f )
\end{align}
given by
\begin{align}
    \sigma_{i} \colon \cE_{i}, \cE_{i+1} \mapsto \cE_{i+1}, R_{\cE_{i+1}}\cE_{i},
\end{align}
so that
\begin{align}
    \sigma_{i}^{-1} \colon \cE_{i}, \cE_{i+1} \mapsto L_{\cE_{i}}\cE_{i+1}, \cE_{i}.
\end{align}

On the other hand, through the standard surjective homomorphism
\(
    \Br _{ N } \twoheadrightarrow \frakS _{ N };
    \quad \sigma _{ i } \mapsto ( i, i + 1 )
\)
to the symmetric group of degree \( N \), the group
\(
    \Br _{ N }
\)
acts naturally on the abelian group
\(
    \bZ ^{ N }
    =
    \Map ( \left\{ 1, \dots, N \right\}, \bZ )
\)
from the left. Let
\begin{align}\label{equation:GN}
    G _{ N } \coloneqq \bZ ^{ N } \rtimes \Br _{ N }
\end{align}
be the semi-direct product corresponding to the action.

One can verify that this, together with the action
\(
    \bZ ^{ N } \curvearrowright \ec _{ N }
\),
where
\(
    [ a _{ 1 }, \dots, a _{ N } ] ^{ T } \in \bZ ^{ N }
\)
sends a collection
\(
    \left( \cE _{ 1 }, \dots, \cE _{ N } \right) \in \ec _{ N }
\)
to
\(
    \left( \cE _{ 1 } [ a _{ 1 } ], \dots, \cE _{ N } [ a _{ N } ] \right)
\),
extends to an action
\(
    G _{ N } \curvearrowright \ec _{ N } ( f )
\).
In particular, one has the following induced action.
\begin{align}\label{eq:induced action of the braid group}
    \Br _{ N } \curvearrowright \ec _{ N } ( f ) / \bZ ^{ N }
\end{align}
When \( \Perf \cX \) admits a full \(f\)-exceptional collection of length \( r \),
then one similarly obtains the action
\(
    G _{ r } \curvearrowright \fec ( \cX )
\).

\begin{remark}\label{rm:transitivity conjecture}
\cite[Conjecture 2.2]{MR1230966} asserts that this action should be transitive.
Note that this conjecture is equivalent to the transitivity of the action
\eqref{eq:induced action of the braid group}. \pref{th:transitivity} below is nothing but the  affirmative answer to
\cite[Conjecture 2.2]{MR1230966} for
\(
    \cX = \hirzebruchtwo
\).
\end{remark}

Let \( X \) be a smooth projective variety over the field \( \bfk \).
The autoequivalences of \( \derived ( X )\) and the notion of exceptional collections are nicely compatible as we explain next.

\begin{lemma}
If
\(
    \cE _{ 1 }, \dots, \cE _{ N } \in \derived ( X )
\)
is an exceptional collection and \( \Phi \in \Auteq ( \derived ( X ) )\), then so is
\(
    \Phi ( \cE _{ 1 } ), \dots, \Phi ( \cE _{ N } ) \in \derived ( X )
\).
In particular, there is the natural action
\begin{align}\label{eq:auteq(f) acting on fec(f)}
    \Auteq ( \derived ( X ) ) \curvearrowright \ec _{ N } ( X ).
\end{align}
\end{lemma}

The following lemma is easy to verify and plays an important role in this paper.
\begin{lemma}\label{lm:actions of Br and Auteq(f) commute}
The actions \eqref{eq:the braid group action} and \eqref{eq:auteq(f) acting on fec(f)} commute.
\end{lemma}

%
%
\subsection{Results obtained via deformation of \( \hirzebruchtwo \) to \( \quadric \)}

Consider the (isotrivial) degeneration \eqref{eq:degeneration of Sigma0 to Sigma2} over an algebraically closed field \( \bfk \).
Consider the discrete valuation ring
\(
    R = \left( \bfk [[t]], ( t ), \bfk \right)
\)
and take the base change by
\( B = \Spec R \to \bA ^{ 1 } _{ t } \)
of the family. We write
\begin{align}\label{eq:formal degenerating family}
    f \colon \cX = \cX \times _{ \bA ^{ 1 } } B \to B
\end{align}
by abuse of notation.
The central fiber \( \cX _{ 0 } \) of \( f \) is isomorphic to \( \hirzebruchtwo \). Also, let
\begin{align}
    \xi \colon \Kbar \coloneqq \overline{\bfk ((t))} \to B
\end{align}
be the geometric generic point of \( B \).
The isotriviality of the family \eqref{eq:degeneration of Sigma0 to Sigma2} outside the origin implies that the geometric generic fiber
\(
    \cX _{ \xi } \to \Spec \Kbar
\)
is isomorphic to
\( \quadric \)
over \(\Kbar\). Throughout this section, we freely use the symbols introduced in this paragraph.

Since the generic fiber \( \quadric \) is a del Pezzo surface, the properties of exceptional collections on it is very well known by \cite{MR1286839}. We list the known properties.

\begin{theorem}\label{th:properties of exceptional collections of Sigma0}
\begin{enumerate}
\item\label{it:fec=fecvb}
Any exceptional object on \(\quadric\) is isomorphic to a shift of an exceptional vector bundle, so that the natural map
\(
    \ec _{ N } ( \quadric ) \hookrightarrow
    \ecvb _{ N } ( \quadric )
\)
is a bijection.

\item\label{it:transitivity for Sigma0}
\(
    \ec _{ 4 } ( \quadric )
    =
    \fec ( \quadric )
\),
and the action
\(
    G _{ 4 } \curvearrowright \fec ( \quadric )
\)
is transitive.

\item\label{it:Constructibility for Sigma0}
Any exceptional collection on \(\quadric\) can be extended to a full exceptional collection.
\end{enumerate}
\end{theorem}

The non-triviality of the group
\( \Auteqtriv ( \derived ( \hirzebruchtwo ) ) \) implies that an exceptional object on
\(
    \hirzebruchtwo
\)
is not uniquely determined by its class in \( K _{ 0 } \), even modulo shifts by
\(
    \bZ [ 2 ]
\). However, if one considers only exceptional vector bundles, then it is the case:

\begin{lemma}[{\(=\)a weaker version of \cite[Lemma 3.5]{MR3431636}}]\label{lm:Lemma 3.5 of [OU]}
Let
\(
    \cE, \cE '
\)
be exceptional vector bundles on
\( \hirzebruchtwo \)
such that
\(
    [ \cE ] = [ \cE ' ] \in \kgr{ \hirzebruchtwo }
\).
Then
\(
    \cE \simeq \cE '
\).
\end{lemma}

The degeneration \eqref{eq:formal degenerating family} allows one to compare various invariants of
\( \hirzebruchtwo \) to those of \( \quadric \).
Recall that
\(
    \cX _{ 0 }
    \simeq
    \hirzebruchtwo
\).

\begin{definition}\label{df:deformation of exceptional collection}
For an exceptional object
\(
    \cE \in \derived ( \cX _{ 0 } )
\),
let
\(
    \cE _{ R }, \gen ( \cE )
\)
denote the unique deformation of
\(\cE\) to \( \cX \) and its restriction to the geometric generic fiber
\(
    \cX _{ \xi }
\), respectively.
For an exceptional collection
\(\cEbar\) on the central fiber, we will similarly write
\(
    \cEbar _{ R }, \gen ( \cEbar )
\)
to mean its (unique) deformation to \( \cX \) and its restriction to \( \cX _{ \xi } \), respectively.
\end{definition}

Let
\begin{align}\label{eq:generalization map}
    \gen \colon \ec _{ 4 } ( \cX _{ 0 } ) \to \fec ( \cX _{ \xi } )
\end{align}
be the map which sends (an isomorphism class of) an exceptional collection
\(
    \cEbar
\)
of length 4 on \( \hirzebruchtwo \) to
\(
    \gen ( \cEbar ) \in \fec ( \cX _{ \xi } )
\),
which is obtained by restricting the deformation of \( \cEbar \) to the \( f \)-exceptional collection, whose existence and uniqueness is guaranteed by \pref{lm:exceptional collection deforms}, to the geometric generic fiber (see \pref{cr:base change of exceptional collections}).
See \pref{cr:gen is surjective} below for the surjectivity of \( \gen \).

One similarly defines the map
\(
    \genvb \colon \ecvb _{ 4 } ( \cX _{ 0 } ) \to \fecvb ( \cX _{ \xi } )
\),
to obtain the following diagram.
\begin{equation}\label{fg:comparison of ec4 and ecvb4}
\begin{tikzcd}
    \ec _{ 4 } ( \cX _{ 0 } ) \arrow[r, "\gen", twoheadrightarrow] & \fec ( \cX _{ \xi } ) \\
    \ecvb _{ 4 } ( \cX _{ 0 } ) \arrow[u, hookrightarrow] \arrow[r, swap,"\genvb"] & \fecvb ( \cX _{ \xi } )
    \arrow[u, hookrightarrow]
\end{tikzcd}
\end{equation}

We next compare \( \kgr { \Perf } \) of the surfaces.
Note that we have the following diagram of schemes (the labels of the arrows in the diagram will be freely used).
\begin{equation}
    \begin{tikzcd}
        \cX _{ 0 } \arrow[r, "i", hookrightarrow] \arrow[d, "f _{ 0 }" ] \arrow[dr, phantom, "\lrcorner", very near start]
        & \cX \arrow[d, "f"]
        & \cX _{ \xi } \arrow[l, "j"'] \arrow[d, "f _{ \xi }" ] \arrow[dl, phantom, "\llcorner", very near start]\\
        \Spec \bfk \arrow[r, "\iota"', hookrightarrow] & \Spec R
        & \Spec \Kbar \arrow[l, "\jbar"]
    \end{tikzcd}
\end{equation}

Applying the functor
\(
    \kgr{ \Perf ( - ) }
\),
we obtain the first two rows of
\pref{fg:Relations among the K0-groups},
which is a commutative diagram of commutative rings with units.
\begin{figure}
\begin{equation}\label{fg:Relations among the K0-groups}
    \begin{tikzcd}
        \kgr{ \cX _{ 0 } } \arrow[d, "f_{ 0 \ast }"' ]
        &
        \kgr{ \cX } \arrow[l, " \simeq ", "i ^{ \ast }"'] \arrow[r, "j ^{ \ast }", "\simeq"'] \arrow[d, "f _{ \ast }" ]
        &
        \kgr{ \cX _{ \xi } } \arrow[d, "f _{ \xi \ast }" ]\\
        \kgr{ \bfk } \arrow[d, "\simeq"' sloped]
        &
        \kgr{ R } \arrow[l, "\iota ^{ \ast }" ] \arrow[r, "\jbar ^{ \ast }"'] \arrow[d, "\simeq" sloped]
        &
        \kgr{ \Kbar } \arrow[d, "\simeq" sloped]\\
        \bZ
        &
        \bZ
        \arrow[l, " = "]
        \arrow[r, " = "']
        &
        \bZ
    \end{tikzcd}
\end{equation}
\end{figure}

The derived dual \( {}^{\vee}\) defined in \eqref{eq:derived dual} induces an automorphism of commutative rings
\begin{align}
    \kgr{ \Perf \cX }
    =
    \kgr{ \left( \Perf \cX \right) ^{ \op } }
    \xrightarrow[\simeq]{\vee}
    \kgr{ \Perf \cX };\quad
    [\cE] \mapsto [ \cE ^{ \vee } ].
\end{align}
The \emph{Euler pairing} on
\(
    \kgr{ \Perf \cX  } = \kgr{ \cX }
\)
(note that \( \cX \) is a regular scheme)
is the following bilinear pairing.
\begin{align}
    \chi _{ \cX } \colon
    \kgr{ \cX } \times \kgr{ \cX }
    \to
    \kgr{ R }; \quad
    ( v, w )
    \mapsto
    f _{ \ast } \left( v ^{ \vee } \cdot w \right)
\end{align}
We can similarly define
\begin{align}
    \chi _{ \cX _{ 0 } } \colon
    \kgr{ \cX _{ 0 } } \times \kgr{ \cX _{ 0 } }
    \to
    \kgr{ \bfk }; \quad
    ( v, w )
    \mapsto
    f _{ 0 \ast } \left( v ^{ \vee } \cdot w \right),\\
    \chi _{ \cX _{ \xi } } \colon
    \kgr{ \cX _{ \xi } } \times \kgr{ \cX _{ \xi } }
    \to
    \kgr{ \Kbar }; \quad
    ( v, w )
    \mapsto
    f _{ \xi \ast } \left( v ^{ \vee } \cdot w \right).
\end{align}

\begin{lemma}
\begin{align}
    \iota ^{ \ast } \circ \chi _{ \cX } = \chi _{ \cX _{ 0 } } \circ ( i ^{ \ast } \times i ^{ \ast } ),\\
    \jbar ^{ \ast } \circ \chi _{ \cX } = \chi _{ \cX _{ \xi } } \circ ( j ^{ \ast } \times j ^{ \ast } )
\end{align}    
\end{lemma}

Let
\(
    \cEbar
\)
be a full exceptional collection of
\(
    \cX _{ 0 }
\),
and
\(
    \cEbar _{ R },
    \cEbar _{ \xi },
\)
be the deformation of \( \cEbar \) to \( \cX \) and the restriction of \( \cEbar _{R } \) to
\( \cX _{ \xi } \) as defined in \pref{df:deformation of exceptional collection}.
As pointed out in \pref{rm:deformation of full collection is full}, both
\(
    \cEbar _{ R }, \cEbar _{ \xi }
\)
are full exceptional collections.
\begin{lemma}
\(
    \kgr{ \cEbar },
    \kgr{ \cEbar _{ R } }.
    \kgr{ \cEbar _{ \xi } }
\)
are bases of
\(
    \kgr{ \cX _{ 0 } },
    \kgr { \cX },
    \kgr{ \cX _{ \xi }}
\),
respectively.
In particular, the horizontal maps \( i ^{ \ast }, j ^{ \ast } \) in the first row of
\eqref{fg:Relations among the K0-groups} are isomorphisms.
\end{lemma}
\begin{proof}
Immediately follows from the fact that the collections
\( \cEbar, \cEbar _{ R }, \cEbar _{ \xi } \)
are full exceptional collections of the triangulated categories over the base of length \(4\).
\end{proof}

On the other hand we obtain the bottom row of the diagram \pref{fg:Relations among the K0-groups}, where the vertical maps to the bottom row are isomorphisms of rings.
Let
\begin{align}\label{eq:generalization map for K0}
    \kgr{ \gen } \colon \kgr{ \cX _{ 0 } } \to \kgr { \cX _{ \xi } }
\end{align}
be the isomorphism of abelian groups obtained from the diagram \pref{fg:Relations among the K0-groups}.
Also, regard
\(
    \chi _{ \cX _{ 0 } },
    \chi _{ \cX _{ \xi } }
\)
as
\( \bZ \)-valued bilinear pairings by the diagram \pref{fg:Relations among the K0-groups}.
With all the preparations above, we can show the desired properties of the map
\(
    \kgr{ \gen }
\).
\begin{proposition}
The isomorphism
\( \kgr{ \gen }\)
of
\eqref{eq:generalization map for K0}
respects the pairings
\(
    \chi _{ \cX _{ 0 } },
    \chi _{ \cX _{ \xi } }
\)
on the source and the target abelian groups. Moreover, it fits in the following commutative diagram.
\[
    \begin{tikzcd}
        \ec _{ 1 } ( \cX _{ 0 } )
        \arrow[r, "\gen"]
        \arrow[d, "\kgr{}"]
        &
        \ec _{ 1 } ( \cX _{ \xi } )
        \arrow[d, "\kgr{}"]\\
         \kgr{ \cX _{ 0 } } \arrow[r, "\kgr{ \gen }"', "\simeq"]
         &
        \kgr { \cX _{ \xi } }
    \end{tikzcd}
\]
\end{proposition}

\begin{proposition}\label{pr:characterization of exceptional objects in the same class}
Let
\(
    \cE, \cE ' \in \derived ( \cX _{ 0 } )
\)
be exceptional objects. The following conditions are equivalent.
\begin{enumerate}
\item\label{it:same class in K0}
\(
    [ \cE  ] = [ \cE ' ] \in \kgr{ \cX _{ 0 } }
\).

\item\label{it:same deformation up to even shift}
\(
    \gen ( \cE ) \simeq \gen ( \cE ' ) [ 2 m ]
\)
for some
\(
    m \in \bZ
\).
\end{enumerate}
\end{proposition}
\begin{proof}
\pref{it:same deformation up to even shift}
\(\Rightarrow\)
\pref{it:same class in K0}
is an consequence of the commutative diagram above and that
\(
    \kgr{ \gen }
\)
is an isomorphism.

Conversely, assume \pref{it:same class in K0}.
It then follows that
\(
    [ \gen ( \cE ) ]
    =
    [ \gen ( \cE ' ) ]
    \in
    \kgr{ \cX _{ \xi } }
\), which in turn implies \pref{it:same deformation up to even shift} by \pref{th:Kuleshov and Orlov} \eqref{it:class determines exceptional object}.
\end{proof}

For exceptional vector bundles on \( \hirzebruchtwo \simeq \cX _{ 0 } \), we have the following reconstruction result.
This is an immediate corollary of \pref{lm:Lemma 3.5 of [OU]}.
\begin{lemma}\label{lm:exceptional vb is determined by K0}
For any
\(
    N = 1, \dots, 4
\),
the map
\(
    \ecvb _{ N } ( \hirzebruchtwo ) \xrightarrow{ \kgr{} } \numec _{ N } ( \hirzebruchtwo )
\)
is injective.
\end{lemma}

See \pref{df:numerical exceptional collection} for the definition of the set
\(
    \numec _{ N } ( \hirzebruchtwo )
\).

\begin{definition}\label{df:numerical exceptional collection}
Let
\(
    S
\)
be a smooth projective variety, and for simplicity let us assume that
\(
    \kgr { S }
\)
is isomorphic to the numerical Grothendieck group; i.e., the Euler pairing \( \chi \) is non-degenerate.
This in particular implies that
\(
    \kgr { S }
\)
is a free abelian group of finite rank.

An exceptional vector is an element
\(
    e \in \kgr{ S }
\)
such that
\(
    e ^{ 2 } = \chi _{ S } ( e, e ) = 1
\).

A \emph{numerical exceptional collection} on \( S \)
is a sequence of exceptional vectors
\(
    e _{ 1 }, \dots, e _{ N } \in \kgr{ S }
\)
such that
\(
    \chi _{ S } ( e _{ j }, e _{ i } )
    =
    0
\)
for any
\(
    1 \le i < j \le N
\).
A numerical exceptional collection is said to be \emph{full} if it is a basis of
\(
    \kgr{ S }
\); i.e, when
\( N = \rank \kgr{ S } \).

The set of numerical exceptional collections of length \( N \) (resp. full) on \( S \) will be denoted by
\(
    \numec _{ N } ( S )
\)
and
\(
    \numfec ( S )
\), respectively.

For a numerical exceptional collection
\(
    ( e, f )
\)
of length \(2\), which will also be called a numerical exceptional pair, its right and left mutations are the new numerical exceptional pairs defined and denoted as follows.
\begin{align}
    \left( f, R _{ f } ( e )
    \coloneqq
    e - \chi _{ S } ( e, f ) f \right)\label{eq:numerical right mutation}\\
    \left( L _{ e } ( f )
    \coloneqq
    f - \chi _{ S } ( e, f ) e
    , e    
    \right)\label{eq:numerical left mutation}
\end{align}
\end{definition}

By similar arguments for the action
\(
    G _{ N } \curvearrowright \ec _{ N }
\),
one can verify that for a surface \( S \) with
\(
    \kgr{ S } \simeq \bZ ^{ N }
\)
there is an action
\(
    G _{ N } \curvearrowright \numfec ( S )
\),
where the subgroup
\(
    \Br _{ N }
\)
acts by the mutations \eqref{eq:numerical right mutation} \eqref{eq:numerical left mutation} and \( \bZ ^{ N } \) by the change of signs (hence the action descends to the quotient
\(
    \left( \bZ / 2 \bZ \right) ^{ N } \rtimes \Br _{ N }
\)).

At last we claim that everything goes together.

\begin{proposition}\label{pr:everything goes together}
There exists the following commutative diagram of sets equipped with the action of the group
\( G _{ 4 } = \bZ ^{ 4 } \rtimes \Br _{ 4 } \).
\[
    \begin{tikzcd}
        \ec _{ 4 } ( \cX _{ 0 } ) \arrow[r, "\gen", twoheadrightarrow] \arrow[d, "\kgr{}"]
         & \fec ( \cX _{ \xi } )  \arrow[d, "\kgr{}"]\\
        \numfec ( \cX _{ 0 } ) \arrow[r, "\kgr{\gen}"', "\simeq"]
        & \numfec ( \cX _{ \xi } )
    \end{tikzcd}
\]
\end{proposition}

Let us now introduce a particular (\(f\)-)exceptional collection of invertible sheaves, which will be called \emph{the standard collection} and serve as the base point of the sets \( \fec \) and \( \fecvb \) in this paper.

\begin{definition}\label{df:the standard exceptional collection on Sigma2}
The \emph{standard full exceptional collection} of invertible sheaves on \(\hirzebruchtwo \simeq \cX _{ 0 }\) is defined as follows.
\begin{align}\label{eq:standard collection}
    \cEstd
    \coloneqq
    \left(
        \cO _{ \hirzebruchtwo },
        \cO _{ \hirzebruchtwo } ( f ),
        \cO _{ \hirzebruchtwo } ( C + 2 f ),
        \cO _{ \hirzebruchtwo } ( C + 3 f )
    \right)
    \in
    \fecvb ( \hirzebruchtwo ).
\end{align}
By \pref{lm:exceptional collection deforms} and \pref{rm:deformation of full collection is full}, it uniquely deforms to a full \(f\)-exceptional collection of invertible sheaves. We will write it
\( \cEstd _{ R } \), and its pullback to
\( \cX _{ \xi } \) will be denoted by
\( \cEstd _{ \xi } \).
Using the deformation invariance of the intersection numbers, one can easily confirm that
\begin{align}\label{eq:standard collection generic}
    \cEstd _{ \xi }
    =
    \left(
        \cO _{ \cX _{ \xi } },
        \cO _{ \cX _{ \xi } } ( 1, 0 ),
        \cO _{ \cX _{ \xi } } ( 1, 1 ),
        \cO _{ \cX _{ \xi } } ( 2, 1 )
    \right)
\end{align}
under an isomorphism
\( \cX _{ \xi } \simeq \quadric \).
\end{definition}

\begin{corollary}\label{cr:gen is surjective}
The map
\(
    \gen \colon
    \ec _{ 4 } ( \cX _{ 0 } )
    \to
    \fec ( \cX _{ \xi } )
\)
is surjective.
\end{corollary}
\begin{proof}
By \pref{th:properties of exceptional collections of Sigma0} \pref{it:transitivity for Sigma0}, we know that
\(
    \fec ( \cX _{ \xi } )
    =
    G _{ 4 } \cdot \cEstd _{ \xi }
\).
Since the diagram of \pref{pr:everything goes together} is
\( G _{ 4 } \)-equivariant and
\(
    \gen ( \cEstd )
    =
    \cEstd _{ \xi }
\),
we obtain the conclusion.
\end{proof}

\begin{corollary}\label{cr:transitivity at the numerical level}
For any
\(
    \cEbar \in \ec _{ 4 } ( \cX _{ 0 } )
\),
there exists
\(
    \sigma \in G _{ 4 }
\)
such that
\(
    \kgr{ \sigma ( \cEbar ) } = \kgr{ \cEstd }
    \in
    \numfec ( \cX _{ 0 } )
\).
\end{corollary}
\begin{proof}
Again by \pref{th:properties of exceptional collections of Sigma0} \pref{it:transitivity for Sigma0}, one can find an element
\(
    \sigma \in G _{ 4 }
\)
such that
\(
    \sigma \gen ( \cEbar ) = \cEstd _{ \xi }
\).
Again by the equivariance, it follows that
\(
    \gen ( \sigma ( \cEbar ) )
    =
    \cEstd _{ \xi }
    =
    \gen ( \cEstd )
\),
which means that
\(
    \kgr{ \sigma ( \cEbar )}
    =
    \kgr{ \cEstd }
\)
by \pref{pr:characterization of exceptional objects in the same class}.
\end{proof}

%
%

\section{Twisting exceptional objects down to exceptional vector bundles}
\label{sc:Twisting exceptional objects down to exceptional vector bundles}

The purpose of this section is to show the following theorem.

\begin{theorem}\label{th:exceptional objects are equivalent to vector bundles}
For any exceptional object
\(
    \cE \in \derived ( \hirzebruchtwo )
\),
there exists an exceptional vector bundle
\(
    \cF
\)
and a sequence of integers
\(
    a _{ 1 }, \dots, a _{ n }, m
\)
such that
\begin{align}\label{eq:exceptional objects are equivalent to vector bundles}
    \cE
    \simeq
    \left( T _{ a _{ n } } \circ \cdots \circ T _{ a _{ 1 } } \right) ( \cF ) [ m ].
\end{align}
\end{theorem}

The similar result for spherical objects is given in \cite[Proposition 1.6]{Ishii-Uehara_ADC}.
In fact, we prove \pref{th:exceptional objects are equivalent to vector bundles} by suitably modifying the proof of \cite[Proposition 1.6]{Ishii-Uehara_ADC}.

\begin{notation}\label{nt:support}
Let
\(
    X
\)
be an integral noetherian scheme.
For
\(
    \cE \in \coh X
\), we define
\begin{align}
    \Supp \cE
    \coloneqq
    \Spec
    \Image
    \left(
        \cO _{ X } \to \cEnd _{ X } ( \cE )
    \right)
    \subset X
\end{align}
and call it the \emph{schematic support} of \( \cE \).
It is universal among the closed subschemes
\(
    \iota \colon Z \hookrightarrow X
\)
which admits a coherent sheaf
\(
    \cE ' \in \coh Z
\)
such that
\(
    \iota _{ \ast } \cE ' \simeq \cE
\).
The underlying closed subset of \(X\), or equivalently the reduced closed subscheme
\(
    ( \Supp \cE ) _{\reduced } \subset X
\),
is called the \emph{reduced support} of \( \cE \).
Also we let
\(
    \tors \cE \subset \cE
\)
be the maximum torsion subsheaf of \( \cE \). For an object
\(
    \cE \in \derived ( X )
\),
we use the following notation.
\begin{align}\label{eq:cohomology object}
    \cohomology ( \cE ) \coloneqq \bigoplus _{ i \in \bZ } \cH ^{ i } ( \cE )
    \quad
    \Supp \cE \coloneqq \bigcup _{ i \in \bZ } \left( \Supp \cH ^{ i } ( \cE ) \right) _{ \reduced }
\end{align}
\end{notation}

%
%
\subsection{First properties of \( \cohomology ( \cE ) \)}
\label{sc:First properties}

As the first step toward the proof of \pref{th:exceptional objects are equivalent to vector bundles}, in this subsection we prove some basic properties of \( \cohomology ( \cE ) \).
Part of them concern
\(
    \left( \Supp \cH ^{ i } ( \cE ) \right) _{ \reduced }
\), which can be summarized as follows.

\begin{quote}
There exists the unique integer
    \(
        i _{ 0 } \in \bZ
    \)
    with the following properties.
    \begin{itemize}
    \item 
    \(
        \Supp \cH ^{ i _{ 0 } } ( \cE ) = \hirzebruchtwo
    \)
    \item
    \(
        \left(\Supp \tors \cH ^{ i _{ 0 } } ( \cE )\right) _{ \reduced } = C
    \)
    \item
    \(
        \left(\Supp \cH ^{ i } ( \cE )\right) _{ \reduced } = C
    \)
    for
    \(
        \forall i \neq i _{ 0 }
    \)
\end{itemize}
\end{quote}

Based on this, the similar statements for the \emph{schematic} supports will be proved in the next subsection; namely, we will remove \(\reduced\) from the second and the third items. In fact, this step has been the main obstacle for the project.

The similar results for spherical objects appear in \cite[Lemma 4.8]{Ishii-Uehara_ADC}, where it is rather easily shown that the schematic support of the cohomology sheaves of a spherical object coincides with \(C\). However, unfortunately, the proof of \cite[Lemma 4.8]{Ishii-Uehara_ADC} does not immediately apply to our situation. The fact \( \Supp \cE = \hirzebruchtwo \) prevents us from studying the problem locally around the curve \( C \).

\begin{lemma}\label{lm:properties of cohomology sheaves}
The cohomology sheaves of an exceptional object $\cE \in \derived ( \hirzebruchtwo )$
enjoy the following properties.
\begin{enumerate}
    \item\label{it:cohomology(E) is rigid}
    \(
        \cohomology ( \cE )
    \)
    is rigid
    (\(
        \iff
        E _{ 2 } ^{ 1, q } = 0 \ \forall q \in \bZ
    \)
    in the spectral sequence \eqref{eq:E_2 spectral sequence of Hom}).\label{it:rigid}
    
    \item\label{it:support}
    There exists a unique integer
    \(
        i _{ 0 } \in \bZ
        \)
        such that
        \begin{itemize}
            \item
            \(
                \Supp \cH ^{ i _{ 0 } } ( \cE )
                =
                \hirzebruchtwo
            \).
                
            \item
            \(
                \cH ^{ i } ( \cE )
            \)
            for
            \(
                i \ne i _{ 0 }
            \)
            and
            \(
                \tors \cH ^{ i _{ 0 } } ( \cE )
            \)
            are pure sheaves with
            \(
                \left(\Supp ( - )\right) _{\reduced} = C
            \),
            unless
            \(
                = 0
            \).
        \end{itemize}
                        
    \item \label{it:more on i0-th cohomology}
    $
        \tors(\cH^{i_0}(\cE) )
    $
    is rigid, and
    \(
        \cH^{i_0}(\cE) /\tors(\cH^{i_0}(\cE) )
    \)
    is an exceptional vector bundle.
\end{enumerate}
\end{lemma}

See the following definition for the notion of rigidity.

\begin{definition}
An object
\(
    \cE \in \derived ( X )
\)
is said to be \emph{rigid} if
\(
    \Ext _{ X } ^{ 1 } ( \cE, \cE ) = 0
\).
\end{definition}

We will freely use the following standard fact on rigid objects.

\begin{lemma}\label{lm:rigid objects are stable under the aut0 action}
Let
\(
    \cE \in \derived ( X )
\)
be a rigid object. Then
\(
    g \cE \simeq \cE
\)
for any
\(
    g \in \Aut ^{ 0 } _{ X / \bfk }
\).
In particular,
\(
    \Supp g \cE = \Supp \cE \subset X
\)
for any such \( g \).
\end{lemma}

We need the following spectral sequence, which also plays the central role for the study of spherical objects in \cite{Ishii-Uehara_ADC}.

\begin{lemma}\label{lm:E_2 spectral sequence of Hom}
For any object
\(
    \cE \in \derived ( X )
\),
there exists the following spectral sequence.
\begin{align}\label{eq:E_2 spectral sequence of Hom}
    E_2^{p,q}
    =
    \bigoplus_i \Ext _{ \hirzebruchtwo } ^p(\cH^i(\cE), \cH^{i+q}(\cE)) \Rightarrow
    E ^{ p + q } = \Hom _{ \hirzebruchtwo } ^{p+q}(\cE, \cE)
\end{align}

Moreover, using the classes
\(
    e^i
    =
    e^i(\cE)
    \in
    \Ext _{ \hirzebruchtwo } ^2( \cH^i(\cE), \cH^{i-1}(\cE) )
\)
canonically determined by \( \cE \), the
\(
    d _{ 2 }
\)
maps of \eqref{eq:E_2 spectral sequence of Hom} are given by
\begin{align}\label{eq:d2 map}
    d_2 ^{ p, q } \colon
    ({\phi_i})_i
    \mapsto
    \left(
        (-1)^{ p + q } \phi _{ i - 1 } \circ e ^{ i } - e ^{ i + q } \circ \phi _{ i }
    \right) _{ i }.    
\end{align}
\end{lemma}

\begin{proof}
See \cite[Section 4.1]{Ishii-Uehara_ADC}, in particular \cite[Proposition 4.1]{Ishii-Uehara_ADC}.
\end{proof}

\begin{proof}[Proof of \pref{lm:properties of cohomology sheaves}]
Throughout the proof, we consider the spectral sequence \eqref{eq:E_2 spectral sequence of Hom} for the exceptional object \( \cE \). We prove the four items one by one.

\pref{it:rigid}
The exceptionality of \( \cE \) is translated into the following conditions.
\begin{align}\label{eq:exceptionality as conditions on spectral sequence}
    E ^{ n } =
    \begin{cases}
        \bfk \id _{ \cE } & n = 0\\
        0 & n \ne 0
    \end{cases}
\end{align}
Since \( \hirzebruchtwo \) is a smooth projective surface,
$
    E_2^{p,q} = 0
$
unless $0 \le p \le 2$.
Hence \eqref{eq:E_2 spectral sequence of Hom} is \( E _{ 3 } \)-degenerate, and moreover is 
\( E _{ 2 } \)-degenerate at
\(
    ( 1, q )
\)
for any
\(
    q \in \bZ
\).
Combined with \eqref{eq:exceptionality as conditions on spectral sequence},
this implies that
\(
    E _{ 2 } ^{ 1, q }
    =
    0
\)
for any
\(
    q \ne - 1
\). In the next paragraph we also confirm
\(
    E _{ 2 } ^{ 1, - 1 } = 0
\),
thereby concluding the rigidity of
\(
    \cohomology ( \cE )
\).

It follows from the explicit description \eqref{eq:d2 map} of \(d _{ 2 } \) maps that
\(
    0
    \neq
    \sum _{ i \in \bZ } \id _{ \cH ^{ i } ( \cE ) } \in E _{ 2 } ^{ 0, 0 }
\)
is in fact contained in
\(
    \Ker d _{ 2 } ^{ 0, 0 }
    \simeq
    E _{ 3 } ^{ 0, 0 }
    \simeq
    E _{ \infty } ^{ 0, 0 }
\),
which implies that
\(
    E _{ \infty } ^{ 0, 0 } \ne 0
\).
Combined with the isomorphism
\(
    E ^{ 0 } \simeq \bfk
\)
from \eqref{eq:exceptionality as conditions on spectral sequence} and the epimorphism
\(
    E ^{ 0 }
    \twoheadrightarrow
    E ^{ 0 } / F ^{ 1 } E ^{ 0 }
    \simeq
    E _{ \infty } ^{ 0, 0 }
\), this implies that
\(
    E _{ \infty } ^{ 0, 0 }
    \simeq
    E ^{ 0 }
    (\simeq \bfk)
\)
and hence
\(
    E _{ 2 } ^{ 1, - 1 }
    \simeq
    E _{ \infty } ^{ 1, - 1 }
    \simeq
    F ^{ 1 } E ^{ 0 } / F ^{ 2 } E ^{ 0 }
    =
    0
\).

\pref{it:support}
We further obtain the following equalities from \eqref{eq:exceptionality as conditions on spectral sequence}.
\begin{align}
    e_2^{0,q}&=e_2^{2, q-1} \quad (\text{if $q \ne 0$})\label{eq:e^{0,q} = e^{2,q-1}} \\
    e_2^{0,0}&=e_2^{2,-1} + 1 \label{eq:e^{0,0} = e^{2,-1}+1}
\end{align}
To see \eqref{eq:e^{0,q} = e^{2,q-1}} for \( q = - 1 \), note that the arguments in the previous paragraph implies
\(
    0 = E _{ \infty } ^{ 2, - 2 } \simeq E _{ 3 } ^{ 2, - 2 }
\).
Moreover, \pref{it:rigid} implies that either
\(
    \Supp \left( \tors \cH^i(\cE) \right) _{ \reduced }
    =
    \hirzebruchtwo,
    C,
    \ \text{or} \ 
    \emptyset
\).
This follows from \pref{lm:rigid objects are stable under the aut0 action} and the orbit decomposition
\(
    \hirzebruchtwo
    =
    \left( \hirzebruchtwo \setminus C \right)
    \coprod
    C
\)
for the action
\(
    \Aut _{ \hirzebruchtwo / \bfk } \curvearrowright \hirzebruchtwo
\).

Suppose for a contradiction that $\Supp \cH^i(\cE) _{ \reduced } \ne \hirzebruchtwo$ for every $i$.
Then $\Supp \cE$ coincides with $C$ and hence there is an isomorphism $\cE \otimes \omega _{ \hirzebruchtwo } \simeq \cE$,
from which we deduce
$
    \ext _{ \hirzebruchtwo } ^2(\cE, \cE)
    =
    \ext _{ \hirzebruchtwo } ^0(\cE, \cE)
$
by the Serre duality.
This contradicts \eqref{eq:exceptionality as conditions on spectral sequence}
and hence there must be at least one integer $i_0$ with
\(
    \Supp \cH^{i_0}(\cE) _{ \reduced } = \hirzebruchtwo
\).

Now consider the Serre duality
\begin{align}
    (E_2^{2,q})^\vee
    \simeq
    \bigoplus_i \Hom_{ \hirzebruchtwo }(\cH^i(\cE), \cH^{i-q}(\cE)\otimes \omega _{ \hirzebruchtwo }).
\end{align}
Fix a non-trivial morphism
$
    s \colon \omega _{ \hirzebruchtwo } \hookrightarrow \cO _{ \hirzebruchtwo }
$
such that the support of its cokernel is disjoint from the set of associated points of
\(
    \cH^{i}(\cE)
\)
for all
\(
    i \in \bZ
\).
This is possible, as the linear system
\(
    |- K _{ \hirzebruchtwo }|
\)
is base point free and there are only finitely many (schematic) points to be avoided (see \cite[p. 8]{Huybrechts-Lehn}). The property which we required for \( s \) implies that for each \( i \in \bZ \) the natural morphism
\begin{align}
    \cH ^{ i - q } ( \cE ) \otimes \omega _{ \hirzebruchtwo }
    \stackrel{\id \otimes s}{\to}
    \cH ^{ i - q } ( \cE )
\end{align}
is injective. Thus we obtain an injection of vector spaces
\[
    \phi_q \colon (E_2^{2,q})^\vee \hookrightarrow E_2^{0, -q}
\]
and hence an inequality
\begin{align}\label{eq:e_2^{2,q} <= e_2^{0, -q}}
    e_2^{2,q} \le e_2^{0, -q}
\end{align}
for any $q$.

Next we prove that
\(
    \phi _{ 0 }
\)
is not surjective on the direct summands indexed by those
\(
    i
\)
with
\(
    \Supp \cH ^{ i } ( \cE ) = \hirzebruchtwo
\).
To see this, (by slight abuse of notation) let
\(
    i _{ 0 }
\)
be one of such indices and apply the functor
\(
    \Hom ( \cH ^{ i _{ 0 } } ( \cE ), - )
\)
to the following short exact sequence.
\begin{align}
    0
    \to
    \cH ^{ i _{ 0 } } ( \cE ) \otimes \omega _{ \hirzebruchtwo }
    \xrightarrow[]{ \id \otimes s }
    \cH ^{ i _{ 0 } } ( \cE )
    \to
    \cH ^{ i _{ 0 } } ( \cE ) \vert _{ Z ( s ) }
    \to
    0.
\end{align}
The rigidity of
\(
    \cH ^{ i _{ 0 } } ( \cE )
\),
which we confirmed in \pref{it:rigid}, implies that we have the following short exact sequence.
\begin{align}
    0
    \to
    \Hom \left( \cH ^{ i _{ 0 } } ( \cE ), \cH ^{ i _{ 0 } } ( \cE ) \otimes \omega _{ \hirzebruchtwo }\right)
    \to
    \Hom \left( \cH ^{ i _{ 0 } } ( \cE ), \cH ^{ i _{ 0 } } ( \cE ) \right)
    \to
    \Hom \left( \cH ^{ i _{ 0 } } ( \cE ), \cH ^{ i _{ 0 } } ( \cE ) \vert _{ Z ( s ) } \right)
    \to
    0
\end{align}
The 3rd term is not \( 0 \) by the assumption
\(
    \Supp \cH ^{ i _{ 0 } } ( \cE ) = \hirzebruchtwo
\) and hence \( \id \otimes s \), which is identified with the direct summand of \( \phi _{ 0 }\) of interest, is not surjective.

Since we showed above that there is at least one such index \( i _{ 0 }\), we have confirmed the non-surjectivity of \( \phi _{ 0 } \) and hence the following inequality.
\begin{align}\label{eq:e_2^{2,0} < e_2^{0,0}}
    e_2^{2,0} < e _{ 2 } ^{ 0, 0 }
\end{align}

Summarizing the results so far, we obtain the following sequence of (in)equalities.
\[
    e_2^{2,-1} \stackrel{\eqref{eq:e_2^{2,q} <= e_2^{0, -q}} ( q = - 1 )}{\le} e_2^{0,1}
    \stackrel{\eqref{eq:e^{0,q} = e^{2,q-1}} ( q = 1 )}{=}
    e_2^{2,0}
    \stackrel{\eqref{eq:e_2^{2,0} < e_2^{0,0}}}{=}
    e_2^{0,0} - \rank \phi _{ 0 }
    \stackrel{\eqref{eq:e^{0,0} = e^{2,-1}+1}}{=}
    e_2^{2,-1} + 1 - \rank \phi _{ 0 }
    \le
    e _{ 2 } ^{ 2, - 1 },
\]
which implies that
\begin{align}\label{eq:the difference is 1}
    \rank \phi _{ 0 }
    =
    e _{ 2 } ^{ 0, 0 }
    -
    e _{ 2 } ^{ 2, 0 }
    =
    1.
\end{align}
This means that the number of the indices $i_0$ with
$
    \Supp \cH^{i_0} ( \cE ) = \hirzebruchtwo
$
is exactly one.

Finally, recall that rigid sheaf with one-dimensional support is pure by \cite[Corollary 2.2.3]{MR1604186}. We already confirmed the rigidity of \( \cH ^{ i } ( \cE ) \) for \( i \ne i _{ 0 }\) above, and for \( \tors \cH ^{ i _{ 0 } } ( \cE ) \) it is proven below.

\pref{it:more on i0-th cohomology} Put
$
    \cH \coloneqq \cH^{i_0}(\cE)
$,
$
    \cT \coloneqq \tors(\cH)
$
and
$
    \cF \coloneqq \cH/\cT
$.
Consider the spectral sequence
\begin{align}\label{eq:E_1 spectral sequence of Hom}
    E_1^{p,q}
    =
    \bigoplus_j \Ext _{ \hirzebruchtwo }^{p+q}(G_j, G_{j+p})
    \Rightarrow
    E ^{ p + q }
    =
    \Ext _{ \hirzebruchtwo }^{p+q}(\cH, \cH)
\end{align}
arising from the short exact sequence $0 \to G_2 \to \cH \to G_1 \to 0$,
where $G_1=\cF$ and $G_2=\cT$.
For obvious reasons we see $E_1^{p,q}=0$ unless $0 \le p+q \le 2$ and $-1 \le p \le 1$.

On the other hand, since \( \cH \) is rigid, it is stable under the action of
\(
    \Aut _{ \hirzebruchtwo / \bfk } = \Aut _{ \hirzebruchtwo / \bfk } ^{ 0 }
\)
by \pref{lm:rigid objects are stable under the aut0 action}.
Since the torsion part of a coherent sheaf is uniquely determined by the sheaf, it follows that
\(
    \cT
\)
is also stable under the same group action. Hence it follows that
\(
    \left(\Supp \cT\right) _{ \reduced } \subset C
\), so that
\begin{align}\label{eq:cT is stable under Serre functor up to shift}
    \cT \otimes \omega _{ \hirzebruchtwo } \simeq \cT.
\end{align}
Combined with the Serre duality, this implies the equality
\begin{align}\label{eq:e_1^{p,q}
    =
    e_1^{-p, 2-q}}
    e_1^{p,q}
    =
    e_1^{-p, 2-q}
\end{align}
for $p \ne 0$.

It then follows that
\(
    0
    =
    E _{ 1 } ^{ - 1, 1 }
    =
    E _{ 1 } ^{ 1, 1 }
\),
where the first equality is the consequence of the fact that $G_2$ is torsion and $G_1$ is torsion free, and the second equality is the case
\(
    ( p, q )
    =
    (1, 1 )
\)
of
\eqref{eq:e_1^{p,q}
=
e_1^{-p, 2-q}}.

Thus we have confirmed that the spectral sequence
\eqref{eq:E_1 spectral sequence of Hom}
is \( E _{ 1 } \)-degenerate at
\(
    ( 0, 1 )
\). Hence
\(
    E _{ 1 } ^{ 0, 1 }
    \simeq
    E _{ \infty } ^{ 0, 1 }
    \simeq
    E ^{ 1 }
    =
    0
\),
where the last vanishing is nothing but the rigidity of \( \cH \),
which we confirmed in \pref{it:rigid}.
Thus we have shown the rigidity of
\(
    \cT
\)
and
\(
    \cF
\).

Again in the spectral sequence \eqref{eq:E_1 spectral sequence of Hom}, the vanishing
\(
    E ^{ 1 } = 0
\)
implies
\(
    0 = E _{ \infty } ^{ 1, 0 } = E _{ 2 } ^{ 1, 0 }
\). Hence $d_1^{0,0}$ is surjective.
Also the vanishing
\(
    E ^{ 1 } = 0
\)
implies
\(
    E _{ \infty } ^{ - 1, 2 }
    =
    0
\).
Since \( 0 = E _{ 1 } ^{ 1, 1 } =  E _{ 2 } ^{ 1, 1 } \), the spectral sequence is \( E _{ 2 }\)-degenerate at \( ( - 1, 2 ) \) and hence
\(
    E _{ \infty } ^{ - 1, 2 }
    \simeq
    E _{ 2 } ^{ - 1, 2 }
\).
This implies that $d_1^{-1,2}$ is injective.
Thus we obtain
\begin{alignat}{2}
    \hom(\cH, \cH)
    &=
    e_1^{1,-1} + \ker d_1^{0,0}
    &&=
    e_1^{1,-1} + e_1^{0,0} -e_1^{1,0}\label{eq:cH1} \\
    \ext^2(\cH, \cH)
    &=
    e_1^{-1,3} + \coker d_1^{-1,2}
    &&=
    e_1^{-1,3} + e_1^{0,2} -e_1^{-1,2}.\label{eq:cH2}
\end{alignat}
Therefore substituting \eqref{eq:cH1} and \eqref{eq:cH2} into
\begin{align}\label{eq:hom-ext^2}
    \eqref{eq:the difference is 1}
    \iff
    \hom(\cH, \cH)
    =
    \ext ^2(\cH, \cH) + 1
\end{align}
and using \eqref{eq:e_1^{p,q} = e_1^{-p, 2-q}}, we obtain
\[
    1
    =
    \hom(\cH, \cH) - \ext^2(\cH, \cH)
    =
    e_1^{0,0} - e_1^{ 0, 2 }
    \stackrel{\eqref{eq:cT is stable under Serre functor up to shift}}{=}
    \hom(\cF, \cF)-\ext^2(\cF, \cF).
\]
Since $\cF$ is rigid, the same proof as in \cite[Lemma 2.2]{MR3431636} shows that $\cF$ is an exceptional vector bundle.
\end{proof}

\begin{definition}\label{df:i0}
For an exceptional object
\(
    \cE \in \derived ( \hirzebruchtwo )
\),
we will write
\(
    i _{ 0 } ( \cE ) = i _{ 0 } \in \bZ
\)
for the unique integer such that
\(
    \Supp \cH ^{ i _{ 0 } } ( \cE ) = \hirzebruchtwo
\).
\end{definition}

Let
\(
    \rank \colon \kgr{ \hirzebruchtwo } \to \bZ
\)
be the rank function. If we let
\(
    \iota \colon \Spec \bfk ( \hirzebruchtwo ) \to \hirzebruchtwo
\)
denote the embedding of the generic point, \( \rank \) is concisely defined as the composition of the map
\( \kgr{ \iota ^{ \ast } } \) and the isomorphism
\(
    \kgr{ \Spec \bfk ( \hirzebruchtwo )} \simeq \bZ
\)
which sends the class of \( \bfk ( \hirzebruchtwo )\) to \( 1 \).
As a corollary of \pref{lm:properties of cohomology sheaves}, we obtain the following

\begin{corollary}\label{cr:rank is never 0}
    The equality
    \begin{align}
        \rank \cE = ( - 1 ) ^{ i _{ 0 } } \rank \cH ^{ i _{ 0 } } ( \cE )
    \end{align}
    holds for any exceptional object \( \cE \in \derived ( \hirzebruchtwo )\). In particular, \( \rank \cE \neq 0 \).
\end{corollary}

\subsection{Properties of the schematic support of \( \tors \cohomology ( \cE ) \)}
\label{sc:Properties of the schematic support}

The aim of this subsection is to prove \pref{pr:chohomology sheaves are O_C modules}, which asserts that the \emph{schematic} support of \( \tors \cohomology ( \cE ) \) coincides with \(C\). This is the most technical part of this paper.

\begin{remark}
    In this paper, \pref{pr:chohomology sheaves are O_C modules} will be used only as a black box.
    Hence one can first assume \pref{pr:chohomology sheaves are O_C modules} to understand the rest of the paper and then later come back to its proof.
\end{remark}

To describe the schematic support of the cohomology sheaves of $\cE$, we consider the anti-canonical morphism \(f \colon \hirzebruchtwo \to P ( 1, 1, 2 ) \) to the weighted projective plane of weight \( ( 1, 1, 2 ) \), which contracts the $(-2)$-curve $C$ to the singularity \( ( 0 : 0 : 1 ) \).

\begin{notation}\label{nt:A1 singularity}
Let  $(R, \frakm )$ denote the local ring of $P ( 1, 1, 2 )$ at the singular point. It is isomorphic to
\( \bfk [ x, y, z ] _{ ( x, y, z ) } / ( z ^{ 2 } - x y ) \), the \(A _{ 1 }\) singularity.
\end{notation}

Our first goal is to give an \(R\)-module structure on \( E _{ 2 } ^{ p, q } \) for \( ( p, q ) \neq ( 0, 0 ) \) of the spectral sequence \eqref{eq:E_2 spectral sequence of Hom} with respect to which the differentials of the spectral sequence are \(R\)-linear. The similar fact for spherical objects is used in \cite{Ishii-Uehara_ADC}, in which case the existence of such an \(R\)-module structure is trivial. In fact, since the support of a spherical object is concentrated in \(C\), one can use the pushforward along
\( f \vert _{ \Spec R } \).

\begin{remark}\label{rm:R-module structure}
    For any \( \cM \in \coh \hirzebruchtwo \) with \( \left( \Supp \cM \right) _{ \reduced } \subset C \), there is a natural homomorphism of \( \bfk \)-algebras \( R \to \End _{ \hirzebruchtwo } ( \cM ) \) which kills \( \frakm ^{ \ell } \) for some \( \ell > 0 \).
    Hence for \( i \in \bZ \) and \( \cN  \in \coh \hirzebruchtwo \), there are natural \( R \)-module structures of finite length on
    \(
        \Ext _{ \hirzebruchtwo } ^{ i } ( \cM, \cN  )
        \simeq
        \Hom _{ \derived ( \hirzebruchtwo ) } ( \cM, \cN  [ i ] )
    \)
    and
    \(
        \Ext _{ \hirzebruchtwo } ^{ i } ( \cN , \cM )
        \simeq
        \Hom _{ \derived ( \hirzebruchtwo ) } ( \cN , \cM [ i ] )
    \)
    given by precomposition and postcomposition, respectively. If \( \cN  \) is also supported in \( C \), we may use \( R \to \End _{ \hirzebruchtwo } ( \cN ) \) instead.
    However, we end up with the same \(R\)-module structures defined via
    \(
        R \to \End _{ \hirzebruchtwo } ( \cM )
    \).
\end{remark}

For $q \ne 0$, the reduced support of either $\cH^i(\cE)$ or $\cH^{i+q}(\cE)$ is contained in $C$ and therefore $E_2^{p,q}$ has a canonical $R$-module structure by \pref{rm:R-module structure}.
For \( q = 0 \) and
\(
    p \neq 0, 2
\),
\(
    E _{ 2 } ^{ p, 0 }
    =
    0
\).

For
\(
    p = 2
\),
we have
\[
    E_2^{2,0} \simeq
    \bigoplus_{i \ne i_0} \Ext _{ \hirzebruchtwo } ^2(\cH^i(\cE), \cH^i(\cE))
    \oplus \Ext _{ \hirzebruchtwo } ^2 ( \cH^{i_0}(\cE), \cH^{i_0}(\cE) ).
\]
Since the reduced support of \( \cH ^{ i } ( \cE ) \) for \( i \ne i _{ 0 } \) is contained in \( C \), all the direct summands but the last one have canonical
\(
    R
\)-module structures again by \pref{rm:R-module structure}.

Recall the following short exact sequence from the proof of \pref{lm:properties of cohomology sheaves} \pref{it:more on i0-th cohomology}.
\begin{align}
    0
    \to
    \cT \coloneqq \tors\cH^{i_0}(\cE)
    \stackrel{\iota}{\to}
    \cH^{i_0}(\cE)
    \to
    \cF \coloneqq \cH^{i_0}(\cE)/\cT
    \to
    0
\end{align}
Take \( \Hom _{ X } ( -, \cH ^{ i _{ 0 } } ( \cE ) ) \) to obtain the following exact sequence
\[
    \Ext _{ \hirzebruchtwo } ^2(\cF, \cH^{i_0}(\cE)) \to \Ext _{ \hirzebruchtwo } ^2(\cH^{i_0}(\cE), \cH^{i_0}(\cE))
    \stackrel{ \iota ^{ \ast } }{\to}
    \Ext _{ \hirzebruchtwo } ^2(\cT, \cH^{i_0}(\cE)) \to 0,
\]
where
\begin{align}\label{eq:R-mod_str}
\begin{alignedat}{3}
\Ext _{ \hirzebruchtwo } ^2(\cH^{i_0}(\cE)&, \cH^{i_0}(\cE))&  &\to &\Ext _{ \hirzebruchtwo } ^2(&\cT, \cH^{i_0}(\cE)) \\
&x  & &\mapsto & &x \circ \iota
\end{alignedat}
\end{align}
is the precomposition (think of \(x\) as a morphism
\(
    \cH ^{ i _{ 0 } } ( \cE ) \to \cH ^{ i _{ 0 } } ( \cE ) [ 2 ]
\)
in \( \derived ( \hirzebruchtwo ) \)).
The following vanishing follows from the exceptionality of \( \cF \).
\[
    \Ext _{ \hirzebruchtwo } ^2 ( \cF, \cH^{i_0}(\cE) )
    \simeq
    \Ext _{ \hirzebruchtwo } ^2 ( \cF, \cF ) = 0
\]
Hence we conclude that the map \( \iota ^{ \ast } \) is an isomorphism of \(\bfk\)-vector spaces.
Since the reduced support of \( \cT \) is contained in \( C \), the right hand side of \eqref{eq:R-mod_str} has a canonical \(R\)-module structure.

\begin{definition}\label{df:R-module structure on E _{ 2 } ^{ 2, 0 }}
We transfer the \(R\)-module structure on
\(
    \Ext _{ \hirzebruchtwo } ^2 ( \cT, \cH^{i_0}(\cE) )
\)
to
\(
    \Ext _{ \hirzebruchtwo } ^2 ( \cH^{i_0}(\cE), \cH^{i_0}(\cE) )
\)
via the isomorphism \( \iota ^{ \ast } \) of \eqref{eq:R-mod_str},
thereby giving an \(R\)-module structure on
\(
    E _{ 2 } ^{ 2, 0 }
\).
\end{definition}

\begin{lemma}\label{lem:R-module}
For $q \ne 0$, the maps
$
    d_2^{0,q} \colon E_2^{0,q} \to E_2^{2, q-1}
$
in the spectral sequence \eqref{eq:E_2 spectral sequence of Hom} are $R$-linear.
\end{lemma}

\begin{proof}
For \( q \ne 0, 1 \), as already explained, the \(R\)-module structures on the source and the target of \( d _{ 2 } ^{ 0, q } \) are naturally defined and hence the \( R \)-linearity of \( d _{ 2 } ^{ 0, q } \) are rather obvious.

Let us show the \( R \)-linearity of \( d _{ 2 } ^{ 0, 1 } \).
Take an arbitrary element
\begin{align}
    (\phi_i)_i \in E_2^{0,1}
    =
    \bigoplus_i \Hom _{ \hirzebruchtwo } (\cH^i(\cE), \cH^{i+1}(\cE)).
\end{align}
It suffices to show that the maps as follows, which appear in the description of $d_2^{0,1}$ given in \pref{lm:E_2 spectral sequence of Hom}, are $R$-linear.
\begin{align}
\phi_{i_0-1} &\mapsto \phi_{i_0-1}\circ e^{i_0} \in \Ext _{ \hirzebruchtwo } ^2(\cH^{i_0}( \cE ), \cH^{i_0}( \cE )) \label{eq:circ_e} \\
\phi_{i_0} &\mapsto  e^{i_0+1}\circ \phi_{i_0} \in \Ext _{ \hirzebruchtwo } ^2(\cH^{i_0}( \cE ), \cH^{i_0}( \cE )) \label{eq:e_circ}
\end{align}
Under the isomorphism \( \iota ^{ \ast } \) of \eqref{eq:R-mod_str}, the map \eqref{eq:circ_e} is identified  with
\[
    \phi_{i_0-1}
    \mapsto
    \phi_{i_0-1} \circ e^{i_0} \circ \iota \in \Ext _{ \hirzebruchtwo } ^2(\cT, \cH^0(\cE)),
\]
and, similarly, the map \eqref{eq:e_circ} is identified with
\[
    \phi_{i_0}
    \mapsto
    e^{i_0+1}\circ \phi_{i_0} \circ \iota \in \Ext _{ \hirzebruchtwo } ^2(\cT, \cH^{i_0}(\cE)).
\]
These maps are $R$-linear, for the reason that the \(R\)-module structures given in \pref{rm:R-module structure} are by means of precompositions. As the \( R \)-module structure on
\(
    \Ext _{ \hirzebruchtwo } ^{ 2 } ( \cH ^{ i _{ 0 } } ( \cE ), \cH ^{ i _{ 0 } } ( \cE ) )
\)
is given in \pref{df:R-module structure on E _{ 2 } ^{ 2, 0 }} by transferring the \(R\)-module structure on
\(
    \Ext _{ \hirzebruchtwo } ^{ 2 } ( \cT, \cH ^{ i _{ 0 } } ( \cE ) )
\) via \( \iota ^{ \ast } \), this is exactly what we had to prove.
\end{proof}

Suppose $\cM$ is a pure coherent sheaf on $\hirzebruchtwo$ with
$
    \left(\Supp \cM\right) _{ \reduced }
    =
    C
$.
Then by \pref{rm:R-module structure},
\(
    M \coloneqq H ^{ 0 } ( \hirzebruchtwo, \cM )
    \simeq
    \Hom _{ \hirzebruchtwo } ( \cO _{ \hirzebruchtwo }, \cM )
\)
has a standard \(R\)-module structure of finite length and hence there is an integer \( \ell \) such that $\frak \frakm ^{ \ell } M = 0$.
In general, for an $R$-module $M$ and an ideal
\(
    I \subset R
\),
we let $(0:I)_M$ denote the \emph{annihilator}
\begin{align}\label{eq:annihilator}
    (0:I)_M \coloneqq \{ x \in M \mid Ix=0\} \subseteq M.
\end{align}
Geometrically speaking, this is the maximum submodule of \( M \) which is ``supported on the closed subscheme \( \Spec R / I \)''.
\begin{lemma}\label{lm:inequality}
For a pure coherent sheaf $\cM$ on $\hirzebruchtwo$ with
\(
    \left( \Supp\cM \right) _{ \reduced }
    =
    C
\), assume
\[
    n \coloneqq  \min \{ \ell \mid \frak \frakm ^{ \ell } M = 0\}>1,
\]
where \( M = H ^{ 0 } ( \hirzebruchtwo, \cM ) \in \module R \) as above.
Then the following \emph{strict} inequality holds.
\begin{align}\label{eq:the strict inequality}
    \dim_{ \bfk } \frac{ M }{ (0 : \frakm ^{ n - 1 } )_{ M } }
    <
    \dim_{ \bfk } \frakm ^{ n - 1 } M
\end{align}
\end{lemma}

\begin{proof}
Without loss of generality, we may assume that
\(
    f ^{ \ast } f _{ \ast } \cM \stackrel{\eta}{\to} \cM
\)
is surjective; i.e., \( \cM \) is \( f \)-globally generated. In fact, if \(\eta\) is not surjective, we can replace \( \cM \) with the image \( \Image \eta \). By standard arguments on the adjoint pair
\(
    f ^{ \ast } \dashv f _{ \ast }
\),
there is a canonical isomorphism \( f _{ \ast } \Image \eta \simto f _{ \ast } \cM = M \) and
\(
    \Image \eta
\)
is automatically \( f \)-globally generated.

The following fact about the \( A_1\)-singularity \( R \) is known well: for each \( i > 0\), the \(i\)-th powers of the ideal sheaf
\begin{align}\label{equation:cIC}
    \cI _{ C }
    =
    \cO _{ \hirzebruchtwo } ( - C )
    \subset
    \cO _{ \hirzebruchtwo }
\end{align}
of $C$ and \(\frakm\) are related to each other as follows.
\begin{align}\label{eq:powers of cO ( - C ) vs frakm}
    \begin{aligned}
        \frakm ^{ i } \cO _{ \hirzebruchtwo }
        =
        \cO _{ \hirzebruchtwo } ( - i C )
        =
        \cI _{ C } ^{ i } \subseteq \cO _{ \hirzebruchtwo }\\
        \left( f \vert _{ \Spec R } \right) _{ \ast } \cI _{ C } ^{ i }
        =
        \frakm ^{ i } \subseteq R
    \end{aligned}
\end{align}

The vanishing $\frakm^n M = 0$ means that $M = f _{ \ast } \cM$ is a sheaf on
$\Spec(R/\frakm^n)$, from which we deduce that $f^*f_*\cM$ and hence its quotient $\cM$ are supported
on $f^{-1}(\Spec(R/\frakm^n)) \stackrel{\eqref{eq:powers of cO ( - C ) vs frakm}}{=} nC$.
This means $\cI_C^n\cM=0$.

On the other hand, note that the image of the canonical morphism
\begin{align}
    H ^{ 0 } \left( \hirzebruchtwo, \cI _{ C } ^{ n - 1 } \right)
    \otimes _{ \bfk }
    H ^{ 0 } ( \hirzebruchtwo, \cM )
    \to
    H ^{ 0 } \left( \hirzebruchtwo, \cI _{ C } ^{ n - 1 } \cM \right),
\end{align}
as a submodule of \( H ^{ 0 } ( \hirzebruchtwo, \cM ) = M \), is \( \frakm ^{ n - 1 } M \) by \eqref{eq:powers of cO ( - C ) vs frakm}.
Thus we see
\(
    H ^{ 0 } \left( \hirzebruchtwo, \cI _{ C } ^{ n - 1 } \cM \right)
    \supset
    \frakm^{n-1} M \ne 0
\),
hence $\cI_C^{n-1} \cM \ne 0$. Therefore we conclude
\[
	n = \min \{ \ell \mid \cI_C ^{ \ell } \cM = 0\}.
\]

Since
\(
    \cI _{ C } ( \cI _{ C } ^{ n - 1 } \cM ) = 0
\),
we can naturally think of \( \cI _{ C } ^{ n - 1 } \cM \) as an object of \( \coh C \).
Since we assumed that \( \cM \) is pure, so is its subsheaf \( \cI _{ C } ^{ n - 1 } \cM \).
Hence \(  \cI _{ C } ^{ n - 1 } \cM \) is a vector bundle on \( C \) of positive rank.

Likewise
\(
    \frac{ \cM }{ (0:\cI_C^{n-1}) _{ \cM } }
\)
can be thought of an object of \( \coh C \), where
\(
    (0:\cI_C^{n-1}) _{ \cM }
\)
is the sheaf version of the annihilator \eqref{eq:annihilator}; i.e., the maximum subsheaf of \( \cM \) whose schematic support is contained in \( ( n - 1 ) C \).
We claim that there is an isomorphism
\begin{align}\label{eq:isom_of_vbs}
    \varphi 
    \colon
    \cO _{ C } ( - ( n - 1 ) C )
    \otimes_{\cO_C}
    \frac{ \cM }{ (0:\cI_C^{n-1}) _{ \cM } }
    \simto
    \cI_C^{n-1}\cM
\end{align}
of vector bundles on $C$.
Indeed \( \varphi \) is induced from the product morphism
\(
    \cI_C^{n-1} \otimes_{\cO _{ \hirzebruchtwo } } \cM \to \cI_C^{n-1}\cM
\),
whose surjectivity implies that of \(\varphi\).
At the generic point $\xi$ of $C$, the stalk $(\cI_C)_\xi$ is the maximal ideal of the
discrete valuation ring $\cO_{ \hirzebruchtwo, \xi}$ and $\cM_\xi$ is a finite length $\cO_{ \hirzebruchtwo, \xi}/(\cI_C^n)_\xi$-module.
Hence one sees that $\varphi$ is an isomorphism at $\xi$ (use the structure theorem for finitely generated modules over a discrete valuation ring).
Therefore it is enough to show that the left hand side of \eqref{eq:isom_of_vbs} is pure; i.e., there is no zero-dimensional subsheaf.

Assume for a contradiction that $\cM/(0:\cI_C^{n-1})_\cM$ is not pure.
Then there is a subsheaf $(0:\cI_C^{n-1})_\cM \subsetneq \cS \subset \cM$
such that $\Supp \left(\cS/(0:\cI_C^{n-1})_\cM\right)$ is zero-dimensional.
Then $\cI_C^{n-1}\cS \subset \cM$ is non-zero and zero-dimensional,
which contradicts the purity of $\cM$.
To see \( \dim \cI _{ C } ^{ n - 1 } \cS = 0 \), note that there is an epimorphism
\begin{align}
    \cI _{ C } ^{ n - 1 } \otimes _{ \cO _{ \hirzebruchtwo } } \cS / ( 0 : \cI _{ C } ^{ n - 1 } ) _{ \cM } \twoheadrightarrow \cI _{ C } ^{ n - 1 } \cS
\end{align}
and that
\(
    \dim \cS / ( 0 : \cI _{ C } ^{ n - 1 } ) _{ \cM } = 0
\).

Now consider the following \(R\)-module.
\[
    V
    \coloneqq
    \frac{ M }{ H ^{ 0 } \left( \hirzebruchtwo, (0:\cI_C^{n-1}) _{ \cM } \right) }
    =
    \frac{ M }{ (0:\frakm ^{ n - 1 } )_{ M } }
\]
To see the equality, note that the denominator of the left hand side, as a submodule of \( M = H ^{ 0 } ( \hirzebruchtwo, \cM ) \), coincides with the following subset.
\[
    \{ s \in M \mid s _{ P } \in ( 0 : \cI _{ C, P } ^{ n - 1 } ) _{ \cM _{ P } } \quad \forall P \in C \}.
\]
Since $ \cI _{ C, P } ^{ n - 1 } = \frakm^{n-1}\cO_{\hirzebruchtwo, P}$, this coincides with the denominator of the right hand side.

If we put
$L \coloneqq \cO_C(-(n-1)C)$ and $E \coloneqq \cM / (0:\cI_C^{n-1})_\cM$,
then \eqref{eq:isom_of_vbs} is rewritten as
\[
    L \otimes_{\cO_C} E \simeq \cI_C^{n-1}\cM.
\]
Noting $H^0( C, L ) \simeq \frakm^{n-1}/\frakm^n$,
we see that $\frakm ^{ n - 1 } M$ is the image of the product map
\[
    \psi \colon H^0( C, L ) \otimes_{ \bfk } V \to H^0( C, \cI_C^{n-1}\cM ).
\]
Then the assertion follows from \pref{lm:auxiliary inequality} below.
\end{proof}

\begin{lemma}\label{lm:auxiliary inequality}
Let $E$ and $L$ be a vector bundle and a line bundle on a smooth projective curve \( C \), respectively.
Suppose $\dim H^0( C, L ) >1$ and let $V \subset H^0( C, E )$ be a non-zero linear subspace.
Then we have a strict inequality
\[
    \dim V < \rank \psi,
\]
where
$
    \psi \colon H^0( C, L ) \otimes_{ \bfk } V \to H^0 ( C, L \otimes_{\cO_C} E )
$
denotes the product map.
\end{lemma}

\begin{proof}
Without loss of generality, we may replace \( E \) with the subsheaf generated by \( V \).
Take a pair of linearly independent global sections $s, t \in H^0( C, L )$. It follows that $\cO_C\cdot t \not\subseteq \cO_C \cdot s$ as subsheaves of $L$. Since \( E\) is locally free, this implies that $(\cO_C\cdot t) \otimes_{\cO_C} E \not\subseteq (\cO_C\cdot s) \otimes_{\cO_C} E$ as subsheaves of $L \otimes_{\cO_C} E $.
Since $(\cO_C\cdot t) \otimes_{\cO_C} E$ and $(\cO_C\cdot s) \otimes_{\cO_C} E$
are the subsheaves  of $L \otimes_{\cO_C} E$ generated by  $ \psi( \bfk t \otimes V)$ and $\psi(\bfk s \otimes V)$ respectively, we conclude
\(
    \psi( \bfk t \otimes V) \not\subseteq \psi(\bfk s \otimes V)
\).
Thus we see
\begin{align}
    V
    \simeq
    \psi( \bfk s \otimes V) \subsetneq \psi( \bfk s \otimes V)+\psi( \bfk t \otimes V)
    \subseteq
    \Image \psi.
\end{align}
Taking \( \dim _{ \bfk } \), we obtain the assertion.
\end{proof}

Let \( \module _{ 0 } R \subset \module R \) be the category of Artinian (\( \iff \dim _{ \bfk } < \infty \)) \(R\)-modules.
There is an (anti-)involution of categories
\[
    D \colon \left( \module _{ 0 } R \right) ^{ \op } \simto \module _{ 0 } R
\]
which is defined as follows.
\begin{align}\label{equation:k-dual}
    DM \coloneqq \Hom_{ \bfk }( M, \bfk )
\end{align}

\begin{corollary}\label{cr:dual inequality}
Under the same notation and assumption as in \pref{lm:inequality},
put $W \coloneqq D M$.
Then the following strict inequality holds.
\[
    \dim_{ \bfk } \frac{ W }{ (0:\frakm ^{ n - 1 } ) _{ W } } > \dim_{ \bfk } \frakm ^{ n - 1 } W
\]
\end{corollary}

\begin{proof}
    To obtain the assertion, apply the anti-involution \( D \) to both sides of the equality \eqref{eq:the strict inequality} (before applying \( \dim _{ \bfk } \)) and then use the following \pref{lm:properties of D}. Note that \( D \) preserves dimensions.
\end{proof}

\begin{lemma}\label{lm:properties of D}
For \( M \in \module _{ 0 } R \) and integer $\ell \ge 0$,
there are isomorphisms as follows.
\begin{alignat}{2}
& D( \frakm ^{ \ell }  M) && \simeq DM/(0: \frakm ^{ \ell } )_{DM}\\
&D\left(M/(0: \frakm ^{ \ell } )_M\right) &&\simeq  \frakm ^{ \ell }  DM
\end{alignat}
\end{lemma}

\begin{proof}
    To obtain the second isomorphism, replace \( M \) with \( D M \) in the first isomorphism and then use \( D ^{ 2 } \simeq \id \). To see the first isomorphism, consider the following short exact sequence. The inclusion \(i\) is the canonical one.
    \begin{align}
        0 \to ( 0 : \frakm ^{ \ell } ) _{ D M } \stackrel{i}{\to} D M \to C \coloneqq \coker i \to 0
    \end{align}
    By applying the anti-involution \( D \) to this, we obtain the following short exact sequence.
    \begin{align}
        0 \to D C \to M \to D ( 0 : \frakm ^{ \ell } ) _{ D M } \to 0
    \end{align}
    Then, as a submodule of \( M \), \( D C \) is computed as follows.
    \begin{align}
        D C
        =
        \{ x \in M \mid y \in ( 0 : \frakm ^{ \ell } ) _{ D M } \Rightarrow i ( y ) ( x ) = 0 \}
        =
        \{ x \in M \mid y \in D M, \frakm ^{ \ell } y = 0 \Rightarrow y ( x ) = 0 \}.\\
    \end{align}
    Note that
    \(
        y \in D M
    \)
    satisfies \(\frakm ^{ \ell } y = 0\) if and only if
    \(
        y ( \frakm ^{ \ell } M ) = 0
    \); i.e.,
    \(
        y \in D ( M / \frakm ^{ \ell } M ) \stackrel{ q ^{ \ast } }{ \hookrightarrow} D M
    \), where \( q \colon M \twoheadrightarrow M / \frakm ^{ \ell } M \) is the quotient map. Hence
    \begin{align}
        D C = \{ x \in M \mid y \in D ( M / \frakm ^{ \ell } M ) \Rightarrow y ( x ) = 0 \}
        =
        D ( M / \frakm ^{ \ell } M ) ^{ \perp}
        =
        \frakm ^{ \ell } M,
    \end{align}
    so that \( C \simeq D ( \frakm ^{ \ell } M ) \).
\end{proof}

Now we are ready to prove the following
\begin{proposition}\label{pr:chohomology sheaves are O_C modules}
For an exceptional object
\(
    \cE \in \derived ( \hirzebruchtwo )
\),
it holds that
\[
    \Supp \tors \cH^{i_0}(\cE)
    =
    \Supp \cH ^{ i } ( \cE ) \ ( i \ne i _{ 0 } )
    =
    C \text{ or } \emptyset.
\]
\end{proposition}

\begin{proof}
We first discuss $\Supp \cH ^{ i } ( \cE )$ for \( i \ne i _{ 0 } \).
We may and will assume $\bigoplus_{i \ne i_0} \cH^i(\cE) \ne 0$, since otherwise there is nothing to prove. Put
\begin{align}
    n \coloneqq \min \bigcap _{ i \ne i _{ 0 } } \left\{ \ell \mid \cI _{ C } ^{ \ell } \cH ^{ i } ( \cE ) = 0 \right\} \ge 1.
\end{align}
In the spectral sequence \eqref{eq:E_2 spectral sequence of Hom} for the exceptional object \( \cE \),
we have the isomorphism
\begin{align}\label{eq:E_2^{0,1}=E_2^{2,0}}
    d _{ 2 } ^{ 0, 1 } \colon
    E_2^{0,1}=\bigoplus_i \Hom _{ \hirzebruchtwo } (\cH^i(\cE), \cH^{i+1}(\cE)) \simto
    E_2^{2,0}=\bigoplus_i \Ext _{ \hirzebruchtwo } ^2(\cH^i(\cE), \cH^i(\cE))
\end{align}
of \emph{$R$-modules} by \pref{lem:R-module}. Note that $\frakm^n E_2^{0,1}=0$,
since for each \( i \) either
\(
    \cI _{ C } ^{ n } \cH ^{ i } ( \cE ) = 0
\)
or
\(
    \cI _{ C } ^{ n } \cH ^{ i + 1 } ( \cE ) = 0
\)
holds.

On the other hand, for each $i \ne i_0$ there is an isomorphism of \( R \)-modules given by the Serre duality:
\[
    \Ext^2 _{ \hirzebruchtwo } (\cH^i(\cE), \cH^i(\cE)) \simeq D\Hom(\cH^i(\cE), \cH^i(\cE))
\]
This implies that $\frakm ^{ n - 1 } E_2^{ 2, 0 } \ne 0$.
Combining this with the isomorphism of \( R \)-modules \eqref{eq:E_2^{0,1}=E_2^{2,0}}, we obtain
\[
    n
    =
    \min \{ \ell \mid  \frakm ^{ \ell }  E_2^{0,1}=0 \}
    =
    \min \{ \ell \mid  \frakm ^{ \ell }  E_2^{2,0}=0 \}.
\]

Now assume for a contradiction that $n>1$.
Let
\begin{align}
    \cM \coloneqq \bigoplus_i \cHom _{ \hirzebruchtwo } (\cH^i(\cE), \cH^{i+1}(\cE)),
\end{align}
so that
\(
    E_2^{0,1} \simeq f _{ \ast } \cM
\).
We can apply \pref{lm:inequality} to \( \cM \), to obtain the strict inequality
\begin{align}\label{eq:<}
    \dim_{ \bfk } \frac{ E _{ 2 } ^{ 0 , 1 } }{ (0 : \frakm ^{ n - 1 } )_{ E _{ 2 } ^{ 0 , 1 } } }
    <
    \dim_{ \bfk } \frakm ^{ n - 1 } E _{ 2 } ^{ 0 , 1 }.
\end{align}

On the other hand, let
\begin{align}
    \cM ' \coloneqq
    \bigoplus _{ i \ne i _{ 0 } } \cEnd _{ \hirzebruchtwo } ( \cH ^{ i } ( \cE ) ) \oplus \cHom _{ \hirzebruchtwo } ( \cH ^{ i _{ 0 } } ( \cE ), \cT ).
\end{align}
By \pref{lem:R-module} and the Serre duality, it follows that there is an isomorphism of \( R \)-modules
\(
    E _{ 2 } ^{ 2, 0 } \simeq D f _{ \ast } \cM '
\).
We can apply \pref{cr:dual inequality} to \( \cM '\), to obtain the strict inequality
\begin{align}\label{eq:>}
    \dim_{ \bfk } \frac{ E _{ 2 } ^{ 2, 0 } }{ (0:\frakm ^{ n - 1 } ) _{ E _{ 2 } ^{ 2, 0 } } }
    >
    \dim_{ \bfk } \frakm ^{ n - 1 } E _{ 2 } ^{ 2, 0 }.
\end{align}

The strict inequalities \eqref{eq:<} and \eqref{eq:>} contradict the isomorphism \eqref{eq:E_2^{0,1}=E_2^{2,0}}. Hence we obtain $n=1$, which means that $\cH^i(\cE)$ is an $\cO_C$-module for any $i \ne i_0$. In fact, by rigidity, \( \cH ^{ i } ( \cE ) \) is a vector bundle on \( C \) for any \( i \ne i _{ 0 } \).

Finally, to investigate $\tors \cH^{i_0}(\cE)$, we consider the derived dual
$
    \cE^\vee
$.
As we show in \pref{lm:dual exceptional object} below, there is an isomorphism as follows.
\[
    \cH^{-i_0+1}(\cE^\vee) \simeq (\tors\cH^{i_0}(\cE))^\vee[1]
\]
Since $-i_0+1 \ne - i _{ 0 } \stackrel{\text{\pref{lm:dual exceptional object}}}{=} i_0(\cE^\vee)$,
by applying what we have just proved to the exceptional object \( \cE ^{ \vee }\), we see that the left hand side, hence the right hand side, is a vector bundle on \( C \). Hence so is \( \tors \cH ^{ i _{ 0 } } ( \cE ) \).
\end{proof}

\begin{remark}
    \pref{pr:chohomology sheaves are O_C modules} will \emph{not} be used in the proof of \pref{lm:dual exceptional object}, so that it is harmless to use \pref{lm:dual exceptional object} in the proof of \pref{pr:chohomology sheaves are O_C modules} (as we did in the last paragraph). Actually, in the proof of \pref{lm:dual exceptional object}, we only use \pref{lm:properties of cohomology sheaves} and some standard facts on homological algebra.
\end{remark}

%
%
\subsection{More on the structure of \(\cohomology ( \cE )\)}
\label{sc:More on the structure}

In this subsection we give a structure theorem for \( \cH ^{ i _{ 0 } } ( \cE ) \) in \pref{lm:exceptional sheaf is a direct summand}. It is then used to give a structure theorem for \( \tors \cohomology ( \cE )\) in \pref{cr:structure of the residual part}.

Below is repeatedly used in this paper.

\begin{lemma}[{\( = \)\cite[Remark 2.3.4]{MR1604186}}]\label{lm:exceptional vector bundle restricted to C}
Let
\(
    \cE \in \derived ( \hirzebruchtwo )
\)
be an exceptional vector bundle of
\(
    \rank \cE = r
\).
Then there is an isomorphism
\begin{align}\label{eq:restriction of exceptional vector bundle to C}
    \cE | _{ C }
    \simeq
    \cO _{ C } ( b ) ^{ \oplus s }
    \oplus
    \cO _{ C } ( b + 1 ) ^{ \oplus r - s }
\end{align}
for some \( b \in \bZ \) and \( s \in \bN \) such that \( 1 \le s \le r \).
\end{lemma}
Note that the integers \(b\) and \(s\) in \pref{lm:exceptional vector bundle restricted to C} are uniquely determined by \(\cE\).

\begin{lemma}\label{lm:exceptional sheaf is a direct summand}
Let \( \cE \in \derived ( \hirzebruchtwo ) \) be an exceptional object.
Then the unique non-torsion cohomology sheaf $\cH^{i_0}(\cE)$ decomposes as
\[
    \cH^{i_0}(\cE) \simeq E \oplus T,
\]
where $E$ is an exceptional sheaf and $T$ is a vector bundle on \( C \).
Moreover, if $E$ is not locally free, then there is an integer $a \in \bZ$ such that
\begin{itemize}
\item The torsion part $\tors E$ is a direct sum of copies of $\cO_C(a)$.
\item $T$ is a direct sum of copies of $\cO_C(a)$ and $\cO_C(a+1)$.
\end{itemize}
\end{lemma}

\begin{definition}\label{df:E and cF}
For an exceptional object
\(
    \cE \in \derived ( \hirzebruchtwo )
\),
let
\( 
    E = E ( \cE )
\)
denote the exceptional sheaf \( E \) of \pref{lm:exceptional sheaf is a direct summand}. Also let
\(
    \cF = \cF ( \cE )
\)
denote the torsion free part of the sheaf
\(
    \cH ^{ i _{ 0 } } ( \cE )
\),
which is known to be an exceptional vector bundle by \pref{lm:properties of cohomology sheaves} \pref{it:more on i0-th cohomology}.
\end{definition}

\begin{proof}[Proof of \pref{lm:exceptional sheaf is a direct summand}]
Consider the following standard short exact sequence.
\begin{align}\label{eq:cT cH io cE cF}
    0 \to \cT \coloneqq \tors \cH^{i_0}(\cE) \to \cH^{i_0}(\cE) \to \cF \coloneqq \cH^{i_0}(\cE)/\cT \to 0
\end{align}
\pref{lm:properties of cohomology sheaves} \eqref{it:more on i0-th cohomology} asserts that \( \cF \) is an exceptional vector bundle. Hence we assume that \( \cT \ne 0 \), since otherwise there is nothing to prove. 

$\cT$ is a rigid sheaf again by \pref{lm:properties of cohomology sheaves} \eqref{it:more on i0-th cohomology}. Combined with \pref{pr:chohomology sheaves are O_C modules}, this implies that there are $a \in \bZ, s > 0, t \ge 0 $ such that
\begin{align}\label{eq:decomposition of cT}
    \cT \simeq \cO_C(a)^{\oplus s} \oplus \cO_C(a+1)^{\oplus t}.
\end{align}
By \pref{lm:exceptional vector bundle restricted to C}, $\cF|_C$ is also rigid and hence there are \( b \in \bZ, s ' > 0, t ' \ge 0 \) such that
\[
    \cF|_C \simeq \cO_C(b)^{\oplus s'} \oplus \cO_C(b+1)^{\oplus t'}.
\]
If $\Ext _{ \hirzebruchtwo } ^1(\cF, \cT)=0$, then the short exact sequence \eqref{eq:cT cH io cE cF} splits and we are done.
Therefore we assume $\Ext _{ \hirzebruchtwo } ^1(\cF, \cT)\ne 0$, which implies
\begin{align}\label{eq:a le b - 1 }
    a \le b - 1
\end{align}

In order to make the argument conceptual, fix a \( \bfk \)-vector space \( V \) of dimension \( s \) and replace \( \cO _{ C } ( a ) ^{ \oplus s } \) with \( V \otimes _{ \bfk } \cO _{ C } ( a ) \). Let
\begin{align}\label{eq:definition of cH'}
\cH' \coloneqq \cH^{i_0}(\cE)/\cO_C(a+1)^{\oplus t}
\end{align}
be the quotient by the subsheaf
$\cO_C(a+1)^{\oplus t} \subset \cT \subset \cH^{i_0}(\cE)$, which fits in the following short exact sequence.
\begin{align}\label{eq:cH'-extension}
    0\to V \otimes _{ \bfk } \cO_C(a) \to \cH' \to \cF \to 0
\end{align}
From this we see
\[
    \Hom _{ \hirzebruchtwo } (\cO(a+1)^{\oplus t}, \cH')=0,
\]
which implies that $\cH'$ is rigid by Mukai's lemma \cite[Lemma~2.1.4.~2.(a)]{MR1604186}, (which is obtained from the spectral sequence of the form \eqref{eq:E_1 spectral sequence of Hom}).

Let
\begin{align}
    [ f \colon \Ext _{ \hirzebruchtwo } ^{ 1 } ( \cF, \cO _{ C } ( a ) ) ^{ \vee } \to V ]
    \in
    \Hom _{ \bfk } ( \Ext _{ \hirzebruchtwo } ^{ 1 } ( \cF, \cO _{ C } ( a ) ) ^{ \vee }, V )
    \simeq
    \Ext _{ \hirzebruchtwo } ^{ 1 } ( \cF, V \otimes _{ \bfk } \cO _{ C } ( a ) )
\end{align}
correspond to the extension \eqref{eq:cH'-extension}. We will show that \( f \) is injective. 
In the long exact sequence obtained by applying $\Hom_{ \hirzebruchtwo }(\cH', -)$ to \eqref{eq:cH'-extension},
\begin{itemize}
\item
the map $\Hom_{ \hirzebruchtwo }(\cH', \cH') \to \Hom_{ \hirzebruchtwo }(\cH', \cF)$ is surjective since $\Hom_{ \hirzebruchtwo }(\cH', \cF) \cong \Hom_{ \hirzebruchtwo }(\cF, \cF) = \bfk \cdot \id_{\cF}$ by the exceptionality of $\cF$ and
\item
$\Ext^1(\cH', \cH')=0$ by the rigidity of $\cH'$.
\end{itemize}
Hence we obtain
\[
    \Ext^1_{ \hirzebruchtwo }(\cH', \cO_C(a))=0.
\]
Next we apply $\Hom_{ \hirzebruchtwo }(-, \cO_C(a))$ to \eqref{eq:cH'-extension} to obtain a surjective map
\begin{equation}\label{eq:fvee}
    V ^{ \vee } \cong \Hom_{ \hirzebruchtwo }(V \otimes \cO_C(a), \cO_C(a)) \to \Ext^1_{ \hirzebruchtwo }(\cF, \cO_C(a)).
\end{equation}
For any $\varphi \in V^\vee$, the map \eqref{eq:fvee} sends $\varphi$
to $(\varphi \otimes \id_{\cO_C(a)})\circ e$,
where $e \in \Ext _{ \hirzebruchtwo } ^{ 1 } ( \cF, V \otimes _{ \bfk } \cO _{ C } ( a ) )$
denotes the extension class \eqref{eq:cH'-extension}.
Take a basis $\{v_1, \dots, v_s\}$ of $V$ and
decompose \( e \) as
\(
    \sum _{ i } e _{ i } \otimes v _{ i }
\)
under the isomorphism
\(
    \Ext _{ \hirzebruchtwo } ^{ 1 } ( \cF, V \otimes _{ \bfk } \cO _{ C } ( a ) )
    \simeq
    \Ext _{ \hirzebruchtwo } ^{ 1 } \left( \cF, \cO _{ C } ( a ) \right) \otimes _{ \bfk } V
\).
Then one can confirm that \eqref{eq:fvee} sends \( \varphi \) to
\(
    \sum _{ i } \varphi ( v _{ i } ) e _{ i }
\)
and that \( f \) sends a linear form
\(
    \xi \in \Ext _{ \hirzebruchtwo } ^{ 1 } ( \cF, \cO _{ C } ( a ) ) ^{ \vee }
\)
to
\(
    \sum _{ i } \xi ( e _{ i } ) v _{ i }
\).
Hence the surjectivity of the map \eqref{eq:fvee} implies that \( e _{ 1 }, \dots, e _{ s } \) generates
\(
    \Ext _{ \hirzebruchtwo } ^{ 1 } \left( \cF, \cO _{ C } ( a ) \right)
\), which in turn is equivalent to the injectivity of \( f \).

Now consider the universal extension of \( \cF \) by \( \cO _{ C } ( a ) \).
\begin{align}\label{eq:universal extension}
    0
    \to
    \Ext _{ \hirzebruchtwo } ^{ 1 } ( \cF, \cO _{ C } ( a ) ) ^{ \vee } \otimes _{ \bfk } \cO _{ C } ( a )
    \to
    E
    \to
    \cF
    \to
    0
\end{align}
The inequality \eqref{eq:a le b - 1 } implies
\(
    \RHom _{ \hirzebruchtwo } ( \cF, \cO _{ C } ( a ) )
    \simeq
    \Ext _{ \hirzebruchtwo } ^{ 1 } ( \cF, \cO _{ C } ( a ) ) [ - 1 ]
\), so that the distinguished triangle which \eqref{eq:universal extension} yields is isomorphic to the defining distinguished triangle \eqref{eq:triangle of inverse spherical twist} for the inverse spherical twist \( T ' _{ a } ( \cF ) \). In particular, there is an isomorphism
\(
    E \simeq T ' _{ a } ( \cF )
\)
and hence \( E \) is an exceptional sheaf. Moreover, the injectivity of \( f \) and the basic properties of universal extensions imply that there is an isomorphism
\begin{align}
    \cH ' \simeq E \oplus \left( \coker f \otimes _{ \bfk } \cO _{ C } ( a ) \right).
\end{align}
In what follows let
\(
    d \coloneqq \dim _{ \bfk } \Ext _{ \hirzebruchtwo } ^{ 1 } ( \cF, \cO _{ C } ( a ) )
\),
so that
\[
    \cH' \simeq E \oplus \cO_C(a)^{\oplus s-d}.
\]
If \( d = 0 \), then \( E \simeq \cF \) is a vector bundle and we are done. So, in the rest of the proof, we assume \( d > 0 \); i.e., we assume that \( \tors E \ne 0 \).

Let us prove
\begin{align}\label{eq:Ext(cH', cOC(a+1))=0}
    \Ext _{ \hirzebruchtwo } ^1(\cH', \cO_C(a+1))=0,
\end{align}
which together with \eqref{eq:definition of cH'} implies
\begin{align}\label{eq:desired decomposition}
    \cH^{i_0}(\cE) \simeq \cH' \oplus \cO_C(a+1)^{\oplus t} \simeq E \oplus \cO_C(a)^{\oplus s-d}  \oplus \cO_C(a+1)^{\oplus t}.
\end{align}
Since \( \tors E \simeq \cO _{ C } ( a ) ^{ \oplus d } \), this is the desired conclusion.

\eqref{eq:Ext(cH', cOC(a+1))=0} follows from the local-to-global spectral sequence and the following vanishings.
\begin{alignat}{2}
&H^1( \hirzebruchtwo, \cHom _{ \hirzebruchtwo } (\cH', \cO_C(a+1)))&&=0 \label{eq:vanishing of H^1}\\
&H^0( \hirzebruchtwo, \cExt _{ \hirzebruchtwo } ^1(\cH', \cO_C(a+1)))&&=0 \label{eq:vanishing of H^0}
\end{alignat}

\eqref{eq:vanishing of H^1}, in turn, follows from
$\cH'=E \oplus \cO_C(a)^{\oplus s-d}$
and \cite[Theorem 1.4(1)]{MR3431636},
which says that an exceptional sheaf $E$ whose torsion part is a non-zero direct sum of copies of $\cO_C(a)$ satisfies
\begin{align}\label{eq:E|C = OC(a+1) oplus ell}
    E|_C \simeq \cO_C(a+1)^{\oplus r }
\end{align}
for some \( r \).

Finally, \eqref{eq:vanishing of H^0} follows from the following isomorphisms.
\[
    \cExt _{ \hirzebruchtwo } ^1(\cH', \cO_C(a+1)) \stackrel{\eqref{eq:cH'-extension}}{\simeq}
    \cExt _{ \hirzebruchtwo } ^1( V \otimes _{ \bfk } \cO_C(a), \cO_C(a+1)) \simeq V ^{ \vee } \otimes \cO_C(-1)
\]
\end{proof}

From \pref{lm:exceptional sheaf is a direct summand} we immediately obtain
\begin{corollary}\label{cr:structure of the residual part}
Suppose that \( \cE \) is an exceptional object with
\(
    \tors E \ne 0
\).
Then, with the notation of \pref{lm:exceptional sheaf is a direct summand}, it holds that
\begin{align}
    \tors \cohomology ( \cE )
    =
    \bigoplus _{ i \neq i _{ 0 } } \cH ^{ i } ( \cE )
    \oplus
    T
    \oplus
    \tors E
    \simeq
    \cO _{ C } ( a ) ^{ \oplus s }
    \oplus
    \cO _{ C } ( a + 1 ) ^{ \oplus \ell ( \cE ) - s }
\end{align}
for some
\(
    1 \le s \le \ell ( \cE )
\).
\end{corollary}

\begin{remark}
This is analogous to \cite[Corollary 4.10]{Ishii-Uehara_ADC} (see also \cite[Section 5]{Ishii-Uehara_ADC})
\end{remark}

\begin{proof}
If $i \ne i _{ 0 }$, then \( \Supp \cH ^{ i } ( \cE ) = C \) and hence
\( \cH ^{ i } ( \cE ) \otimes \omega _{ \hirzebruchtwo } \simeq \cH ^{ i } ( \cE ) \). Thus we see
\[
    \Ext _{ \hirzebruchtwo } ^1(\cH^{i _{ 0 }}(\cE), \cH^i(\cE))
    \stackrel{\text{Serre duality}}{\simeq}
    \Ext _{ \hirzebruchtwo } ^1(\cH^{i}(\cE), \cH^{i _{ 0 }}(\cE)) ^{ \vee }
    \stackrel{\text{\pref{lm:properties of cohomology sheaves} \eqref{it:cohomology(E) is rigid}}}{=}
    0.
\]
Since $E$ is a direct summand of $\cH^{i _{ 0 }}(\cE)$, this implies
\(
    \Ext _{ \hirzebruchtwo } ^1(E, \cH^i(\cE)) = 0, \Ext _{ \hirzebruchtwo } ^1( \cH^i(\cE), E ) = 0
\).
Moreover, since $\dim \Supp \cH^i(\cE)=1$,
the local to global spectral sequence for Ext groups implies
\begin{align}\label{eq:vanishing of two H^1(cHom)}
    \begin{aligned}
    H^1(\hirzebruchtwo, \cHom _{ \hirzebruchtwo } (E, \cH^i(\cE))) = 0,\\
    H^1(\hirzebruchtwo, \cHom _{ \hirzebruchtwo } (\cH^i(\cE), E))=0.
\end{aligned}
\end{align}

On the other hand, by \pref{pr:chohomology sheaves are O_C modules} and the rigidity,
there is a vector bundle \(\cV\) on $C$ such that $\cH^i(\cE) \simeq \iota _{ \ast } \cV $. Hence there are isomorphisms as follows.
\begin{align}
    \begin{aligned}
        \cHom _{ \hirzebruchtwo } (E, \cH^i(\cE)) \simeq \iota _{ \ast } \cHom _{ C } ( E \vert _{ C }, \cV )
        \stackrel{\eqref{eq:E|C = OC(a+1) oplus ell}}{\simeq}
        \iota _{ \ast } \cHom _{ C } ( \cO _{ C } ( a + 1 ), \cV ) ^{ \oplus r }\\
        \cHom _{ \hirzebruchtwo } ( \cH^i(\cE), E ) \stackrel{\eqref{eq:universal extension}}{\simeq} \cHom _{ \hirzebruchtwo } ( \cH ^{ i } ( \cE ), \cO _{ C } ( a ) ) ^{ \oplus d }
        \simeq
        \iota _{ \ast } \cHom _{ C } ( \cV, \cO _{ C } ( a ) ) ^{ \oplus d }
    \end{aligned}
\end{align}
Combining these isomorphisms with \eqref{eq:vanishing of two H^1(cHom)},
we obtain the following vanishings.
\begin{align}
    \begin{aligned}
        H^1(C, \cHom_C(\cO_C(a+1), \cV) = 0\\
        H^1(C, \cHom_C(\cV, \cO_C(a)) = 0
    \end{aligned}
\end{align}
From this we deduce that $\cV$ is of the form
\[
    \cV \simeq \cO_C(a)^{\oplus s_i} \oplus \cO_C(a+1)^{\oplus t_i}
\]
for some $s_i$ and $t_i$, concluding the proof.
\end{proof}

%
%
\subsection{Derived dual of exceptional objects}
\label{sc:Derived dual of exceptional objects}

Let $\cE$ be an exceptional object which is not isomorphic to a shift of a vector bundle, and let \(E = E ( \cE )\) be the exceptional sheaf in \pref{df:E and cF}.
In what follows we will mainly discuss the case where \( \tors E \ne 0 \). If \( \tors E = 0 \) (\(\iff\) \(E\) is a vector bundle), we will replace \( \cE \) with its derived dual \( \cE ^{ \vee } \) and reduce the problem to the main case. What we mean by this will be made precise by \pref{cr:E(cE vee) has non-trivial torsion}.

\begin{lemma}\label{lm:dual exceptional object}
For an exceptional object $\cE$ on $\hirzebruchtwo$, the cohomology sheaves of the derived dual
\(
    \cE ^{ \vee }
\)
are related to those of $\cE$ as follows:
\begin{itemize}
    \item $i_0(\cE^\vee)=-i_0$.
    
    \item If $i \ne -i_0$, then
    $
        \cH^i(\cE^\vee)
        \simeq
        \cExt _{ \hirzebruchtwo } ^1(\cH^{-i+1}(\cE), \cO _{ \hirzebruchtwo } )
    $.
    
    \item For $i=-i_0$, $\cH^{-i_0}(\cE^\vee)$ fits into an exact sequence
    \[
        0
        \to
        \cExt _{ \hirzebruchtwo } ^1(\cH^{i_0+1}(\cE), \cO _{ \hirzebruchtwo } )
        \to
        \cH^{-i_0}(\cE^\vee)
        \to
        \cHom _{ \hirzebruchtwo } (\cH^{i_0}(\cE), \cO _{ \hirzebruchtwo } )
        \to
        0.
    \]
    
    \item For $i=-i_0+1$, the cohomology sheaf can also be written as
    \[
        \cH^{-i_0+1}(\cE^\vee)
        \simeq
        \cExt _{ \hirzebruchtwo } ^1(\tors \cH^{i_0}(\cE), \cO _{ \hirzebruchtwo } )
        \simeq
        \left( \tors \cH ^{ i _{ 0 } } ( \cE ) \right) ^{ \vee } [ 1 ].
    \]
\end{itemize}
\end{lemma}

\begin{proof}
Consider the following spectral sequence.
\[
    E_2^{p,q}
    =
    \cExt _{ \hirzebruchtwo } ^p(\cH^{-q}(\cE), \cO _{ \hirzebruchtwo } ) \Rightarrow \cH^{p+q}(\cE^\vee)
\]
Since \( \hirzebruchtwo \) is a smooth projective surface,
\(
    E _{ 2 } ^{ p, q } = 0
\)
if \( p < 0 \) or \( p > 2 \).
Since \( \cH ^{ i } ( \cE ) \) is torsion for \( i \ne i _{ 0 } \) by \pref{lm:properties of cohomology sheaves} \pref{it:support},
\(
    E _{ 2 } ^{ 0, q } = 0
\)
for \( q \ne - i _{ 0 } \).
Moreover, for any \( i \ne i _{ 0 } \), \( \cH ^{ i } ( \cE ) \) is pure by \pref{lm:properties of cohomology sheaves} \pref{it:support}. Furthermore,
\(
    \tors \cH ^{ i _{ 0 } } ( \cE )
\)
is pure again by \pref{lm:properties of cohomology sheaves} \pref{it:support} and 
\(
    \cH ^{ i _{ 0 } } ( \cE ) / \tors \cH ^{ i _{ 0 } } ( \cE )
\)
is locally free by \pref{lm:properties of cohomology sheaves} \pref{it:more on i0-th cohomology}. Hence by \cite[Theorem~1.1.10,~\(1) \Rightarrow 2)\)]{Huybrechts-Lehn},
\(
    E _{ 2 } ^{ 2, q } = 0
\)
for any \( q \in \bZ \).

Summing up, we see that
\(
    E _{ 2 } ^{ p, q } \ne 0
\)
only if
\(
    ( p, q ) = ( 0, - i _{ 0 } )
\) or
\(
    p = 1
\).
In particular, this spectral sequence is \(E _{ 2 } \)-degenerate. All assertions follow from these observations.
\end{proof}

\begin{lemma}\label{lm:a sufficient condition for cE to be a vector bundle}
Let $\cE$ be an exceptional object such that both $E = E ( \cE )$ and \( E ( \cE ^{\vee} ) \) are vector bundles.
Then $\cE \simeq E[-i_0]$.
\end{lemma}
\begin{proof}
We consider the spectral sequence
\[
    E_2^{p,q} = \Ext _{ \hirzebruchtwo } ^p(E, \cH^q(\cE)) \Rightarrow \Hom^{p+q}(E, \cE).
\]
As \( \hirzebruchtwo \) is a smooth surface, \( E _{ 2 } ^{ p, q } \ne 0 \) only if \( 0 \le p \le 2 \). In particular, it is \(E _{ 3 }\)-degenerate. Since $E_2^{2, i_0-1}\simeq \Hom _{ \hirzebruchtwo } (\cH^{i_0-1}(\cE), E)^\vee=0$
for \(\cH^{i_0-1}(\cE)\) being torsion and $E$ being torsion free, it follows that \( d _{ 2 } ^{ 0, i _{ 0 } } = 0 \) and hence
\begin{align}\label{eq:E2 degeneration at ( 0, i0 )}
    E _{ 2 } ^{ 0, i _{ 0 } }
    =
    \ker d _{ 2 } ^{ 0, i _{ 0 } }
    \simeq
    E _{ 3 } ^{ 0, i _{ 0 } }
    \simeq
    E _{ \infty } ^{ 0, i _{ 0 } }.
\end{align}
Take any
\[
    [\varphi\colon E[-i_0] \to \cE] \in \Ext ^{ i _{ 0 }} ( E, \cE )
\]
whose image under the surjection
\(
    \Ext ^{ i _{ 0 }} ( E, \cE ) \twoheadrightarrow E _{ \infty } ^{ 0, i _{ 0 } }
\)
of the spectral sequence, which in fact is $\cH^{i_0}(\varphi)$, corresponds to the natural inclusion $E \hookrightarrow \cH^{i_0}(\cE)$ under the isomorphisms \eqref{eq:E2 degeneration at ( 0, i0 )}.
Then the derived dual
\[
    \varphi^\vee\colon \cE^\vee \to E^\vee[i_0]
\]
induces the surjection
\[
    \cH^{-i_0}(\cE^\vee) \to \cHom _{ \hirzebruchtwo } ( \cH ^{ i _{ 0 } } ( \cE ), \cO _{ \hirzebruchtwo } ) \simeq E^\vee
\]
in \pref{lm:dual exceptional object}.
As we assumed that $E ( \cE ^{ \vee } ) \simeq E^\vee$ is torsion free, by applying what we have just confirmed to \( \cE ^{  \vee }\), we similarly obtain a morphism
\[
    \psi \colon E^\vee [i_0] \to \cE^\vee
\]
such that \( \cH ^{ - i _{ 0 } } ( \psi ) \) is the inclusion $E^\vee \simeq E ( \cE ^{ \vee } ) \hookrightarrow \cH^{-i_0}(\cE^\vee)$ as a direct summand.
Then it follows that the composite $\varphi^\vee \circ \psi$ is an automorphism of $E^\vee[i_0]$
and therefore $\cE^\vee$ splits as a direct sum of $E^\vee[i_0]$ and $\Cone \psi$.
Since the exceptional object $\cE^\vee$ is indecomposable, this implies $\Cone \psi = 0$ and hence
$\cE^\vee \simeq E^\vee[i_0]$, which is equivalent to $\cE \simeq E[-i_0]$.
\end{proof}

Thus we immediately obtain the following
\begin{corollary}\label{cr:E(cE vee) has non-trivial torsion}
    Let \( \cE \) be an exceptional object which is not isomorphic to a shift of a vector bundle. If \( E = E ( \cE ) \) is torsion free, then
    \( \tors E ( \cE ^{ \vee } ) \neq 0 \).
\end{corollary}

\subsection{Length of the torsion part}
\label{sc:Length of the torsion part}

Now we introduce the notion of length, which measures for an exceptional object the distance from a shift of a vector bundle. The proof of \pref{th:exceptional objects are equivalent to vector bundles} is reduced to the assertion \pref{th:decrease the length by appropriate twist} that one can always reduce the length by an appropriate spherical twist.

\begin{definition}\label{df:length}
Let $\gamma$ be the generic point of $C$ and $\cO_{\hirzebruchtwo, \gamma}$ the local ring of $\hirzebruchtwo$ at $\gamma$.
For a coherent sheaf $\cH$, put
\[
    \ell (\cH)=\length_{\cO_{\hirzebruchtwo, \gamma}} \tors \cH_\gamma
\]
where $\cH_\gamma$ is the stalk of $\cH$ at $\gamma$ and $\length_{\cO_\gamma} \tors\cH_\gamma$
is the length of its torsion part.

For an object $\cE \in D(\hirzebruchtwo)$, we define
\[
    \ell (\cE) \coloneqq  \sum_i \ell (\cH^i(\cE)).
\]
\end{definition}

Let us give a bit more concrete description for \(\ell (\cE)\). Let
\(
    \cE _{ \gamma } \in \derived ( \cO _{ \hirzebruchtwo, \gamma } )
\)
be the pull-back of \( \cE \) by the flat morphism \( \Spec \cO _{ \hirzebruchtwo, \gamma } \to \hirzebruchtwo \).
The length function \(\ell\) is defined for objects in \(\derived ( \cO _{ \hirzebruchtwo, \gamma } )\) as well, and it immediately follows from the exactness of the (underived) pull-back functor that
\(
    \ell ( \cE _{ \gamma } ) = \ell ( \cE )
\).
On the other hand, since \(\cO _{ \hirzebruchtwo, \gamma }\) is a DVR, for \( \cM \in \derived ( \cO _{ \hirzebruchtwo, \gamma }) \) there is an isomorphism
\(
    \cM \simeq \bigoplus _{ i \in \bZ } \cH ^{ i } ( \cM ) [ - i ]
\)
and hence the equality
\(
    \ell ( \cM ) = \sum _{ i \in \bZ } \ell ( \cH ^{ i } ( \cM ) )
\).
By the structure theorem for finitely generated modules over a DVR, \( \cH ^{ i } ( \cM )\) is a direct sum of finite copies of \(\cO _{ \hirzebruchtwo , \gamma } \) and
\(
    \cO _{ \hirzebruchtwo , \gamma } / ( t ^{ p } )
\)
for various \( p \ge 1 \), where \(t\) is a generator of the maximal ideal. Based on this, we obtain the following invariance.
\begin{lemma}\label{lm:length is invariant under dual}
    For any object \( \cE \in \derived ( \hirzebruchtwo ) \), \( \ell ( \cE ^{ \vee } ) = \ell ( \cE ) \).
\end{lemma}

\begin{proof}
    Note first that
    \(
        \ell ( \cE ^{ \vee } )
        =
        \ell ( ( \cE ^{ \vee} ) _{ \gamma } )
        =
        \ell ( ( \cE _{ \gamma } ) ^{ \vee } )
    \).
    Hence it is enough to show
    \(
        \ell ( \cM ^{ \vee } ) = \ell ( \cM )
    \)
    for all \( \cM \in \derived ( \cO _{ \hirzebruchtwo, \gamma } ) \). By the explicit descriptions of \(\cM\) we gave above, it is enough to show this for \( \cM = \cO _{ \hirzebruchtwo, \gamma } / ( t ^{ p } ) \). In this case one can easily confirm \( \cM ^{ \vee } \simeq \cM [ - 1 ] \), so we are done.
\end{proof}

If $\cE$ is an exceptional object with $\Supp(\cH^{i_0})=\hirzebruchtwo$, then \pref{pr:chohomology sheaves are O_C modules} implies
\begin{align}\label{eq:length of an exceptional object}
    \begin{aligned}
        \ell (\cE) & = \sum_i \rank_C \tors(\cH^i(\cE)) \\
	&= \sum_{i\ne i_0} \rank_C \cH^i(\cE)+\rank_C \tors \cH^{i_0}(\cE),
\end{aligned}
\end{align}
where $\rank_C$ denotes the rank of a coherent sheaf on $C$. Note that thus defined \( \ell ( \cE ) \) is the same as \(\ell ( \cE )\) in \pref{cr:structure of the residual part}.
From this we immediately obtain the following characterization of exceptional vector bundles among exceptional objects.

\begin{lemma}\label{lm:cE is a vector bundle iff ell = 0}
    An exceptional object $\cE$ is isomorphic to a shift of a vector bundle if and only if $\ell (\cE)=0$.
\end{lemma}

The following (sub)additivity of the length function with respect to short exact sequences will be useful later.

\begin{lemma}\label{lm:inequality for l}
For an exact sequence
\[
0 \to \cH_1 \to \cH_2 \to \cH_3 \to 0
\]
in \(\coh \hirzebruchtwo\), an inequality
\[
    \ell (\cH_2) \le \ell (\cH_1) + \ell (\cH_3)
\]
holds. This is an equality if $\cH_1$  is a torsion sheaf.
\end{lemma}

\begin{proof}
    By taking the stalks at \(\gamma\), it is enough to show the similar statements for \( \coh \cO _{ \hirzebruchtwo, \gamma } \).
    It follows from the standard arguments on local cohomology that
    \begin{align}
        0
        \to
        \tors \cH _{ 1, \gamma }
        \to
        \tors \cH _{ 2, \gamma }
        \to
        \tors \cH _{ 3, \gamma }
    \end{align}
    is exact and that the last map is a surjection if \( \cH _{ 1 } \) is itself torsion (note that \( \tors \) is isomorphic to the 0-th local cohomology functor \( \Gamma _{ \{0\} } \) at the closed point \(0 \in \Spec \cO _{ \hirzebruchtwo, \gamma } \)).
    The assertions immediately follow from this and the additivity of the length of modules under short exact sequences.
\end{proof}
%
%
\subsection{Proof of \pref{th:exceptional objects are equivalent to vector bundles}}
\label{sc:Proof of Theorem 3.1}

Let us complete the proof of \pref{th:exceptional objects are equivalent to vector bundles}.
Recall again the decomposition
\[
    \cH^{i_0}(\cE) \simeq E \oplus T
\]
from \pref{lm:exceptional sheaf is a direct summand}, where \(E\) is an exceptional sheaf and \(T\) is a torsion sheaf. There are two possibilities as follows.

\begin{enumerate}
\item\label{it:The case where E is not torsion free}
\( E \) is not torsion free. Then by \pref{lm:exceptional sheaf is a direct summand}, there are
\(
    a \in \bZ, d \in \bZ _{ > 0 }
\)
such that
\begin{align}\label{eq:torsE and E|_C}
    \tors E
    \simeq
    \cO_C(a)^{\oplus d}.
\end{align}

\item\label{it:The case where E is torsion free}
\( E \) is torsion free. Then by \pref{cr:E(cE vee) has non-trivial torsion},
\(\cE ^{ \vee }\) fits into the case \pref{it:The case where E is not torsion free}.
Namely, the exceptional sheaf
\(
    E ' \coloneqq E ( \cE ^{ \vee } )
\)
has a non-trivial torsion and there is a torsion sheaf \( T ' \), \( a' \in \bZ \), and \( d ' > 0 \) such that
\(
    \cH^{-i_0}(\cE^\vee) \simeq E ' \oplus T '
\)
and
\(
    \tors E ' \simeq \cO _{ C } ( a ' ) ^{ \oplus d ' }
\).
\end{enumerate}

\pref{th:twisting exceptional collection to vector bundles} obviously follows from the following theorem, which asserts that it is always possible to decrease the length of \(\cE\) by an appropriate spherical twist.

\begin{theorem}\label{th:decrease the length by appropriate twist}
Under the notation of the previous paragraphs, the following holds.
\begin{itemize}
    \item 
    In the case \pref{it:The case where E is not torsion free}, i.e., if \(E\) is not torsion free, then
    \( \ell ( T _{ a } \cE ) < \ell ( \cE ) \).

    \item
    In the case \pref{it:The case where E is torsion free}, i.e., if \(E\) is torsion free, then
    \( \ell ( T _{ - a ' - 3 } \cE ) < \ell ( \cE ) \).
\end{itemize}
\end{theorem}

\begin{proof}
In this proof, to make life easy, we assume \( i _{ 0 } = i _{ 0 } ( \cE ) = 0 \).
The general case is easily reduced to this just by
replacing \(\cE\) with \(\cE[ i _{ 0 } ] \).

Consider first the case \pref{it:The case where E is not torsion free}.
For each $i \ne 0$, by \pref{cr:structure of the residual part} there are $s_i, t _{ i } \ge 0$ such that
\(
    \cH ^{ i } ( \cE ) \simeq \cO_C(a)^{\oplus s_i} \oplus \cO_C(a+1)^{\oplus t_i}
\).
Putting $s_{0} \coloneqq s-d$ and $t_{0} \coloneqq t$ in \eqref{eq:desired decomposition}, we can also write
\(
    \cH^{0}(\cE) \simeq E \oplus \cO_C(a)^{\oplus s_{0}} \oplus \cO_C(a+1)^{\oplus t_{0}}
\).
Consider the following spectral sequence.
\[
    E_2^{p,q}
    =
    \cH^p ( T_a ( \cH^q( \cE ) ) )
    \Rightarrow
    \cH^{p+q} ( T _{ a } \cE )
\]
By direct computations we easily see that \(E _{ 2 } ^{ p, q }\) are as follows.
\begin{align}
    E _{ 2 } ^{ p, q } \simeq
    \begin{cases}
        \cO _{ C } ( a - 1 ) ^{ \oplus t _{ q } } & p = - 1 \\
        \cF ( = E / \tors E ) & ( p, q ) = ( 0 , 0 ) \\
        \cO _{ C } ( a ) ^{ \oplus s _{ q } } & p = 1 \\
        0 & \text{otherwise}
    \end{cases}
\end{align}
Thus we see
\begin{align}\label{eq:inequality for E_2}
    \sum_{p,q} \ell (E_2^{ p, q }) = \ell (\cE) - \ell (\tors E)<\ell (\cE).
\end{align}
Noting that the differential map $d_2^{p,q}$ is non-zero only if $p=-1$, we easily see that
\( E _{ 3 } ^{ p, q } \) are as follows.
\[
    E_3^{p,q}
    \simeq
    \begin{cases}
        \ker d_2^{-1,q} &p=-1 \\
		\coker d_2^{-1,q+1} & p=1\\
		E_2^{p,q} &\text{otherwise}
	\end{cases}
\]
For each \( q \in \bZ \), since \( E _{ 2 } ^{ - 1, q } \) is torsion, it follows from \pref{lm:inequality for l} that
\(
    \ell ( \coker d _{ 2 } ^{ - 1, q + 1 } ) \le \ell ( E _{ 2 } ^{ 1, q - 1 } )
\).
Also, the inclusion \( \ker d _{ 2 } ^{ - 1, q } \hookrightarrow E _{ 2 } ^{ - 1 , q } \) implies \( \ell (\ker d _{ 2 } ^{ - 1, q }) \le  \ell (E _{ 2 } ^{ - 1 , q }) \). Thus we have confirmed the inequality
\begin{align}\label{eq:inequality for E_3}
    \ell (E_3^{p,q}) \le \ell (E_2^{p,q}) \quad \forall ( p, q ).
\end{align}
Finally, the spectral sequence degenerates at $E_3$
and hence \pref{lm:inequality for l} implies
\begin{align}\label{eq:inequality for cE}
    \ell (T_a \cE) \le \sum _{ p, q } \ell ( E _{ \infty } ^{ p, q } )
    =
    \sum_{p,q} \ell (E_3^{p,q}).
\end{align}
Combining \eqref{eq:inequality for E_2}, \eqref{eq:inequality for E_3}, and \eqref{eq:inequality for cE},
we conclude $\ell (T_a\cE)<\ell (\cE)$.

Consider the other case \pref{it:The case where E is torsion free}. By \pref{cr:E(cE vee) has non-trivial torsion}, \(\cE ^{ \vee } \) is then as in the case \pref{it:The case where E is not torsion free}. By applying our conclusion for the case \pref{it:The case where E is not torsion free} to \(\cE ^{ \vee } \), we obtain the inequality
\begin{align}\label{eq: Ta' decreases the length of cEvee}
    \ell ( T_{ a ' } ( \cE^\vee ) )
    <
    \ell ( \cE^\vee ).
\end{align}
On the other hand, the left hand side is computed as follows.
\begin{align}\label{eq:computation of length}
    \begin{aligned}
        \ell ( T_{ a ' } \left( \cE ^{ \vee } \right) )
        =
        \ell ( \cO _{ \hirzebruchtwo } ( - C ) \otimes T_{ a ' } \left( \cE ^{ \vee } \right) )
        \stackrel{\eqref{eq:Ta Ta+1 = O(C)}}{=}
        \ell ( T ' _{ a ' + 1 } \left( \cE ^{ \vee } \right) )\\
        \stackrel{\eqref{eq:dual and spherical twists}}{=}
        \ell \left( \left( T _{ - a ' - 3 } \cE \right) ^{ \vee } \right)
        \stackrel{\text{\pref{lm:length is invariant under dual}}}{=}
        \ell \left( T _{ - a ' - 3 } \cE \right).
    \end{aligned}
\end{align}
Thus we see that
\(
    \ell ( T _{ - a ' - 3 } \cE )
    \stackrel{\eqref{eq:computation of length}}{=}
    \ell ( T_{ a ' } \left( \cE ^{ \vee } \right) )
    \stackrel{\eqref{eq: Ta' decreases the length of cEvee}}{<}
    \ell ( \cE ^{ \vee } )
    \stackrel{\text{\pref{lm:length is invariant under dual}}}{=}
    \ell ( \cE )
\).
\end{proof}

\begin{remark}
    Suppose both \( \tors E ( \cE ) \ne 0 \) and \( \tors E ( \cE ^{ \vee } ) \ne 0 \). From the third item of \pref{lm:dual exceptional object} and \eqref{eq:structure of cH i0 + 1 ( cE )} in the proof of \pref{lm:reduction of length} below, it actually follows that \( a = - a' - 3 \). This does not seem to be a mere coincidence, but the authors do not have a good account of this.
\end{remark}

%
%
\section{Exceptional objects sharing the same class in \( \kgr{ \hirzebruchtwo }\)}
\label{sc:Exceptional objects sharing the same class}

The goal of this section is to prove \pref{cr:exceptional objects in the same numerical class} on the set of exceptional objects sharing the same class in \(\kgr{\hirzebruchtwo}\).

Let \( \cE \in \derived ( \hirzebruchtwo ) \) be an exceptional vector bundle of rank \( r \).
Recall from \eqref{eq:restriction of exceptional vector bundle to C} of \pref{lm:exceptional vector bundle restricted to C} that there is an isomorphism
\begin{align}
    \cE | _{ C }
    \simeq
    \cO _{ C } ( b ) ^{ \oplus s }
    \oplus
    \cO _{ C } ( b + 1 ) ^{ \oplus r - s }
\end{align}
for some \( b \in \bZ \) and \( s \in \bN \) such that \( 1 \le s \le r \).
We freely use this result, especially the symbols \( r, s\), and \( b \), throughout this section.

\begin{lemma}\label{lm:Tb-1 cE is again a vb}
Let \( \cE \in \derived ( \hirzebruchtwo ) \) be an exceptional vector bundle.
Then
\(
    T _{ b - 1} \cE
\)
and
\(
    T ' _{ b } \cE
\)
are exceptional vector bundles.
\end{lemma}

\begin{proof}
From the defining distinguished triangle of spherical twists \eqref{eq:triangle of spherical twist}, it immediately follows that
\(
    T _{ b - 1 } \cE
\)
is an exceptional sheaf. Note that it is isomorphic to
\(\left( T ' _{ b } \cE \right) ( C )\). It then follows from the defining distinguished triangle for the inverse spherical twist \eqref{eq:triangle of inverse spherical twist} that \( T ' _{ b } \cE \) is torsion free. Thus we see that \( T _{ b - 1 } \cE \) and \( T ' _{ b } \cE \) are both exceptional vector bundles.
\end{proof}

\begin{lemma}\label{lm:Tb-1 cE and Tb-2 cE are both vbs}
Let \( \cE \in \derived ( \hirzebruchtwo ) \) be an exceptional vector bundle such that
\(
    s = r
\)
holds in \eqref{eq:restriction of exceptional vector bundle to C}.
Then
\(
    T _{ b - 1 } \cE \simeq \cE
\)
and
\(
    T _{ b - 2 } \cE \simeq \cE ( C )
\).
\end{lemma}

\begin{proof}
The first assertion immediately follows from the vanishing
\(
    \RHom _{ \hirzebruchtwo } ( \cO _{ C } ( b - 1 ), \cE ) = 0
\).
The second assertion follows from \( \cE \simeq T ' _{ b - 1 } \cE \), which is nothing but the first assertion, and the isomorphism
\(
    T _{ b - 2 } \cE \simeq \cO _{ \hirzebruchtwo } ( C ) \otimes _{ \cO _{ \hirzebruchtwo } } T ' _{ b - 1 } \cE
\) (which follows from \eqref{eq:square of inverse twists as square of twists and O(-2C)} and \eqref{eq:exchanging Ta and O(C)}).
\end{proof}

\begin{theorem}\label{th:exceptional object is K0-trivial equivalent to vector bundle}
Let
\(
    \cE \in \bfD ( \hirzebruchtwo )
\)
be an exceptional object with \( \rank \cE > 0 \).
Then there exists
\(
    b \in \btriv
\)
and
\(
    m ' \in \bZ 
\)
such that
\(
    b ( \cE ) [ 2 m ' ]
\)
is an exceptional vector bundle.
\end{theorem}

\begin{proof}
Let us choose an isomorphism as in \eqref{eq:exceptional objects are equivalent to vector bundles}.
If \( n \) happens to be an odd number, noting that
\(
    \cF \simeq T _{ b } \left( T ' _{ b } \cF \right)
\)
and
\(
    T ' _{ b } \cF
\)
is an exceptional vector bundle for suitable
\(
    b \in \bZ
\)
by \pref{lm:Tb-1 cE is again a vb}, we may assume without loss of generality that
\(
    n
\)
is even. Now the assertion is an immediate consequence of \pref{pr:TaTb as product of O (m C ) and squares of twists}, since
\(
    T _{ a } ^{ \pm 2 } \in \btriv
\)
for any
\(
    a \in \bZ
\).
\end{proof}

\begin{corollary}\label{cr:exceptional objects in the same numerical class}
Let \(\cE \in \derived ( \hirzebruchtwo )\) be an exceptional object. Then
\begin{enumerate}
\item\label{it:exceptional object is numerically equivalent to a vb}
There exists a unique exceptional vector bundle \( \cF \) on \( \hirzebruchtwo \) such that
\begin{align}
    [ \cE ] = 
    \begin{cases}
        [ \cF ] \in \kgr{ \hirzebruchtwo } & \text{if } \rank \cE > 0\\
        - [ \cF ] \in \kgr{ \hirzebruchtwo } & \text{if } \rank \cE < 0
    \end{cases}
\end{align}
(recall that \(\rank \cE \neq 0\) by \pref{cr:rank is never 0}).

\item\label{it:transitivity in the same numerical class}
The action of the group
\(
    \btriv \times 2 \bZ
\)
on the following set is transitive.
\begin{align}
    \left\{
        \cE ' \in \derived ( \hirzebruchtwo )
        \mid
        \text{exceptional object such that }
        [ \cE '] = [ \cE ] \in \kgr { \hirzebruchtwo }
    \right\}
\end{align}
\end{enumerate}
\end{corollary}

\begin{proof}
The existence of \( \cF \) as in \eqref{it:exceptional object is numerically equivalent to a vb} is a direct consequence of \pref{th:exceptional object is K0-trivial equivalent to vector bundle}. The uniqueness of such \( \cF \) is \pref{lm:Lemma 3.5 of [OU]}.
One can prove \eqref{it:transitivity in the same numerical class} again by \pref{th:exceptional object is K0-trivial equivalent to vector bundle}, by showing that any \( \cE ' \) is in the same orbit of \( \cF \) or \( \cF [ 1 ] \), depending on
\(
    \rank \cE > 0
\)
or
\(
    \rank \cE < 0
\).
\end{proof}

%
%
\section{Constructibility of exceptional collections}
\label{sc:Constructibility of exceptional collections}

The aim of this section is \pref{cr:constructibility}, which asserts that any exceptional collection on \( \hirzebruchtwo \) is extendable to a full exceptional collection.

We first show that any exceptional collection on \(\hirzebruchtwo\) is sent to an exceptional collection consisting of (shifts of) vector bundles by a sequence of spherical twists.

\begin{lemma}\label{lm:reduction of length}
Let
\( 
    ( \cB, \cE ) \in \ec _{ 2 } ( \hirzebruchtwo )
\)
be an exceptional pair such that
\(
    \cB
\)
is a vector bundle and
$
\cE
$
is not isomorphic to a shift of a sheaf.
Suppose also that \( E = E ( \cE ) \) defined in \pref{df:E and cF} is not torsion free, so that
\(
    \tors E
\) 
is a direct sum of copies of \( \cO _{ C } ( a ) \) for some \( a \in \bZ \). Then $T_a \cB$ is isomorphic to either
\(
    \cB
\)
or
\(
    \cB ( C )
\); in particular, it remains to be a vector bundle.
\end{lemma}

\begin{proof}
Recall the decomposition
\(
    \cH ^{ i _{ 0 } } ( \cE )
    =
    E
    \oplus
    T
\)
from \pref{lm:exceptional sheaf is a direct summand}.
Set $\tors E \simeq \cO _{ C } ( a ) ^{ \oplus d}$, where \( d > 0 \) by the assumption.
By \pref{cr:structure of the residual part}, there is an isomorphism
\begin{align}
    \tors \cohomology ( \cE )
    \simeq
    \bigoplus _{ i \neq i _{ 0 } } \cH ^{ i } ( \cE )
    \oplus
    T
    \oplus
    \tors E
    \simeq
    \cO _{ C } ( a ) ^{ \oplus s }
    \oplus
    \cO _{ C } ( a + 1 ) ^{ \oplus \ell ( \cE ) - s }
\end{align}
for some
\(
    1 \le s \le \ell ( \cE )
\).
Also, again by \pref{lm:exceptional vector bundle restricted to C}, there is an isomorphism
$$
    \cB | _{ C }
    \simeq
    \cO _{ C } ( b ) ^{ \oplus s ' }
    \oplus
    \cO _{ C } ( b + 1  ) ^{ \oplus r ' - s ' }
$$
for some \( b \in \bZ \) and
\(
    1 \le s ' \le r '
\).

Consider the following spectral sequence.
\begin{align}\label{eq:E_2 spectral sequence for (cB, cE)}
    E_2^{p,q}
    =
    \Ext _{ \hirzebruchtwo } ^{ p } ( \cH ^{ - q } ( \cE ), \cB )
    \Rightarrow
    \Ext _{ \hirzebruchtwo } ^{ p + q } ( \cE, \cB )
\end{align}
The assumption
\(
    \RHom _{ \hirzebruchtwo } ( \cE, \cB )
    =
    0
\)
implies the vanishing of the limit
\(
    \Ext _{ \hirzebruchtwo } ^{ n } ( \cE, \cB ) = 0
\)
for all \( n \in \bZ \). Also, since \( \hirzebruchtwo \) is a smooth projective surface,
\(
    E _{ 2 } ^{ p, q } \ne 0
\)
only if
\(
    0 \le p \le 2
\)
and hence \eqref{eq:E_2 spectral sequence for (cB, cE)} is \( E _{ 3 } \)-degenerate everywhere and \( E _{ 2 } \)-degenerate at \( p = 1 \).
These imply that
\begin{itemize}
    \item \( E _{ 2 } ^{ 0, q } = 0 \) for all \( q \ne - i _{ 0 } \), for \( \cH ^{ - q } ( \cE ) \) being torsion and \( \cB \) being torsion free. Hence \eqref{eq:E_2 spectral sequence for (cB, cE)} is \( E _{ 2 } \)-degenerate at \( ( 2, q ) \) for each \( q \ne - i _{ 0 } - 1 \), so that \( E _{ 2 } ^{ 2, q } \simeq E _{ \infty } ^{ 2, q } = 0 \).
    \item \( E _{ 2 } ^{ 1, q } = 0 \) for all \( q \).
\end{itemize}
Thus we have confirmed that
\begin{itemize}
    \item \( E _{ 2 } ^{ p, q } \neq 0 \) only if \( ( p, q ) = ( 0, - i _{ 0 } ) \) or \( ( 2, - i _{ 0 } - 1 ) \).
    Moreover,
    \item \( d _{ 2 } ^{ 0, - i _{ 0 } } \colon E _{ 2 } ^{ 0, - i _{ 0 } } \to E _{ 2 } ^{ 2, - i _{ 0 } - 1 } \)
    is an isomorphism by the \( E _{ 3 } \)-degeneracy of \eqref{eq:E_2 spectral sequence for (cB, cE)} and the vanishing of the limit.
\end{itemize}

Note that
\(  
    0 = E_2^{2, - i _{ 0 } } \simeq \Ext _{ \hirzebruchtwo } ^{ 2} ( E \oplus T, \cB )
\)
and the Serre duality imply $\Hom _{ \hirzebruchtwo }  (\cB,\tors E) = 0$. Thus we see
\begin{align}
    a \le b - 1.
\end{align}

Take any
\(
    q \neq - i _{ 0 }, - i _{ 0 } - 1
\).
We know that
\(
    \cH ^{ - q } ( \cE )
\)
is a direct sum of invertible sheaves on \( C \simeq \bP ^{ 1 } \), so the vanishings
\(
    E _{ 2 } ^{ 1, q }
    =
    0
    =
    E _{ 2 } ^{ 2, q }
\)
imply
\(
    \cH ^{ - q } ( \cE ) = 0
\)
if \( s ' < r ' \), and that
\(
    \cH ^{ - q } ( \cE )
\)
is a direct sum of (possibly \(0\)) copies of
\(
    \cO _{ C } ( b - 1 )
\)
if \( s ' = r ' \).

Let us first settle the case \( s ' < r ' \). As we mentioned in the previous paragraph, in this case \( \cH ^{ i } ( \cE ) \ne 0 \) only if \( i = i _{ 0 } \) or \( i _{ 0 } + 1\).
Note in fact that \( \cH ^{ i _{ 0 } + 1 } ( \cE ) \ne 0 \), since it is assumed in the statement that \( \cE \) is not isomorphic to a shift of a sheaf.

It follows from the vanishing 
\(
    0 =
    E _{ 2 } ^{ 0, - 1 }
       =
    \Hom _{ \hirzebruchtwo } ( \cH ^{ i _{ 0 } + 1 } ( \cE ), \cH ^{ i _{ 0 } } (\cE) )
\)
in the spectral sequence \eqref{eq:E_2 spectral sequence of Hom}
that
$0=  \Hom _{ \hirzebruchtwo } ( \cH ^{ i _{ 0 } + 1 } ( \cE ), \tors E)$ and thus
\begin{align}\label{eq:structure of cH  i0 + 1 ( cE )}
    \cH ^{ i _{ 0 } + 1 } ( \cE )
    \simeq
    \cO _{ C } ( a + 1 ) ^{ \oplus t }
\end{align} for some $t > 0$.
On the other hand, we know that
$$
    0
    =
    E _{ 2 } ^{ 1, - i _{ 0 } - 1 }
    =
    \Ext _{ \hirzebruchtwo } ^{ 1 } ( \cH ^{ i _{ 0 } + 1 } ( \cE ), \cB ),
$$
so that
\(
    a + 1 \ge b
\)
(recall \( r ' - s ' > 0\)).
Thus we see that
\(
    a = b - 1
\).
Hence
\(
    T _{ a } (\cB) \simeq \cB
\)
by \pref{lm:Tb-1 cE is again a vb}.

Next consider the case $s'=r'$. Note that at least one of the sheaves $\cO_C(a)$ or $\cO_C(a+1)$ appears as a direct summand of
$\bigoplus_{i\ne i_0}\cH^i (\cE)$, since $\cE$ is not a shift of a sheaf. Then the vanishing 
$0=E_2^{1,q}=\Ext^1(\cH^{-q}(\cE),\cB)$
for all
\(
    q \neq -i_0
\)
in \eqref{eq:E_2 spectral sequence for (cB, cE)}
implies that either $a= b - 1 $ or $ b - 2 $. Then by \pref{lm:Tb-1 cE and Tb-2 cE are both vbs}, $T_a(\cB)$ is isomorphic to
\(
    \cB
\)
and
\(
    \cB ( C )
\),
respectively.
\end{proof}

\begin{lemma}\label{lm:reduction of length for sheaves}
Let $(\cB, \cE) \in \ec_{2}(\hirzebruchtwo)$ be an exceptional pair. Suppose that $\cB$ is a vector bundle and $\cE$ is a sheaf such that
\[
    \cT \coloneqq \tors \cE = \tors E \simeq \cO_C(a)^{\oplus d} \ne 0.
\]
Then $T_a \cB$ is isomorphic to either
\(
    \cB
\)
or
\(
    \cB ( C )
\)
and it holds that $(T_a\cB, T_a\cE) \in \ecvb_{2}(\hirzebruchtwo)$.
\end{lemma}

\begin{proof}
Consider the short exact sequence as follows.
\begin{align}\label{eq:TEF}
    0 \to \cT \to \cE \to \cF \to 0
\end{align}
By \cite[Theorem~1.4~(1)]{MR3431636}, we know that
\(
    d = \hom _{ \hirzebruchtwo } ( \cO _{ C } ( a ), \cE )
\)
and
\( T _{ a } \cE \simeq \cF \) is an exceptional vector bundle.
Hence all we have to show is that \( T _{ a } \cB \) is a vector bundle.

The vanishing $\RHom(\cE, \cB)=0$ implies
\(
    \Ext _{ \hirzebruchtwo } ^i(\cT, \cB) \simeq \Ext _{ \hirzebruchtwo } ^{i+1}(\cF, \cB)
\).
Especially one has
\begin{align}
\Ext _{ \hirzebruchtwo } ^1(\cT, \cB) &\simeq \Ext _{ \hirzebruchtwo } ^2(\cF, \cB), \label{eq:Ext^1=Ext^2}\\
\Ext _{ \hirzebruchtwo } ^2(\cT, \cB)&=0.\label{eq:Ext^2=0}
\end{align}
If $\Ext _{ \hirzebruchtwo } ^1(\cT, \cB)=0$, then it follows that $\RHom(\cO_C(a), \cB)=0$
and hence $T_a \cB = \cB$.
Therefore we may assume $\Ext _{ \hirzebruchtwo } ^1(\cT, \cB) \ne 0$, or by taking the dual,
\begin{align}\label{eq:non-vanishing}
 \Ext _{ \hirzebruchtwo } ^1(\cB, \cT \otimes K _{ \hirzebruchtwo } ) \ne 0.
\end{align}
The Serre dual of \eqref{eq:Ext^1=Ext^2} and its restriction to $C$ yield the commutative square as follows.
\begin{equation}\label{eq:restrict to C}
\begin{tikzcd}
    \Hom _{ \hirzebruchtwo } (\cB, \cF \otimes K _{ \hirzebruchtwo } ) \arrow[r, "\simeq"] \arrow[d] & \Ext _{ \hirzebruchtwo } ^1(\cB, \cT \otimes K _{ \hirzebruchtwo } )
    \arrow[d, "\simeq"] \\
    \Hom_C(\cB|_C, \cF|_C) \arrow[r, twoheadrightarrow] & \Ext_C^1(\cB|_C, \cT)
\end{tikzcd}
\end{equation}
In the second row of the diagram \eqref{eq:restrict to C}, we omit $\otimes K _{ \hirzebruchtwo } $ by fixing an isomorphism $K _{ \hirzebruchtwo }  \otimes \cO_C \simeq \cO_C$ and regard \( \cT \) as a sheaf on \( C \).
By restricting the locally split short exact sequence \eqref{eq:TEF} to \( C \), we obtain the following short exact sequence.
\begin{align}\label{eq:restriction of cE}
0 \to \cT \to \cE|_C \to \cF|_C \to 0
\end{align}
From \eqref{eq:restriction of cE} and the surjectivity of the second row in \eqref{eq:restrict to C}, we obtain the following isomorphism.
\begin{align}\label{eq:isom of Ext^1}
    \Ext_C^1(\cB|_C, \cE|_C) \simeq \Ext_C^1(\cB|_C, \cF|_C)
\end{align}
As in the proof of \pref{lm:reduction of length}, put
\[
    \cB | _{ C }
    \simeq
    \cO _{ C } ( b ) ^{ \oplus s ' }
    \oplus
    \cO _{ C } ( b + 1 ) ^{ \oplus r ' - s ' }
\]
for
\(
    b \in \bZ
\)
and
\(
    1 \le s ' \le r '
\)
where $r'=\rank \cB$.
Then \eqref{eq:Ext^2=0} implies
\[
    a \le  b - 1.
\]
To determine the value of $a$, let us compute the dimensions of the both
sides of \eqref{eq:isom of Ext^1}.

In order to compute the dimension of the left hand side, recall that
\(
    \cE|_C \simeq \cO_C(a+1) ^{ \oplus e }
\)
by \cite[Theorem 1.4(1)]{MR3431636}. We know moreover that $ e =d+r$ by \eqref{eq:restriction of cE}, where $r=\rank \cF$.
Thus we obtain the following descriptions.
\begin{align}\label{eq:ext_C^1(cB|_C, cE|_C)}
\ext_C^1(\cB|_C, \cE|_C)=
	\begin{cases}
		0 & \text{if } a = b - 1 \\
		e \left((r'-s')( b - a - 1 )+s'( b - a - 2 )\right) & \text{if } a \le b - 2
	\end{cases}
\end{align}

Let us compute the dimension of the right hand side of \eqref{eq:isom of Ext^1}.
Since \( \cF \) is an exceptional vector bundle, there are \( f \in \bZ \) and
\(
    1 \le s \le r
\)
such that
\[
   \cF | _{ C }
    \simeq
    \cO _{ C } ( f ) ^{ \oplus s }
    \oplus
    \cO _{ C }( f + 1 ) ^{ \oplus r - s }.
\]
Note that by \eqref{eq:non-vanishing} and \eqref{eq:restrict to C} we have
\[
    \Hom_C(\cB|_C, \cF|_C) \ne 0,
\]
which implies $f \ge b - 1$.

Suppose for a contradiction that $\ext_C^1(\cB|_C, \cF|_C) \ne 0$.
Then we obtain $f = b - 1 $ and
\begin{align}\label{eq:ext^1(cB|_C, cG|_C)}
    0 \ne \ext_C^1(\cB|_C, \cF|_C)=s(r'-s').
\end{align}
Using
\(
    s \le r < r + d = e
\),
we immediately see that \eqref{eq:ext^1(cB|_C, cG|_C)} is strictly smaller than the second line of the right hand side of \eqref{eq:ext_C^1(cB|_C, cE|_C)}.
This contradicts the isomorphism \eqref{eq:isom of Ext^1}. Thus we have confirmed
\[
    \ext_C^1(\cB|_C, \cE|_C)
    \stackrel{\eqref{eq:isom of Ext^1}}{=}
    \ext_C ^{ 1 } (\cB|_C, \cF|_C) = 0.
\]
Then \eqref{eq:ext_C^1(cB|_C, cE|_C)} implies either $a = b - 1 $ or $(a, r'-s')=( b - 2, 0 )$.
Thus we see that $T_a \cB$ is isomorphic to \( \cB \) or \( \cB ( C ) \), respectively, by \pref{lm:Tb-1 cE is again a vb} and \pref{lm:Tb-1 cE and Tb-2 cE are both vbs}. Thus we conclude the proof.
\end{proof}

\begin{corollary}\label{cr:decrease ell while preserving the vector bundle}
    Let \( ( \cB, \cE ) \in \ec _{ 2 } ( \hirzebruchtwo ) \) be an exceptional pair such that \( \cB \) is a vector bundle and \( \ell ( \cE ) > 0 \).
    \begin{enumerate}
        \item Suppose that \( \tors E ( \cE ) \) is non-zero and is a direct sum of copies of \( \cO _{ C } ( a ) \). Then \( \ell ( T _{ a } ( \cE ) ) < \ell ( \cE ) \) and \( T _{ a } ( \cB ) \) is isomorphic to either \( \cB \) or \( \cB ( C ) \); in particular, it is a vector bundle.

        \item Suppose that \( \tors E ( \cE ) = 0 \), so that \( \tors E ( \cE ^{ \vee } ) \) is non-zero by \pref{lm:a sufficient condition for cE to be a vector bundle}. Suppose that it is a direct sum of copies of \( \cO _{ C } ( a ' ) \). Then \( \ell ( T _{ - a ' - 3 } ( \cE ) ) < \ell ( \cE ) \) and \( T _{ - a ' - 3 } ( \cB ) \) is isomorphic to either \( \cB \) or \( \cB ( C ) \); in particular, it is a vector bundle.
    \end{enumerate}
    In particular, there is \( c = c ( \cE ) \in \bZ \) which depends only on \( \cE \), independent of \( \cB \) in particular, such that
    \(
        \ell ( T _{ c } \cE ) < \ell ( \cE )
    \)
    and
    \(
        T _{ c } \cB
    \)
    is isomorphic to either \( \cB \) or \( \cB ( C ) \).
\end{corollary}

\begin{proof}
    The assertions on the lengths are already proven in \pref{th:decrease the length by appropriate twist}.

    Let us first assume \( \tors E ( \cE ) \ne 0 \). The case where \( \cE \) is not isomorphic to a shift of a sheaf is settled in \pref{lm:reduction of length}. The case \( \cE \) is isomorphic to a shift of a sheaf is settled in \pref{lm:reduction of length for sheaves} (we may assume \( i _{ 0 } = 0 \), without loss of generality).
    
If
\(
    \tors E ( \cE )
    =
    0
\),
we reduce the proof to the first case.
In fact, it follows from \pref{lm:a sufficient condition for cE to be a vector bundle} that
\(
    \tors E ( \cE ^{ \vee }) \ne 0
\).
Note that
\begin{align}\label{eq:dual pair}
    \left( \cB ^{ \vee } ( K _{ \hirzebruchtwo } ), \cE ^{ \vee } \right)
\end{align}
is also an exceptional pair such that the first component is a vector bundle and
\(
    \ell ( \cE ^{ \vee } ) \stackrel{\text{\pref{lm:length is invariant under dual}}}{=} \ell ( \cE ) > 0
\).
Suppose that
\(
    \tors E ( \cE ^{ \vee } )
\)
is a direct sum of copies of \( \cO _{ C } ( a ' ) \).
Then, by applying the conclusion of the previous paragraph to the pair \eqref{eq:dual pair}, it follows that
\(
    T _{ a '} ( \cB ^{ \vee} ( K _{ \hirzebruchtwo } ) )
\)
is isomorphic to either \(\cB ^{ \vee} ( K _{ \hirzebruchtwo } )\) or \(\cB ^{ \vee} ( K _{ \hirzebruchtwo } + C ) \). Then from the following computation we see that \(T _{ - 3 - a '} ( \cB )\) is isomorphic to either \( \cB \) or \( \cB ( C ) \).
\begin{align}
    T _{ a '} ( \cB ^{ \vee} ( K _{ \hirzebruchtwo } ) )
    \stackrel{\text{\pref{lm:conjugation of spherical twist}}}{\simeq}
    \left( T _{ a ' } ( \cB ^{ \vee} ) \right) ( K _{ \hirzebruchtwo } )
    \stackrel{\text{\pref{lm:spherical twist and dual} \eqref{eq:dual and spherical twists}}}{\simeq}
    \left( T ' _{ - 2 - a ' } ( \cB ) \right) ^{ \vee } ( K _{ \hirzebruchtwo } )\\
    \stackrel{\eqref{eq:square of inverse twists as square of twists and O(-2C)}}{\simeq}
    \left(T _{ - 1 - a '} \left( \cB ( - C ) \right)\right) ^{ \vee } ( K _{ \hirzebruchtwo } )
    \stackrel{\text{\pref{lm:conjugation of spherical twist}}}{\simeq}
    \left(T _{ - 3 - a '} ( \cB )\right) ^{ \vee } ( K _{ \hirzebruchtwo } + C )
\end{align}
\end{proof}

\begin{theorem}\label{th:twisting exceptional collection to vector bundles}
For any
\(
    N \in \bZ
\)
with
\(
    1 \le N \le 4
\)
and any exceptional collection
\(
    \cEbar
    =
    ( \cE _{ 1 }, \cE _{ 2 }, \dots, \cE _{ N } )
    \in
    \ec _{ N } ( \hirzebruchtwo )
\),
there exists a product of spherical twists of the form
\(
    T _{ a }
\)
for some
\(
    a \in \bZ
\),
denoted by
\( b \)
such that
\(
    b ( \cEbar )
    \in
    \bZ ^{ N } \cdot \ecvb _{ N } ( \hirzebruchtwo )
\).
\end{theorem}

\begin{proof}
Without loss of generality, we may and will assume
\(
    i _{ 0 } ( \cE _{ i } ) = 0
\)
for all
\(
    i = 1, \dots, N
\).
We prove the assertion by an induction on \( N \).

The case when \( N = 1 \) is nothing but \pref{th:exceptional objects are equivalent to vector bundles}.
Consider the case when \( N > 1 \). By applying the induction hypothesis to the subcollection
\(
    \left( \cE _{ 1 }, \dots, \cE _{ N - 1 } \right)
\),
we may and will assume that these are already vector bundles, say,
\(
    \left( F _{ 1 }, \dots, F _{ N - 1 } \right)
\).
Suppose \( \cE _{ N } \) is not a vector bundle; i.e., \( \ell ( \cE _{ N } ) > 0 \). Otherwise there is nothing to show.
In this case, since \( ( F _{ i }, \cE _{ N } ) \in \ec _{ 2 } ( \hirzebruchtwo ) \) for each \( 1 \le i \le N - 1 \), if we take \( c = c ( \cE _{ N } )\) as in \pref{cr:decrease ell while preserving the vector bundle}, then 
\(
    T _{ c } F _{ i }
\)
remains to be a vector bundle for all
\(
    i = 1, \dots, N - 1 
\)
and it holds that
\(
    \ell ( T _{ c } \cE _{ N } )
    <
    \ell ( \cE _{ N } )
\).
By repeating this process until \( \ell ( \cE _{ N } ) \) reaches \( 0 \), we achieve our goal.
\end{proof}

The constructibility for exceptional collections consisting of vector bundles is shown in \cite[Theorem 3.1.8.2]{MR1604186} for a class of weak del Pezzo surfaces. Though \( \hirzebruchtwo\) is not contained in the class, we can deduce the same assertion for \(\hirzebruchtwo\) from it:

\begin{theorem}\label{th:Constructibility for bundle collections}
Any exceptional collection on \( \hirzebruchtwo \) consisting of vector bundles can be extended to a full exceptional collection.
\end{theorem}

\begin{proof}
Let
\begin{align}
    \pi \colon Y \to \hirzebruchtwo
\end{align}
be the blowup of \( \hirzebruchtwo \) in a point outside of the curve \( C \).
Let
\( E \subset Y \)
be the exceptional curve.

Take
\(
    \cEbar
    \coloneqq
    \left( \cE _{ 1 }, \dots, \cE _{ N } \right) \in \ecvb _{ N } ( \hirzebruchtwo )
\),
so that
\(
    \left( \cO _{ E } ( - 1 ), \pi ^{ \ast } \cEbar\right) \in \ec _{ N + 1 } ( Y ).
\)
Write
\(
    F \coloneqq \pi ^{ \ast } \cE _{ 1 }
\), and consider the right mutation
\(
    R _{ F } \cO _{ E } ( - 1 )
\).
We claim that this is an exceptional vector bundle.

To see this, note first that the semiorthogonality
\(
    \RHom _{ Y } ( F, \cO _{ E } ( - 1 ) ) = 0
\)
implies
\(
    F | _{ E }
    \simeq
    \cO _{ E } ^{ \oplus r }
\),
where
\(
    r = \rank F
\).
Thus we obtain the following short exact sequence.

\begin{align}
    0
    \to
    F ^{ \oplus r }
    \to
    R _{ F } \cO _{ E } ( - 1 )
    \to
    \cO _{ E } ( - 1 )
    \to
    0.
\end{align}
Suppose for a contradiction that the torsion part of the exceptional sheaf
\(
    R _{ F } \cO _{ E } ( - 1 )
\)
is nontrivial. Then it should map injectively to \( \cO _{ E } ( - 1 ) \). It then implies that the canonical morphism from the locally free sheaf
\(
    F ^{ \oplus r }
\)
to the torsion free part of
\(
    R _{ F } \cO _{ E } ( - 1 )
\)
is both injective and surjective in codimension 1, which hence is an isomorphism. This contradicts the indecomposability of
\(
    R _{ F } \cO _{ E } ( - 1 )
\).

Thus we have obtained
\(
    \left( F, R _{ F } \cO _{ E } ( - 1 ),
    \pi ^{ \ast } \cE _{ 2 },
    \dots,
    \pi ^{ \ast } \cE _{ N }\right)
    \in
    \ecvb _{ N + 1 } ( Y )
\).
By \cite[Theorem~3.1.8.2]{MR1604186}, it extends to a full exceptional collection on \( Y \) (note that \(Y\) is a weak del Pezzo surface obtained by blowing up \( \bP ^{ 2 } \) first in a point and then in a point on the \((-1)\)-curve, (hence) that \( | - K_{Y} | \) is base point free and \( K _{ Y } ^{ 2 } = 7 > 1 \)). By applying the left mutation again, we obtain a full exceptional collection
\(
    \left(
    \cO _{ E } ( - 1 ),
    \pi ^{ \ast } \cE _{ 1 },
    \dots,
    \pi ^{ \ast } \cE _{ N },
    \cE ' _{ N + 1 },
    \dots,
    \cE ' _{ 4 }
    \right)
    \in
    \fec ( Y )
\).
Since
\(
    \cE ' _{ N + 1 },
    \dots,
    \cE ' _{ 4 } \in {} ^{\perp} \langle \cO _{ E } ( - 1 ) \rangle
\),
there are some objects
\(
    \cE _{ N + 1 }, \dots, \cE _{ 4 }
    \in
    \derived ( \hirzebruchtwo )
\)
such that
\(
    \cE ' _{ N + 1 } \simeq \bL \pi ^{ \ast } \cE _{ N + 1 },
    \dots,
    \cE ' _{ 4 } \simeq \bL \pi ^{\ast } \cE _{ 4 }
\).
Since
\( \bL \pi ^{ \ast } \colon \derived ( \hirzebruchtwo ) \to
{} ^{\perp} \langle \cO _{ E } ( - 1 ) \rangle
\)
is an equivalence, we see that
\(
    \cE _{1}, \dots, \cE _{ N }, \cE _{ N + 1 }, \dots, \cE _{ 4 }
\)
is a full exceptional collection of \(\derived ( \hirzebruchtwo  )\).
\end{proof}

We finally obtain the following constructibility theorem for \( \hirzebruchtwo \).
\begin{corollary}\label{cr:constructibility}
Any exceptional collection on \( \hirzebruchtwo \) can be extended to a full exceptional collection.
\end{corollary}

\begin{proof}
For any exceptional collection
\(
    \cEbar \in \ec _{ N } ( \hirzebruchtwo )
\),
by \pref{th:twisting exceptional collection to vector bundles}, there exists
\(
    b \in B
\)
such that
\(
    b \left( \cEbar \right) \in \bZ ^{ N } \cdot \ecvb _{ N } ( \hirzebruchtwo )
\).
Then by \pref{th:Constructibility for bundle collections},
\(
    b \left( \cEbar \right)
\)
can be extended to a full exceptional collection on \( \hirzebruchtwo \).
Applying \( b ^{ - 1 }\) to the extended collection, one obtains the desired full exceptional collection which extends
\(
    \cEbar
\).
\end{proof}

\begin{remark}
Contrary to \pref{th:exceptional object is K0-trivial equivalent to vector bundle}, an exceptional collection (consisting of  objects of \(\rank > 0\)) of length at least \(2\) is not necessarily numerically equivalent to an exceptional collection of vector bundles. Namely, for each \( N = 2, 3, 4, \) there is an exceptional collection
\(
    ( \cE _{ 1 }, \dots, \cE _{ N } ) \in \ec _{ N } ( \hirzebruchtwo )
\)
of length \(N\) and \( \rank \cE _{ i } > 0 \) for \( i = 1, \dots, N \) for which there is no exceptional collection of vector bundles
\(
    ( F _{ 1 }, \dots, F _{ N } ) \in \ecvb _{ N } ( \hirzebruchtwo )
\)
such that
\(
    [ \cE _{ i } ] = [ F _{ i } ] \in \kgr{ \hirzebruchtwo }
\)
for \( i = 1, \dots, N \).
\pref{eg:the counter example} below is such an example for \( N = 2 \). Examples for \( N = 3, 4 \) are obtained by extending examples of length \(2\) by \pref{cr:constructibility}.
For these exceptional collections, in particular, in \pref{th:twisting exceptional collection to vector bundles} one can not take \( b \) from \( \btriv \). This is in contrast to \pref{th:from numerically standard to standard}.
\end{remark}

\begin{example}\label{eg:the counter example}
Consider the following exceptional pair.
$$
    \cEbar
    \coloneqq
    ( \cE _{ 1 }, \cE _{ 2 } )
    \coloneqq
    T_{-1} \left( \cO _{ \hirzebruchtwo }, \cO _{ \hirzebruchtwo } ( C + 4 f ) \right)
    \in \ec_2(\hirzebruchtwo).
$$
Then there is no exceptional pair of vector bundles
\(
    ( F _{ 1 }, F _{ 2 } ) \in \ecvb _{ 2 } ( \hirzebruchtwo )
\)
such that
\(
    [ \cE _{ i } ] = [ F _{ i }] \in \kgr{ \hirzebruchtwo }
\)
for
\(
    i = 1, 2
\)
for the following reason.

Note first that \( \cE _{ 1 } \simeq \cO _{ \hirzebruchtwo } \) and \( \cE _{ 2 } \) is a sheaf with $\tors \cE _{ 2 }=\cO_C(-2)$ by \cite[Theorem 1.4]{MR3431636}. Suppose that there is a pair \( ( F _{ 1 }, F _{ 2 } ) \) as above.
Then $ F _{ 1 } \simeq \cO _{ \hirzebruchtwo } $ by \pref{lm:Lemma 3.5 of [OU]}, and since
$$
    [ F _{ 2 } ] = [\cE _{ 2 }]=[\cO _{ \hirzebruchtwo } (C+4f)]-\chi (\cO _{ \hirzebruchtwo } (C+4f),\cO_C(-1)) \left[ \cO_C(-1) \right] = [ \cO _{ \hirzebruchtwo } (3C+4f) ],
$$
it follows that \( F _{ 2 } \simeq \cO _{ \hirzebruchtwo } ( 3 C + 4 f ) \) again by \pref{lm:Lemma 3.5 of [OU]}.
This, however, leads to the following contradiction.
\begin{align}
    0
    =
    \Ext _{ \hirzebruchtwo } ^{ 2 } ( F _{ 2 }, F _{ 1 } )
    =
    \Ext^2 _{ \hirzebruchtwo }  ( \cO _{ \hirzebruchtwo } (3C+4f),\cO _{ \hirzebruchtwo } )
    \neq
    0
\end{align}
\end{example}

\begin{remark}
Given an exceptional collection
\(
    \cEbar
    =
    \left( \cE _{ 1 }, \dots, \cE _{ N } \right) \in \ec _{ N } ( \hirzebruchtwo )
\)
such that \( \rank \cE _{ i } > 0 \) for all \( i \), by \pref{cr:exceptional objects in the same numerical class} \eqref{it:exceptional object is numerically equivalent to a vb} there is a unique sequence of exceptional vector bundles
\(
    F _{ 1 }, \dots, F _{ N }
\)
such that
\(
    [ \cE _{i } ]
    =
    [ F _{ i } ]
    \in \kgr{ \hirzebruchtwo }
\)
for all
\(
    i = 1, \dots, N
\).
\pref{eg:the counter example} implies that \( ( F _{ 1 }, \dots, F _{ N } ) \) is \emph{not} necessarily an exceptional collection.

It also implies that the map
\( \gen | _{ \ecvb _{ N } ( \hirzebruchtwo ) } \colon \ecvb _{ N } ( \hirzebruchtwo ) \to \ec _{ N } ( \quadric) \) (the case \( N = 4 \) appears in \pref{fg:comparison of ec4 and ecvb4}) is \emph{not} surjective for \( N = 2, 3, 4\), though it is for \( N = 1 \) by \cite[Lemma~4.6]{MR3431636} and \pref{pr:characterization of exceptional objects in the same class}.
In fact, let
\(
    \cEbar = ( \cE _{ 1 }, \dots, \cE _{ N } ) \in \ec _{ N } ( \hirzebruchtwo )
\)
be an exceptional collection (consisting of objects of \(\rank > 0 \)) which is not numerically equivalent to an exceptional collection of vector bundles. Then
\(
    \gen ( \cEbar  )
\)
is not in the image of \(\gen | _{ \ecvb _{ N } ( \hirzebruchtwo ) }\). In fact, an exceptional collection of vector bundles
\(
    \cFbar = ( F _{ 1 }, \dots, F _{ N } ) \in \ecvb _{ N } ( \hirzebruchtwo )
\)
such that
\(
    \gen ( \cFbar ) = \gen ( \cEbar ) \in \ec _{ N } ( \quadric )
\)
must satisfy
\(
    [ \cE _{ i } ] = [ F _{ i } ] \in \kgr{\hirzebruchtwo}
\)
for
\(
    i = 1, \dots, N
\),
which contradicts the choice of \( \cEbar \).
\end{remark}

Despite \pref{eg:the counter example}, by \pref{cr:btriv is generated by squares}, one can always bring an arbitrary exceptional collection (consisting of objects of \(\rank > 0\)) to an exceptional collection of vector bundles by an element of \( \btriv \) \emph{up to a twist by \( T _{ 0 }\)}.

\begin{theorem}\label{th:normal form of exceptional collections}
Let \( \cEbar = ( \cE _{ 1 }, \dots, \cE _{ N } ) \in \ec _{ N } ( \hirzebruchtwo ) \) be an exceptional collection with \( 2 \le N \le 4 \) and \( \rank \cE _{ i } > 0 \) for all \( i = 1, \dots, N \).
Then there exists
\(
    b \in \btriv
\)
such that
\(
    b ( \cEbar )
    \in
    ( 2 \bZ ) ^{ N }
    \cdot
    \left(
    \ecvb _{ N } ( \hirzebruchtwo )
    \cup
    T _{ 0 } \left( \ecvb _{ N } ( \hirzebruchtwo ) \right)
    \right)
\).
\end{theorem}

\begin{proof}
We may assume without loss of generality that
\(
    i _{ 0 } = 0
\)
for any member of the collection \( \cEbar \).
By \pref{th:twisting exceptional collection to vector bundles}, there is \( b \in B \) such that
\(
    b ( \cEbar ) \in \ecvb _{ N } ( \hirzebruchtwo )
\).
Now the assertion immediately follows from the general description of elements of \( B \) given in \eqref{eq:even} and \eqref{eq:odd} (note that \( T _{ 0 } = T ' _{ 0 } ( T _{ 0 } ^{ 2 } ) \)).
\end{proof}

\begin{remark}\label{rm:stability conditions}
    Here we give some speculations on the spaces of Bridgeland stability conditions and a resulting question.

    To start with, it is conceivable that there is a local homeomorphism \( \varphi \) as in the following commutative diagram whose restriction to the subspaces of algebraic stability conditions is compatible with the generalization map
    \begin{align}
        \gen \colon \fsec ( \hirzebruchtwo ) \to \fsec ( \quadric ),
    \end{align}
    where
    \( \fsec ( \bullet )\) denote the set of isomorphism classes of full strong exceptional collections on \( \bullet \) (recall that a full strong exceptional collection yields a chamber of algebraic stability conditions in \( \Stab ( \bullet )\)). The fact that \(\gen\) restricts to the sets of \emph{strong} exceptional collections follows from \pref{rm:deformation of full collection is full} and \pref{cr:base change of exceptional collections}.

    \begin{equation}
        \begin{tikzcd}
            \Stab ( \hirzebruchtwo ) \arrow[r, dashed, "\varphi"] \arrow[d,"p"] & \Stab ( \quadric ) \arrow[d,"q"]\\
            \Hom ( \kgr{ \hirzebruchtwo }, \bC ) \arrow[r, "\sim", "( \kgr{\gen} ^{ - 1 })  ^{ \ast }"'] & \Hom ( \kgr{ \quadric }, \bC )
        \end{tikzcd}
    \end{equation}
    
    Conjecturally, the Galois group (\( = \) the group of fiber-preserving automorphisms of \( \Stab ( \hirzebruchtwo )\)) of \( p \) coincides with \( \btriv \times 2 \bZ \). As \( \btriv \) do not deform to \( \quadric \), it is conceivable that \( \btriv \) coincides with the Galois group of \( \varphi \).
    Therefore it seems reasonable to ask the following question, which is an analogue of \pref{cr:exceptional objects in the same numerical class} \eqref{it:transitivity in the same numerical class}. Unfortunately, \pref{th:normal form of exceptional collections} is not strong enough to answer it in the affirmative.
\end{remark}

\begin{question}
    For each \( \cEbar \in \fsec ( \hirzebruchtwo )\), the action of the group
    \(
        \btriv \times 2 \bZ
    \)
    on the following set is transitive.
    \begin{align}
        \left\{
            \cEbar ' \in \fsec ( \hirzebruchtwo )
            \mid
            [ \cEbar '] = [ \cEbar ] \in \kgr { \hirzebruchtwo } ^{ 4 }
        \right\}
    \end{align}    
\end{question}

%
%
\section{Braid group acts transitively on the set of full exceptional collections}
\label{sc:Braid group acts transitively on the set of full exceptional collections}

This section is devoted to the proof of the following theorem.
\begin{theorem}\label{th:transitivity}
    The action
\(
    G _{ 4 } \curvearrowright \ec _{ 4 } ( \hirzebruchtwo )
\)
is transitive. Namely, \pref{cj:Bondal Polishchuk} holds true for \( \hirzebruchtwo \).
\end{theorem}

The proof is divided into 3 steps. Let
\(
    \cEbar \in \ec _{ 4 } ( \hirzebruchtwo )
\)
be the given exceptional collection of length \( 4 \).
\begin{step}\label{st:reduction to numerically standard collection}
By \pref{cr:transitivity at the numerical level}, there exists
\(
    \sigma \in G _{ 4 }
\)
such that
\begin{align}\label{eq:numerically standard up to twist}
    [ \sigma ( \cEbar ) ] = [ \cEstd ] \in \numfec ( \hirzebruchtwo ),
\end{align}
where
\(
    \cEstd
\)
is the standard full exceptional collection defined in \eqref{eq:standard collection}.
Recall that there might be a difference between \( \sigma ( \cEbar ) \) and \( \cEstd \) which is invisible on the numerical level. The rest of the proof is devoted to killing this (possible) difference.
\end{step}

\begin{step}\label{st:from numerically standard to standard}
Next, we show the following theorem.
\begin{theorem}\label{th:from numerically standard to standard}
For any
\(
    \cEbar \in \ec _{ 4 } ( \hirzebruchtwo )
\)
satisfying
\begin{align}\label{eq:numerically standard}
    [ \cEbar ] = [ \cEstd ] \in \numfec ( \hirzebruchtwo ),
\end{align}
there exists
\(
    b \in \btriv
\)
such that
\(
    b ( \cEbar ) \simeq \cEstd
\).
\end{theorem}

\begin{proof}[Proof of \pref{th:from numerically standard to standard}]
Write
\(
    \cEbar = \left( \cE _{ 1 }, \cE _{ 2 }, \cE _{ 3 }, \cE _{ 4 } \right)
\).
We may and will assume that
\(
    i _{ 0 } = 0
\)
for all of the objects in the collection.
By \pref{cr:exceptional objects in the same numerical class} \pref{it:transitivity in the same numerical class}, there exists
\(
    b \in \btriv
\)
such that
\(
    b ( \cE _{ 1 } ) \simeq \cO _{ \hirzebruchtwo }
\). Hence by replacing \( \cEbar \) with \( b ( \cEbar ) \), we may and will assume that
\(
    \cE _{ 1 } = \cO _{ \hirzebruchtwo }
\).

Next, by \pref{cr:decrease ell while preserving the vector bundle}, there is
\(
    b \in B
\)
such that
\(
    b ( \cO _{ \hirzebruchtwo } )
\)
and \( b ( \cE _{ 2 } ) \) are both vector bundles. Recall that \( b \) is like either \eqref{eq:even} or \eqref{eq:odd}. If \( b \) is like \eqref{eq:even}, then it follows that both
    \(
        b _{ 0 } ( \cO _{ \hirzebruchtwo } )
    \)
    and
    \(
        b _{ 0 } ( \cE _{ 2 } )
    \)
    are line bundles for some \( b _{ 0 } \in \btriv \). Since
    \(
        [ b _{ 0 } ( \cE _{ 2 } ) ] = [ \cE _{ 2 } ] = [ \cO _{ \hirzebruchtwo } ( f ) ]
    \)
    and
    \(
        [ b _{ 0 } ( \cO _{ \hirzebruchtwo } ) ]
        =
        [ \cO _{ \hirzebruchtwo } ]
    \),
    it follows from \pref{lm:Lemma 3.5 of [OU]} that
    \(
        b _{ 0 } ( \cE _{ 2 } ) \simeq \cO _{ \hirzebruchtwo } ( f )
    \)
    and
    \( b _{ 0 } ( \cO _{ \hirzebruchtwo } ) \simeq \cO _{ \hirzebruchtwo } \).

    If \( b \) is like \eqref{eq:odd} for $a_0=-1$, then both
    \(
        T _{ - 1 } b _{ 0 } ( \cO _{ \hirzebruchtwo } )
    \)
    and
    \(
        T _{ -1 } b _{ 0 } ( \cE _{ 2 } )
    \)
    are line bundles for some \( b _{ 0 } \in \btriv \). Then it follows from the following computations and \pref{lm:Lemma 3.5 of [OU]} that
    \(
        T _{ - 1 } b _{ 0 } ( \cO _{ \hirzebruchtwo }, \cE _{ 2 } )
        =
        ( \cO _{ \hirzebruchtwo}, \cO _{ \hirzebruchtwo } ( C + f ) )
    \).
    \begin{align}
        [ T _{ - 1 } b _{ 0 } ( \cO _{ \hirzebruchtwo } ) ]
        =
        [ T _{ - 1 } ( \cO _{ \hirzebruchtwo } ) ]
        =
        [ \cO _{ \hirzebruchtwo } ],\\
        [ T _{ -1 } b _{ 0 } ( \cE _{ 2 } ) ]
        =
        [ T _{ -1 } ( \cE _{ 2 } ) ]
        =
        [ T _{ -1 } ( \cO _{ \hirzebruchtwo } ( f ) ) ]
        =
        [ \cO _{ \hirzebruchtwo } ( C + f ) ],
    \end{align}
    This immediately implies
    \(
        b _{ 0 } ( \cO _{ \hirzebruchtwo }, \cE _{ 2 } )
        =
        ( \cO _{ \hirzebruchtwo}, \cO _{ \hirzebruchtwo } ( f ) )
    \). Hence we may and will assume \( \cE _{ i } = \cEstd _{i } \) for \( i = 1, 2 \).

At this point, in fact, we are done. To see this, note that both
\(
    ( \cE _{ 3 }, \cE _{ 4 } )
\)
and
\(
    ( \cEstd _{ 3 }, \cEstd _{ 4 } )
\)
are exceptional pairs of the triangulated subcategory
\(
    {} ^{ \perp }\langle \cEstd _{ 1 }, \cEstd _{ 2 } \rangle
    \subset
    \derived ( \hirzebruchtwo )
\),
which is equivalent to \( \derived ( \bP ^{ 1 } ) \),
satisfying
\(
    [ \cE _{ i } ]
    =
    [ \cEstd _{ i } ]
    \in
    K _{ 0 } \left( {} ^{ \perp }\langle \cEstd _{ 1 }, \cEstd _{ 2 } \rangle \right)
    \hookrightarrow
    \kgr{ \derived ( \hirzebruchtwo ) }
\)
for
\( i = 3, 4 \).
It is well known that any exceptional pair of \( \derived ( \bP ^{ 1 } ) \) is (up to shifts) of the form
\(
    \left( \cO _{ \bP ^{ 1 } } ( a ), \cO _{ \bP ^{ 1 } } ( a + 1 ) \right)
\),
hence is uniquely determined (up to shifts) by the class in the Grothendieck group.
This immediately implies that
\(
    \cE _{ i } \simeq \cEstd _{ i }
\)
for
\(
    i = 3, 4
\), hence the conclusion.
\end{proof}
\end{step}

\begin{step}\label{st:replace b with mutations}
In the previous step, we killed the possible difference between \( \cEbar \) and \( \cEstd \) by spherical twists; more precisely, we found \( b \in \btriv \) such that
\begin{align}\label{eq:cEbar = b ( cEstd )}
    \cEbar = b ( \cEstd )
\end{align}
(here we put \(b\) on the right hand side intentionally).
Recall that we wanted to kill the difference by a sequence of mutations and shifts, rather than spherical twists. In this last step, we confirm that
\(
    b \in B
\)
in \eqref{eq:cEbar = b ( cEstd )} can be replaced by a sequence of mutations.
We begin with a lemma.

\begin{lemma}\label{lm:T0(OX) as mutations}
The following isomorphisms hold.
\begin{align}
    R _{ \cO _{ \hirzebruchtwo } } ( \cO _{ \hirzebruchtwo } ( - C ) )
    \simeq
    T _{0} \cO _{ \hirzebruchtwo }
    \simeq
    L _{ \cO _{ \hirzebruchtwo } ( f ) } \cO _{ \hirzebruchtwo } ( C + 2 f ).
\end{align}
\end{lemma}
\begin{proof}
By a direct computation, one can check that
\(
    T _{ 0 } \cO _{ \hirzebruchtwo }
\)
is the cone of the (essentially) unique non-trivial morphism
\begin{align}
    \cO _{ C } [ - 2 ]
    \to
    \cO _{ \hirzebruchtwo }.
\end{align}
The assertion immediately follows from this observation.
\end{proof}

\begin{theorem}
For any
\(
    b \in B
\),
there exists
\(
    \sigma \in \Br _{ 4 }
\)
such that
\(
    b ( \cEstd ) = \sigma ( \cEstd )
\).
\end{theorem}

\begin{proof}
By \pref{th:generator of the group B}, \(b\) is a product of copies of \(T _{ - 1 }, T _{ 0 }\) and their quasi-inverses. On the other hand, since autoequivalences and the action of the braid group commutes by \pref{lm:actions of Br and Auteq(f) commute}, it is enough to show the assertion only for the two cases
\(
    b = T _{ - 1 }, T _{ 0 }
\).
We treat each case separately.

For
\(
    T _{ 0 }
\),
we have
\begin{align}
    T _{0} ( \cEstd )
    =
    \left(
        T _{ 0 } \cO _{ \hirzebruchtwo },
        \cO _{ \hirzebruchtwo } ( f ),
        \cO _{ \hirzebruchtwo } ( C + 2 f )
        \otimes
        T _{ 0 } \cO _{ \hirzebruchtwo },
        \cO _{ \hirzebruchtwo } ( C + 3 f )
    \right).
\end{align}
By \pref{lm:T0(OX) as mutations}, one can easily verify that
\begin{align}
    \left(
        \sigma _{ 1 } ^{ - 1 } \circ \sigma _{ 2 } \circ \sigma _{ 1 }
    \right)
    ( \cEstd )
    =
    T _{0} ( \cEstd ).
\end{align}
For
\(
    T _{ - 1 }
\),
we have
\begin{align}
    T _{ - 1 } ( \cEstd )
    =
    \left(
        \cO _{ \hirzebruchtwo },
        \cO _{ \hirzebruchtwo } ( C + f ),
        \cO _{ \hirzebruchtwo } ( C + 2 f ),
        \cO _{ \hirzebruchtwo } ( 2 C + 3 f )
    \right).
\end{align}
By a direct computation, one can verify the following assertion.
\begin{align}
    \left(
            \sigma _{ 3 } \circ \sigma _{ 2 } \circ \sigma _{ 3 } ^{ - 1 }
    \right)
    ( \cEstd )
    =
    \left(
            \left( \sigma _{ 3 } \circ \sigma _{ 2 } \circ \sigma _{ 1 } \right)\circ \sigma _{ 1 } ^{ - 1 } \circ \sigma _{ 3 } ^{ - 1 }
    \right)
    ( \cEstd )
    =
    T _{ - 1 } ( \cEstd ).
\end{align}
\end{proof}
\end{step}

Below is an important consequence of \pref{th:transitivity}.
\begin{corollary}
\(
    \ec _{ 4 } ( \hirzebruchtwo )
    =
    \fec ( \hirzebruchtwo )
\).
\end{corollary}
\begin{proof}
Since
\(
    \cEstd
\)
is full and the fullness is preserved under the action of the group
\(
    G _{ 4 }
\),
this immediately follows from \pref{th:transitivity}.
\end{proof}

%
%

\bibliographystyle{alpha}
\bibliography{mainbibs}
\end{document}